\documentclass[10pt,electronic,reqno]{amsart}  

\usepackage[applemac]{inputenc}
\usepackage[T1]{fontenc}
\usepackage[small,it]{caption}
\usepackage[english,french]{babel}
\usepackage{url}
\usepackage[hmargin=3.3cm,vmargin=3.2cm]{geometry} 
\usepackage{amsmath}
\usepackage{amsthm}                      
\usepackage{amssymb}   

\usepackage{tikz}
\usepackage{tikz-cd}
\usepackage{mathrsfs}
\usepackage{enumerate}

\tikzset{
	graphnode/.style = {align=center, inner sep=0pt, scale=0.2, text centered,
		font=\sffamily},
	vi/.style = {align=center, inner sep=0.2pt, scale=0.15, font=\sffamily\bfseries, rectangle, draw=black,
		fill=black, text width=1em},
	vi2/.style = {align=center, inner sep=0pt, xscale=0.15,yscale=0.01,  font=\sffamily\bfseries, rectangle, draw=white,
		fill=white, text width=1em},
	ve/.style = {graphnode, circle, draw=black,fill=white, 
		text width=1em, thick},
	Subseteq/.style={
		draw=none,
		every to/.append style={
			edge node={node [sloped, allow upside down, auto=false]{$\subseteq$}}}
	},
	Isom/.style={
		draw=none,
		every to/.append style={
			edge node={node [sloped, allow upside down, auto=false]{$\cong$}}}
	},
}

\theoremstyle{plain}                                                           
\newtheorem{thm}{Theorem}[section]
\newtheorem{lem}[thm]{Lemma}
\newtheorem{prop}[thm]{Proposition}
\newtheorem{cor}[thm]{Corollary}

\theoremstyle{definition}

\newtheorem{remarkx}[thm]{Remark}
\newtheorem{defn}[thm]{Definition}

\newenvironment{rem}
{\pushQED{\qed}\renewcommand{\qedsymbol}{\lower-0.3ex\hbox{$\triangleleft$}}\remarkx}
{\popQED\endremarkx}

\DeclareMathOperator{\GL}{GL}
\DeclareMathOperator{\id}{id}
\DeclareMathOperator{\Sym}{Sym}
\DeclareMathOperator{\Ind}{Ind}
\DeclareMathOperator{\End}{End}   
\DeclareMathOperator{\Hom}{Hom}                  
\DeclareMathOperator{\Sp}{Sp}
\DeclareMathOperator{\ob}{ob}
\DeclareMathOperator{\Res}{Res}
\newcommand{\B}{\mathcal B}
\newcommand{\z}{\mathfrak z}
\newcommand{\gr}{\mathfrak{gr}}
\newcommand{\ct}{\mathsf{ct}}
\newcommand{\rt}{\mathsf{rt}}
\newcommand{\field}[1]{\ensuremath{\mathbf{#1}}}
\newcommand{\Q}{\ensuremath{\field{Q}}}        

\newcommand{\sym}{\ensuremath{\mathfrak{S}}}
\newcommand{\Z}{\ensuremath{\field{Z}}} 
\newcommand{\N}{\ensuremath{\field{N}}}

\newcommand{\MM}{\overline{{M}}}
\newcommand{\IH}{{I\! H}}
\newcommand{\CH}{\mathrm{CH}}
\newcommand{\V}{\mathbb V}

\newcommand{\RH}{{R\hspace{-0.5pt} H}}\newcommand{\FZ}{\mathsf{FZ}}
\renewcommand{\B}{\mathfrak {Br}}
\renewcommand{\vee}{\ast}
\newcommand{\VV}{\mathbf V}
\renewcommand{\O}{\boldsymbol{\mathcal O}}
\renewcommand{\u}{\mathfrak u}
\newcommand{\Corr}{\mathsf{Corr}}
\newcommand{\Mot}{\mathsf{Mot}}
\renewcommand{\L}{\mathbb L}

\newcommand{\llambda}{{\langle\lambda\rangle}}
\usepackage{euscript}
\renewcommand{\mathcal}{\EuScript}
\title{Tautological classes with twisted coefficients}

\author{Dan Petersen}
\email{dan.petersen@math.su.se}
\author{Mehdi Tavakol}
\email{mehdi.tavakol@unimelb.edu.au}
\author{Qizheng Yin}
\email{qizheng@math.pku.edu.cn}
                      

\begin{document} 
 \maketitle   

\selectlanguage{english}
\begin{abstract}
	Let $M_g$ be the moduli space of smooth genus $g$ curves. We define a notion of Chow groups of $M_g$ with coefficients in a representation of $\Sp(2g)$, and we define a subgroup of tautological classes in these Chow groups with twisted coefficients. Studying the tautological groups of $M_g$ with twisted coefficients is equivalent to studying the tautological rings of all fibered powers $C_g^n$ of the universal curve $C_g \to M_g$ simultaneously. By taking the direct sum over all irreducible representations of the symplectic group in fixed genus, one obtains the structure of a twisted commutative algebra on the tautological classes. We obtain some structural results for this twisted commutative algebra, and we are able to calculate it explicitly when $g \leq 4$. Thus we completely determine the tautological rings of all fibered powers of the universal curve over $M_g$ in these genera. We also give some applications to the Faber conjecture. 
\end{abstract}
	%

\section{Introduction}

Say that $g \geq 2$, and let $C_g^n$ be the moduli space of smooth genus $g$ curves with $n$ ordered not necessarily distinct marked points. Equivalently, $C_g^n$ is the $n$-fold fibered power of the universal curve over $M_g$ with itself. Suppose that we want to study the {cohomology} of the spaces $C_g^n$. A natural approach is to apply the Leray--Serre spectral sequence for the fibration $f \colon C_g^n \to M_g$ that forgets the $n$ markings. Since $f$ is smooth and proper, the spectral sequence degenerates by Deligne's decomposition theorem \cite{delignedegeneration}, and
$$ H^k(C_g^n,\Q) \cong \bigoplus_{p+q=k}H^p(M_g,R^q f_\ast\Q).$$
To each dominant weight $\lambda$ of $\Sp(2g)$ there is associated a local system $\V_\llambda$ on $A_g$, the moduli space of principally polarized abelian varieties of dimension $g$.
The sheaves $R^q f_\ast\Q$ decompose into direct sums of local systems $\V_\llambda$, where we use the notation $\V_\llambda$ also for their pullback along the Torelli map. The local systems $\V_\llambda$ occuring as summands of $Rf_\ast\Q$ are precisely those with $\vert \lambda \vert \leq n$.

It follows that the two collections of cohomology groups 
$$ \text{$H^\bullet(C_g^n,\Q)$ for $n \leq N$} \qquad \qquad \text{and} \qquad \qquad \text{$H^\bullet(M_g,\V_\llambda)$ for $\vert \lambda \vert \leq N$ } $$
contain more or less the same information. However, this information is ``packaged'' in a much more efficient way in the local systems. The cohomology groups of $C_g^n$ are generally very large, but when expressed in terms of local systems we see that most of the cohomology just encodes how the complex $Rf_\ast\Q$ decomposes into summands --- that is, it encodes the K\"unneth formula for the $n$-fold self-product of a genus $g$ curve, and some representation theory of $\Sp(2g)$. By studying the local systems we may focus our attention on the ``interesting'' part of the cohomology in a systematic way. 

Our first goal of this paper is to do the same thing for the \emph{tautological rings} of $C_g^n$. We remind the reader that the tautological ring $R^\bullet(C_g^n)$ is the subalgebra of $\CH^\bullet(C_g^n)$ generated by the classes of the diagonal loci $\Delta_{ij}$ where two markings coincide, the classes $\psi_1,\ldots,\psi_n$ which are the Chern classes of the $n$ cotangent line bundles at the marked points, and the Morita--Mumford--Miller classes $\kappa_d$. The image of $R^\bullet(C_g^n)$ in cohomology under the cycle class map is denoted $\RH^\bullet(C_g^n)$.

We will be able to define tautological cohomology groups $\RH^\bullet(M_g,\V_\llambda) \subseteq H^\bullet(M_g,\V_\llambda)$, with the property that the collections of tautological groups
$$ \text{$\RH^\bullet(C_g^n)$ for $n \leq N$} \qquad \qquad \text{and} \qquad \qquad \text{$\RH^\bullet(M_g,\V_\llambda)$ for $\vert \lambda \vert \leq N$ } $$
bear exactly the same relation to each other as the collections of cohomology groups $H^\bullet(C_g^n,\Q)$ and $H^\bullet(M_g,\V_\llambda)$. Thus we are able to decompose the tautological groups of $C_g^n$ into pieces indexed by local systems; the tautological groups of the local systems package all the information about the tautological groups of $C_g^n$ in a much more efficient way, and working with twisted coefficients allows us to ``zoom in'' on particularly interesting parts of the tautological groups. Moreover, the groups $\RH^\bullet(M_g,\V_\llambda)$ turn out to be more computable than the groups $\RH^\bullet(C_g^n)$.

In fact, we will actually not only do this on the level of cohomology groups, but for Chow groups. (The results are new already on the level of cohomology, though.) For this we should not work with local systems on $M_g$, but with relative Chow motives over the base $M_g$. Instead of decomposing the complex $Rf_\ast\Q$ into local systems $\V_\llambda$, we will decompose the Chow motive $h(C_g^n/M_g)$ into Chow motives $\VV_\llambda$ which are motivic lifts of the local systems $\V_\llambda$. Once the correct framework is in place, working with motives rather than local systems provides no extra difficulties.

The utility of working with the local systems is illustrated by our Theorem \ref{lowgenustheorem}, in which we completely determine all tautological groups with twisted coefficients when $g=2,3,4$. It is an easy matter to compute from Theorem \ref{lowgenustheorem} the ranks of all the groups $R^k(C_g^n)$ when $g \leq 4$, the decompositions of these tautological groups into $\sym_n$-representations, and the socle pairing. Thus a lot of useful information about the tautological rings is encoded in a few lines of information about the local systems.

Since the tautological rings are defined in terms of explicit generators, understanding the tautological rings is equivalent to finding the complete list of relations between these generators. A conjectural complete description of the tautological rings was formulated by Faber \cite{faberconjectures}. Namely, a theorem of Looijenga \cite{looijengatautological} asserts that $R^{g-2+n}(C_g^n) \cong \Q$, and that the tautological ring vanishes above this degree. Thus any two monomials of degree $g-2+n$ in the generators of the tautological ring are proportial to each other, and the proof of the $\lambda_g\lambda_{g-1}$-conjecture \cite{gp98,givental2} gives explicit proportionalities. (In fact, both Looijenga's theorem and the proportionalities were part of Faber's original conjecture.) What Faber then conjectured was that any possible relation which is consistent with the pairing into the top degree is a true relation; that is, the ring $R^\bullet(C_g^n)$ should satisfy Poincar\'e duality. The general belief is now that this conjecture should fail. One reason is that the original conjecture was later extended to a ``trinity'' of conjectures for the spaces $M_{g,n}^\rt$, $M_{g,n}^\ct$ and $\MM_{g,n}$ \cite{pandharipandequestions,faberjapan}, and the conjectures for $\MM_{2,n}$ and $M_{2,n}^{\ct}$ are known to fail when $n \geq 20$ and $n \geq 8$, respectively \cite{petersentommasi,m28ct}. The Faber conjecture for the spaces $M_{g,n}^\rt$ is equivalent to the Faber conjecture for $C_g^n$ \cite{universalcurverationaltails}, and is still open. It has more recently been conjectured that Pixton's extension of the FZ relations (see Section \ref{FZsection}) give rise to all relations between tautological classes, and this conjecture is known to contradict the Faber conjecture \cite{pixtonthesis}. 

An interesting aspect of our work is that even though the decomposition of the tautological groups $R^\bullet(C_g^n)$ into pieces indexed by representations of $\Sp(2g)$ is not compatible with the ring structure, the multiplication into the top degree behaves very well: the matrix describing the top degree pairing is block diagonal with respect to our decomposition of the tautological groups. This has the consequence that the Faber conjecture can be fruitfully studied from the perspective of the motives $\VV_\llambda$ --- Poincar\'e duality can be checked for each $\VV_\llambda$ separately. Using this we show that the Faber conjecture is true for the moduli space $C_g^n$ (hence also the space $M_{g,n}^{\rt}$) when $g \leq 4$ and $n$ is arbitrary, and we make some progress in trying to understand likely failures of the Faber conjectures in higher genera.

A completely different perspective on our results is provided by work of Kawazumi--Morita and Hain. For a fixed genus $g \geq 2$, one can define a structure of commutative ring on the direct sum
$$ \mathsf T_g = \bigoplus_\lambda H^\bullet(M_g,\V_\llambda) \otimes \V_\llambda^\ast,$$
where the direct sum is taken over all dominant weights $\lambda$ of $\Sp(2g)$. Let $\mathsf A_g = \bigwedge \V_{\langle 1,1,1\rangle}^\ast /(\V_{\langle 2,2\rangle}^\ast)$ denote the exterior algebra on the representation $\V_{\langle 1,1,1\rangle}^\ast$, modulo the ideal generated by the subrepresentation $\V_{\langle 2,2\rangle}^\ast \subset \wedge^2 \V_{\langle 1,1,1\rangle}^\ast$. If $\V_{\langle 1,1,1\rangle}^\ast$ is placed in degree $1$, then one can define a natural $\Sp(2g)$-equivariant homomorphism of graded commutative rings $\varphi \colon \mathsf A_g \to \mathsf T_g$. In particular, we get a morphism between the subalgebras of symplectic invariants, $$\varphi^{\Sp(2g)} \colon \mathsf A_g^{\Sp(2g)} \to \mathsf T_g^{\Sp(2g)} = H^\bullet(M_g,\Q).$$ According to a theorem of Kawazumi--Morita \cite{kawazumimorita} the image of $\varphi^{\Sp(2g)}$ coincides with the tautological cohomology ring of $M_g$. For this reason it is natural, following Hain \cite{hainnormalfunctions}, to define the image of $\varphi$ as the ``tautological subalgebra'' $\mathsf R_g \subset \mathsf T_g$. By considering the individual summands of $\mathsf R_g$ one obtains an a priori completely different definition of tautological subgroup of $H^\bullet(M_g,\V_\llambda)$. We prove in Theorem \ref{h-tautological} that the two definitions coincide, which in particular re-proves the theorem of Kawazumi--Morita. 
Our low genus results can be seen as calculations of the ring $\mathsf R_g$ for $g \leq 4$; these are the first nontrivial cases where this ring is completely known. A consequence of our general results is that the morphism $\varphi \colon \mathsf A_g \to \mathsf T_g$ can be lifted to take values in Chow groups rather than cohomology groups, answering a question of Hain. A final remark is that Morita has conjectured \cite{moritacohomologicalstructure} that the morphism $\varphi^{\Sp(2g)}$ is injective, so that $\mathsf A_g^{\Sp(2g)}$ is \emph{isomorphic} to the tautological ring of $M_g$ for any $g$. By extension it is natural to ask also whether $\varphi$ is injective. Our results in genus four show that this is not the case, however: $\mathsf A_4 \to \mathsf R_4$ is not an isomorphism.

Let us now state in some more detail what we do in this paper:
\begin{enumerate}[(i)]
	\item For any partition $\lambda_1 \geq \lambda_2 \geq \ldots \geq \lambda_g \geq 0$, we construct a relative Chow motive $\VV_\llambda$ over the moduli space $M_g$, which is a motivic version of the local system over $M_g$ associated to a representation of $\Sp(2g)$ of highest weight $\lambda$.
	\item For any $n\geq 0$, we let $h(C_g^n/M_g)$ be the relative Chow motive over $M_g$ given by the $n$-fold fibered power of the universal curve. We prove that there exists a direct sum decomposition
	\begin{equation}\label{intro-decomposition1}
	h(C_g^n/M_g) \cong \bigoplus_i \VV_{\langle \lambda_i \rangle} \otimes \mathbb L^{m_i},
	\end{equation} 
	where $\mathbb L$ denotes the Lefschetz motive, and in particular we get upon taking Chow groups
	\begin{equation}\label{intro-decomposition2}
	\CH^k(C_g^n) \cong \bigoplus_i \CH^{k-m_i}(M_g,\VV_{\langle \lambda_i \rangle}).
	\end{equation} 
	\item We construct an algebra of correspondences defined by tautological classes, which acts on the motives $h(C_g^n/M_g)$, and hence on the Chow groups $\CH^k(C_g^n)$. Using this algebra we obtain a \emph{canonical} choice of decomposition \eqref{intro-decomposition1}, and a method for computing the projection of any class in $\CH^k(C_g^n)$ into any particular summand on the right hand side of \eqref{intro-decomposition2}.
	\item We define subgroups $R^k(M_g,\VV_\llambda) \subseteq \CH^k(M_g,\VV_\llambda)$ with the property that for any decomposition of $ h(C_g^n/M_g)$ as in Eq. \eqref{intro-decomposition1}, we have
	\begin{equation*}\label{intro-decomposition3}
	R^k(C_g^n) \cong \bigoplus_i R^{k-m_i}(M_g,\VV_{\langle \lambda_i \rangle}).
	\end{equation*}
	We call the groups $R^k(M_g,\VV_\llambda)$ the \emph{tautological groups of $M_g$ with twisted coefficients.} Informally, all information about the tautological rings $R^\bullet(C_g^n)$ is contained in the groups $R^k(M_g,\VV_\llambda)$.
	\item The motives $\VV_\llambda$ come with a duality pairing $\VV_\llambda \otimes \VV_\llambda \to \mathbb L^{\vert \lambda\vert}$, which is the motivic avatar of the fact that all representations of the symplectic group are self-dual. We prove that the socle pairing 
	\begin{equation} \label{intro-socle1}
	R^k(C_g^n) \otimes R^{g-2+n-k}(C_g^n) \to R^{g-2+n}(C_g^n) \cong \Q
	\end{equation}
	is a direct sum of pairings of the form
	\begin{equation}\label{intro-socle2}
	 R^k(M_g,\VV_\llambda) \otimes R^{g-2+\vert\lambda\vert-k}(M_g,\VV_\llambda) \to R^{g-2+\vert\lambda\vert}(M_g,\mathbb L^{\vert \lambda\vert}) = R^{g-2}(M_g)  \cong \Q\end{equation}
	for $\vert \lambda \vert \leq n$. In particular, $R^\bullet(C_g^n)$ is a Gorenstein algebra --- that is, \eqref{intro-socle1} is a perfect pairing for all $k$ --- if and only if \eqref{intro-socle2} is a perfect pairing for all $k$ and all $\vert \lambda \vert \leq n$.
	\item If we fix a genus and consider the \emph{direct sum} over all partitions,
	$$ \bigoplus_\lambda R^\bullet(M_g,\VV_\llambda) \otimes \sigma_{\lambda^T},$$
	where $\sigma_{\lambda^T}$ denotes the representation of the symmetric group corresponding to the partition conjugate to $\lambda$, we obtain the structure of a \emph{twisted commutative algebra}. The $\lambda=0$ component of this twisted commutative algebra is just the tautological ring of $M_g$. We prove using the FZ relations that this twisted commutative algebra is finitely generated  with an explicit bound for the degrees of the generators, which for $\lambda=0$ specializes to the theorem of Ionel--Morita \cite{ionel,moritagenerators}.
	\item For $g=2,3,4$ we completely determine the groups $R^k(M_g,\VV_\llambda)$ for all $k$ and $\lambda$. A key input is that the twisted commutative algebra described in point (vi) above will in these low genera have only $0$, $2$ and $3$ generators, respectively, by our generalization of the theorem of Ionel--Morita. As a consequence we can compute the ranks $R^k(C_g^n)$ for all $k$ and $n$ in these genera, and how these tautological groups decompose into irreducible representations of $\sym_n$. It also follows that $R^\bullet(C_g^n)$ is always a Gorenstein algebra in these genera.
\end{enumerate}

The algebra of projectors described in point (iii) above seems like a particularly powerful tool for ``zooming in'' on a specific part of the tautological ring. As explained in Section \ref{genus5} we expect that the Faber conjecture fails when $g=5$ and $n=8$ for the motives $\VV_{\langle 2,2,2,2\rangle}$ and $\VV_{\langle 3,2,2,1\rangle}$. Using our algebra of projectors we can project specific tautological classes onto these summands, which gives explicit classes which pair to zero with everything in complementary degree but which are conjecturally nonzero.

It would be interesting to try to extend our results from $M_g$ to the Deligne--Mumford compactification $\MM_g$. On the level of cohomology this would correspond to studying the forgetful maps $f \colon \MM_{g,n} \to \MM_g$ rather than $C_g^n \to M_g$. Then $f$ is no longer smooth, but still proper, so by the decomposition theorem \cite{bbd} the complex $Rf_\ast\Q$ is a direct sum of perverse sheaves. In fact, each of these perverse sheaves will be the pushforward along some gluing map
$$ \Big(\!\prod_{v \in \mathrm{Vert}(\Gamma)}\!\!\! \MM_{g(v),n(v)}\Big)/\mathrm{Aut}(\Gamma) \to \MM_{g,n}$$
of the intermediate extension of a product of local systems associated to representations of the smaller symplectic groups $\Sp(2g(v))$. This suggests that one should try to define a subspace of tautological classes inside the intersection cohomology groups $\IH^\bullet(\MM_g,\V_\llambda)$ for each genus $g$ and dominant weight $\lambda$, and that these tautological groups should ``govern'' all of the tautological groups $R^\bullet(\MM_{g,n})$ much in the same way as the tautological groups of $\V_\lambda$ on $M_g$ govern all the tautological groups $R^\bullet(C_g^n)$. Similarly there should be tautological classes inside $\IH^\bullet(M_g^\ct,\V_\llambda) = H^\bullet(M_g^\ct,\V_\llambda)$ which govern all the tautological rings $R^\bullet(M_{g,n}^\ct)$. The result \cite[Theorem 3.4]{m28ct} can be seen as calculating the tautological subspace of $H^\bullet(M_2^\ct,\V_\llambda)$ for all $\lambda$. Moreover, it should be possible to carry out the suggestions in this paragraph also on the level of Chow groups, using the intersection Chow motives of Corti--Hanamura \cite{cortihanamura}.

\subsection{How to read this paper}

As mentioned already, this paper is written in the language of Chow motives. Readers who would prefer not to know what a motive is should still be able to follow the arguments by translating the arguments to cohomology using the following table:
\begin{small}\begin{align*}
\text{Chow motive $h(X/S)$ of a family $f \colon X \to S$} &&& \text{Complex $Rf_\ast\Q$ in the derived category of $S$} \\
\text{Decomposition $h(X/S) \cong \bigoplus_i h^i(X/S)$}&&& \text{Decomposition $Rf_\ast\Q \cong \bigoplus_i R^i f_\ast\Q[-i]$} \\
\text{Chow group $\CH^k(S,h^i(X/S))$} &&& 
\text{Cohomology group $H^{2k-i}(S,R^if_\ast\Q)$} \\
\text{Lefschetz motive $\L$} &&& \text{\parbox{\textwidth}{Constant sheaf $\Q$, considered as a complex\\concentrated in degree $2$}}\\
\text{Chow motive $\VV_\llambda$ over $M_g$} &&& 
\text{\parbox{\textwidth}{Local system $\V_\llambda$ on $M_g$, considered as a complex \\ concentrated in degree $\vert\lambda\vert$}}
\end{align*}
\end{small}Thus the only real complication is the indexing: the $k$th Chow group of the motive $\VV_\llambda \otimes \L^i$ corresponds to the $(2k - \vert \lambda \vert - 2i)$th cohomology group of the local system $\V_\llambda$.

Sections 2--4 of this paper explain necessary preliminary material from the representation theory of the symplectic group and about Chow motives. In particular, Section 4 explains a result of Ancona which is used to lift our methods from cohomology to Chow groups. It could be a good idea for the reader to start from Section \ref{mainsection} and refer back to the previous sections only as needed. Section \ref{mainsection} provides the theoretical backbone to the article, and Section \ref{examplesection} provides some simple (hopefully instructive) example calculations. In Sections \ref{Gorsection}--\ref{lowgenuscalculations} the theory is applied and our main results are proven. The concluding Sections \ref{looijenga-section} and \ref{hain-section} explain the relationship between what we do and previous work of Looijenga, Hain, Morita, and Kawazumi.

\subsection{Conventions}

Representations of groups will be considered as left representations unless specified otherwise. However, if $V$ is a left representation, then we consider its dual $V^\vee$ as a right representation. (Recall that the dual of a left module over a noncommutative ring is a right module, and vice versa.)

Chow groups are always taken with rational coefficients. 

We occasionally consider cohomology groups as well as Chow groups. Although we write cohomology with rational coefficients throughout, it will be clear that all results could have been carried out equally well in the \'etale setting, with coefficients in $\Q_\ell$. 


\subsection{Acknowledgements} The first author would like to thank Dick Hain and Allen Knutsen for useful comments. The second author was supported by the research grant IBS-R003-S1. We are grateful to Rahul Pandharipande for his interest in this work.

\section{Chow motives}
The results of this paper will be formulated in the language of Chow motives. The first parts of this section briefly recall standard definitions for the reader's convenience and to fix conventions. For a more detailed and motivated introduction see e.g.\ \cite{classicalmotives}.

\subsection{The category of Chow motives}Let $S$ be a smooth connected scheme or Deligne--Mumford stack over a field $k$ that we assume algebraically closed for simplicity. Let $X$ and $Y$ be smooth proper schemes over $S$.\footnote{If $S$ is a Deligne--Mumford stack we do not impose the condition that $X$ and $Y$ are schemes, only that the maps to $S$ are representable in schemes.} We define a graded vector space of \emph{correspondences over $S$} as follows: if $X$ is connected and $X \to S$ is of relative dimension $d$ then $\Corr_S(X,Y) = \CH^{d+\bullet}(X \times_S Y)$; in the general case we define $\Corr_S(\coprod_\alpha X_\alpha,Y) = \prod_{\alpha} \Corr_S(X_\alpha,Y)$. We write $$ f \colon X \vdash Y$$ to denote that $f$ is a correspondence from $X$ to $Y$. The composition of $f \colon X \vdash Y$ and $g \colon Y \vdash Z$ is defined by
$$ g \circ f = (p_{13})_\ast( p_{12}^\ast(f) \cdot p_{23}^\ast(g)),$$
(note the reversed ordering!) where $p_{ij}$ denotes the projection from $X \times_S Y \times_S Z$ onto the $i$th and $j$th factor of the fibered product. One checks that composition of correspondences is associative and that the diagonal, considered as a correspondence $X \vdash X$, acts as the identity $\id_X$, so that $\Corr_S$ is a category.

We say that a correspondence $p \colon X \vdash X$ of degree $0$ is \emph{idempotent} if $p \circ p = p$. We also say that $p$ is a \emph{projector}.

We define the category $\Mot_S$ of \emph{Chow motives over $S$.} The objects of $\Mot_S$ are triples $(X,p,n)$ where $X$ is smooth and proper over $S$, $p \colon  X \vdash X$ is a projector, and $n \in \Z$. Morphisms are defined by
$$ \Mot_S((X,p,n),(Y,q,m)) = q \circ \Corr_S^{m-n}(X,Y)\circ p \subset \Corr_S(X,Y),$$
where $\Corr_S^{r}(X,Y)$ denotes the degree $r$ part of $\Corr_S(X,Y)$, and $q \circ \Corr_S^{m-n}(X,Y)\circ p$ denotes the joint image of the projectors $p$ and $q$ acting on $\Corr_S(X,Y)$ on the right and on the left, respectively. 


The \emph{Lefschetz motive over $S$} is defined as $(S,\id,-1)$ and will be denoted by $\L_S$. If $S$ is clear from context we will omit the subscript and write $\L$.

We define a tensor product on motives as follows. If $M=(X,p,n)$ and $N=(Y,q,m)$ then $M\otimes N = (X \times_S Y, p \times q,n+m)$. This makes $\Mot_S$ a symmetric monoidal category with monoidal unit $\mathbf 1 = (S,\id,0)$. The category is in fact rigid symmetric monoidal, i.e.\ every object has a dual: if $X$ is of pure dimension $d$ over $S$, then the dual of $M = (X,p,n)$ is $M^\vee = (X,p^t,d-n)$, where $p^t$ denotes the transpose correspondence. The category also has direct sums. The sum $(X,p,n) \oplus (Y,q,m)$ is easiest to define when $n=m$, in which case it is given by $(X \sqcup Y, p \oplus q, n)$. 

Let $\mathcal V_S$ be the category of smooth proper schemes over $S$. There is a contravariant functor $\mathcal V_S \to \Mot_S$ which is given on objects by $X \mapsto (X,\id,0)$ and which maps an $S$-morphism $f \colon Y \to X$ to the class of the transpose of its graph in $\CH^{\dim_S(X)}(X \times_S Y)$. For $X\to S$ smooth and proper we denote by $h(X/S)$ the corresponding Chow motive over $S$. Note that $h(X/S)^\vee \cong h(X/S)\otimes \L^{-\dim_S X}$ (Poincar\'e duality). 

\subsection{Chow groups and cohomology groups of a Chow motive}
Let $M$ be a Chow motive over $S$. We define its Chow groups by 
$\CH^k(S,M) = \Mot_S(\L_S^k,M)$. We can make this definition more explicit as follows. Note that for $X$ a smooth proper scheme over $S$ we have $\CH^\bullet(X) = \Corr_S(S,X)$. As such, the algebra $\Corr_S(X,X)$ acts on the Chow groups of $X$ on the left. Let $M = (X,p,n)$ be a Chow motive over $S$. Then its Chow groups are given by
$$ \CH^k(S,M) = p \circ \CH^{k+n}(X).$$
\begin{rem}\label{transposeremark}
	It is also true that $\CH^\bullet(X) = \Corr_S(X,S)$ (up to a degree shift), so that $\Corr_S(X,X)$ acts on the Chow groups of $X$ on the right. One could also define $\CH^{k}(S,M) = \CH^{k+n}(X) \circ p^t$, where $p^t$ denotes the transpose correspondence of $p$.  
\end{rem}
Let us suppose that $S$ is a complex algebraic variety. There is a \emph{Betti realization} functor $\mathsf{real}: \Mot_S \to D^b(S)$ into the bounded derived category of sheaves of $\Q$-vector spaces on $S$. For $\lambda \colon X \to S$ a smooth proper scheme over $S$ we have
\[ \mathsf{real}\,h(X/S) = R\lambda_\ast \Q.\]
The algebra $\Corr_S(X,X)$ acts on the complex $R\lambda_\ast\Q$, and for $p \colon X \vdash X$ idempotent we define $\mathsf{real}\,(X,p,n) = \mathrm{Im}(p_\ast \colon R\lambda_\ast\Q \to R\lambda_\ast\Q)[2n]$, where $[2n]$ denotes the suspension functor in $D^b(S)$. There is a \emph{cycle class map}
$$ \CH^k(S,M) \to \mathbb H^{2k}(S,\mathsf{real}\, M)$$
(where $\mathbb H$ denotes hypercohomology) which on motives of the form $h(X/S)$ agrees with the usual cycle class map:
$$ \CH^k(X) = \CH^k(S,h(X/S)) \to \mathbb H^{2k}(S,\mathsf{real}\, M) = \mathbb H^{2k}(S,R\lambda_\ast \Q) = H^{2k}(X,\Q).$$
Over an arbitrary field there is an analogous realization functor from $\Mot_S$ to the derived category of \'etale $\Q_\ell$-sheaves on $S$. 

\subsection{K\"unneth decomposition, the summands $h^0$ and $h^{2d}$}\label{kunnethsection}Let $\lambda \colon X \to S$ be smooth, proper and purely of relative dimension $d$. By Deligne's theorem \cite{delignedegeneration} there is an isomorphism in $D^b(S)$:
$$ R\lambda_\ast\Q \cong \bigoplus_{i=0}^{2d} R^i\lambda_\ast\Q[-i].$$
In particular, this decomposition implies that the Leray spectral sequence for $\lambda$ degenerates, and $H^k(X,\Q) \cong \bigoplus_{p+q=k} H^p(S,R^q\lambda_\ast\Q).$
It is expected that this decomposition always lifts to the category of Chow motives. Thus there should be an isomorphism 
$$ h(X/S) \cong \bigoplus_{i=0}^{2d} h^i(X/S)$$
for which $\mathsf{real}\, h^i(X/S) \cong R^i\lambda_\ast\Q[-i]$. In particular one would have
$$ \CH^k(X) \cong \bigoplus_{i=0}^{2d} \CH^k(S,h^i(X/S)).$$
The summands $h^0$ and $h^{2d}$ can easily be constructed unconditionally. Let us suppose that $X$ is connected, and let $\z \in \CH^d(X)$ be a cycle of degree $1$ on each fiber of $X \to S$, e.g.\ a section. One checks that the two correspondences $X \vdash X$ given by 
$$ \pi_0 = [\z \times X] \qquad \text{and} \qquad \pi_{2d} = [X \times \z]$$
are idempotent. If we define $h^0(X/S) = (X,\pi_0,0)$ and $h^{>0}(X/S) = (X,\id_X - \pi_0,0)$ then 
$$ h(X/S) \cong h^0(X/S) \oplus h^{>0}(X/S)$$
which on realizations gives the decomposition
$$ R\lambda_\ast\Q \cong R^0\lambda_\ast\Q \oplus \tau_{\geq 1}R\lambda_\ast\Q,$$
where $\tau$ denotes a truncation functor in the derived category $D^b(S)$. Similarly we get decompositions $h(X/S) \cong h^{<2d}(X/S) \oplus h^{2d}(X/S)$ with realization $\tau_{\leq 2d-1}R\lambda_\ast\Q \oplus R^{2d}\lambda_\ast\Q[-2d]$.

\begin{lem}\label{h2d}Let $X$ and $\z$ be as above. Then $h^0(X/S) \cong \mathbf 1$ and $h^{2d}(X/S) \cong \L^d$.	
\end{lem}

\begin{proof}
	We prove only the second isomorphism. By definition we have
	\begin{align*}
	\Mot_S(\L^d,h^{2d}(X/S)) = \pi_{2d} \circ \CH^d(X) \\
	\Mot_S(h^{2d}(X/S),\L^d) = \CH^0(X) \circ \pi_{2d}
	\end{align*}
	It is clear that $\pi_{2d} \circ \z = \z$ and $1 \circ \pi_{2d} = \pi_0 \circ 1 = 1$ (cf. Remark \ref{transposeremark}). As such the cycle $\z$ and the fundamental class $1$ define morphisms $\L^d \to h^{2d}(X/S) \to \L^d$. Moreover, their composition in $\Mot_S(\L^d,\L^d) = \CH^0(S)$ is given by $\lambda_\ast(\z) = 1$, the identity. Their composition in $\Mot_S(h^{2d}(X/S),h^{2d}(X/S)) = \pi_{2d}\circ \CH^d(X\times_S X) \circ \pi_{2d}$ is given by the correspondence $\pi_{2d}$, which is also the identity.
\end{proof}

One might want to define a motive $h^\star(X/S) = (X,\id-\pi_0-\pi_{2d})$ to get a decomposition $h(X/S) \cong h^0(X/S) \oplus h^\star(X/S) \oplus h^{2d}(X/S)$, where the Betti realization of $h^\star(X/S)$ is the complex $\tau_{\geq 1 }\tau_{\leq 2d-1} R\lambda_\ast\Q$. Unfortunately the correspondence $\id-\pi_0-\pi_{2d}$ is not in general idempotent, since $\pi_0$ and $\pi_{2d}$ are not in general orthogonal when the base scheme $S$ is nontrivial; this will be clear from the proof of the following lemma. To make the projectors orthogonal one needs to slightly modify the cycle $\z$.

\begin{lem}\label{curvedecomposition} Let $\lambda \colon X \to S$ be as above. Let $\z \in \CH^d(X)$ be a cycle of degree $1$ on each fiber over $S$. Define $\z' = \z - \frac 1 2 \lambda^\ast \lambda_\ast (\z^2)$. Then $\z'$ also has degree $1$ on each fiber, the projectors $\pi_0 = [\z' \times X]$, $\pi_{2d} = [X \times \z']$ are orthogonal, and there is a decomposition
	$$ h(X/S) \cong h^0(X/S) \oplus h^{\star}(X/S) \oplus h^{2d}(X/S)$$
	with $h^0(X/S) = (X,\pi_0,0)$, $h^\star(X/S) = (X,\id_X - \pi_0 -\pi_{2d},0)$ and $h^{2d}(X/S) = (X,\pi_{2d},0)$. 
\end{lem}

\begin{proof}
	We check that $\pi_0$ and $\pi_{2d}$ are orthogonal. We have 
	$$ \pi_0 \circ \pi_{2d} = (p_{13})_\ast (p^\ast_{12}(\pi_{2d}) \cdot p_{23}^\ast(\pi_0)) = (p_{13})_\ast(p_1^\ast (\z') \cdot p_3^\ast(\z')) = 0$$
	and 
	\begin{align*}
	\pi_{2d} \circ \pi_{0} &= (p_{13})_\ast (p^\ast_{12}(\pi_{0}) \cdot p_{23}^\ast(\pi_{2d})) = (p_{13})_\ast(p_2^\ast (\z')^2). 
	\end{align*}
	From the cartesian diagram
	\[\begin{tikzcd}
	X \times_S X \times_S X \arrow[r,"p_{13}"] \arrow[d,"p_2"]& X \times_S X \arrow[d,"\lambda \times \lambda"]\\
	X \arrow[r,"\lambda"]& S
	\end{tikzcd}\]
	we get $(p_{13})_\ast (p_2^\ast (\z')^2)) = (\lambda\times\lambda)^\ast (\lambda_\ast (\z')^2).$ But now
	\[ \lambda_\ast (\z')^2 = \lambda_\ast (\z^2 - \z \cdot \lambda^\ast \lambda_\ast \z^2 + \frac 1 4 (\lambda^\ast \lambda_\ast \z^2)^2 ) = \lambda_\ast(\z^2)- \lambda_\ast(\z^2) + 0 = 0. \qedhere \]
\end{proof}

\begin{rem}The decomposition $ h(X/S) \cong h^0(X/S) \oplus h^{\star}(X/S) \oplus h^{2d}(X/S)$ is not unique: it depends very much on the choice of a cycle $\z$. Nevertheless each of the summands on the right hand side is determined up to a canonical isomorphism, independently of $\z$. Indeed after Lemma \ref{h2d} we only need to verify this for $h^1$, and a small verification shows that the diagonal in $X \times_S X$ composed with the respective projectors gives the required isomorphism. \end{rem}

\subsection{K\"unneth decomposition for abelian schemes}

The decomposition of Lemma \ref{curvedecomposition} provides a K\"unneth decomposition for families of curves (or surfaces with no odd cohomology). The other case we will use in this paper is the existence of a motivic K\"unneth decomposition for abelian schemes:

\begin{thm}[Shermenev, Deninger--Murre, K\"unnemann]
	Let $A \to S$ be an abelian scheme of relative dimension $g$. There exists a K\"unneth decomposition 
	$$ h(A/S) \cong \bigoplus_{i=0}^{2g}h^i(A/S)$$
	in which we have $h^i(A/S)\cong \Sym^i h^1(A/S)$ for all $i$, and $\Sym^i h^1(A/S)=0$ for $i>2g$.
\end{thm}

Shermenev's proof of this fact starts by choosing a curve $C$ such that its jacobian maps surjectively onto $A$. This makes the decomposition highly noncanonical, and limits the construction to the absolute case  --- there is no reason for a general abelian scheme to be a quotient of the jacobian of a family of curves. By contrast, the constructions of Deninger--Murre \cite{deningermurre} and K\"unnemann \cite{kunnemann} use Fourier theory and are canonical and functorial. 

The Deninger--Murre decomposition can be described using \emph{Manin's identity principle} (see e.g.\ \cite[2.3]{classicalmotives}), which says that although a Chow motive $M$ over $S$ is not determined by its Chow groups, it is determined by its Chow groups after any base change; more precisely, the functor that takes a smooth proper morphism $f \colon T\to S$ to the Chow groups of the relative Chow motive $f^\ast M$ over $T$, determines $M$ completely. For an abelian scheme $A \to S$ we have the morphism $[N] \colon A \to A$ of multiplication by $N>1$, and the Chow groups of $A$ (and any base change of $A$) can be decomposed canonically into eigenspaces for $N$, for all $N$. The summand $h^i(A/S)$ in the Deninger--Murre decomposition corresponds to the $N^i$-eigenspace of the Chow groups of $A$.  

\begin{prop}\label{curvejacobianisomorphism}Let $C \to S$ be a family of smooth curves, and $J \to S$ the jacobian. Then there exists an isomorphism of Chow motives $h^1(C/S) \cong h^1(J/S)$, where $h^1(C/S)$ is the summand of $h(C/S)$ described in Lemma \ref{curvedecomposition} and $h^1(J/S)$ is the summand of $h(J/S)$ provided by the decomposition of Deninger--Murre.
\end{prop}

\begin{proof}[Sketch of proof.]
	After replacing $S$ with a finite \'etale Galois cover we may assume that $C \to S$ has a section. It is enough to prove the isomorphism under this assumption; since we work with $\Q$-coefficients we may then take Galois invariants to obtain the conclusion over our original base scheme. The section defines an Abel--Jacobi map $C \to J$ and puts us in the situation considered by Shermenev \cite{shermenev}. Shermenev's construction provides a motivic K\"unneth decomposition $h(J/S) = \bigoplus_i h^i(J/S)$ for which it is clear from construction that $h^1(C/S)=h^1(J/S)$. Unfortunately the resulting motivic decomposition of $h(J/S)$ is in general different from that of Deninger--Murre.

	The claim is now that even though the two direct sum decompositions of the Chow groups of $J$ are different, they give rise to the same descending filtration of the Chow groups, so that the two associated graded objects are isomorphic. Since $S$ is arbitrary it will then hold also after base change to an arbitrary smooth proper $S' \to S$, and from Manin's identity principle  it will then follow that the motives $h^i(J/S)$ obtained from Shermenev's decomposition are isomorphic to those of Deninger--Murre. The fact that the two descending filtrations of Chow groups coincide is part of a theorem of Moonen--Polishchuk \cite[Theorem 4]{moonenpolishchuk}.
\end{proof}

\section{Preliminaries from representation theory}\label{schurweylsection}

We define a \emph{partition} to be a non-increasing sequence of natural numbers which eventually reaches zero: $\lambda = (\lambda_1 \geq \lambda_2 \geq \lambda_3 \geq \ldots \geq 0 \geq 0 \geq \ldots)$. The \emph{weight} of a partition is defined as $\vert \lambda \vert = \sum_i \lambda_i$. The \emph{length} of a partition is defined as $\ell(\lambda) = \max\{ i \,:\, \lambda_i \neq 0 \}$. Partitions are often identified with \emph{Young diagrams}. Our convention for Young diagrams is that the numbers $\lambda_i$ are the lengths of the rows in the diagram. We denote the conjugate partition of $\lambda$, obtained by reflecting the Young diagram across the diagonal, by $\lambda^T$. 

\subsection{Schur--Weyl duality}

Let $V$ be a vector space over $\Q$. (Everything that follows is true more generally over any field of characteristic zero.) Partitions with $\ell(\lambda) \leq \dim(V)$ are in a natural bijection with irreducible finite dimensional representations of $\GL(V)$ via the theory of highest weight vectors. We write $V_\lambda$ for the representation of $\GL(V)$ corresponding to $\lambda$. For example, if $\lambda = (n \geq 0 \geq 0 \geq \ldots)$, then $V_\lambda = \Sym^n(V)$ and $V_{\lambda^T} = \wedge^n V$. If $\ell(\lambda)>\dim(V)$ then we define $V_\lambda$ to be zero. The representations of the symmetric group are also indexed by partitions: the partitions with $\vert \lambda \vert = n$ are in natural bijection with the representations of the symmetric group $\sym_n$, which we denote by $\sigma_\lambda$. If $\lambda = (n \geq 0 \geq 0 \geq \ldots)$ then $\sigma_\lambda$ is the trivial representation and $\sigma_{\lambda^T}$ is the sign representation.

The vector space $V^{\otimes n}$ carries commuting left and right actions by $\GL(V)$ and $\sym_n$, respectively. \emph{Schur--Weyl duality} in its most basic form is  an expression for how to decompose $V^{\otimes n}$ into irreducible representations under this action of $ \GL(V)\times\sym_n $:
\[ V^{\otimes n} = \bigoplus_{\vert \lambda \vert = n} V_\lambda \otimes \sigma_\lambda^\vee  .\]
Although $\sigma_\lambda \cong \sigma_\lambda^\vee$, we dualize to emphasize that we are considering a right action. 

Schur--Weyl duality can be formulated more abstractly in terms of mutual centralizers. Namely, $V^{\otimes n}$ admits commuting actions of $\GL(V)$ and the group algebra $\Q[\sym_n]$, and Schur--Weyl duality is equivalent to the claim that the centralizer of $\GL(V)$ in $\End_\Q(V^{\otimes n})$ equals the image of $\Q[\sym_n]$ in $\End_\Q(V^{\otimes n})$, and vice versa. A useful consequence of this more abstract viewpoint is that it produces for all $\lambda$ an explicit idempotent endomorphism of $V^{\otimes n}$ whose image is exactly the summand $V_\lambda\otimes \sigma_\lambda^\vee$. 
Namely, the group algebra $\Q[\sym_n]$ contains a family of orthogonal idempotents called \emph{Young symmetrizers}. If $c_\lambda$ denotes a Young symmetrizer corresponding to the partition $\lambda$, then the image of $c_\lambda$ is the summand $V_\lambda\otimes \sigma_\lambda^\vee$ of $V^{\otimes n}$. 

\subsection{Symplectic groups and Weyl's construction}\label{weylconstruction}
Suppose that the vector space $V$ is equipped with a symplectic form. Then partitions of length $\ell(\lambda) \leq \frac 1 2 \dim(V)$ are in a  bijection with irreducible finite dimensional representations of $\Sp(V)$, and we write $V_\llambda$ for the representation of $\Sp(V)$ corresponding to $\lambda$. Similarly we set $V_\llambda = 0$ if $\ell(\lambda)>\frac 1 2 \dim(V)$. The decomposition of $V^{\otimes n}$ into irreducible representations of $\Sp(V)\times \sym_n $ is  more complicated than that for $\GL(V)$. The first nontrivial example is the case $n=2$:
$$ V^{\otimes 2} = V_{\langle 2\rangle}  \otimes \sigma_2^\vee\oplus  V_{\langle 1,1\rangle }\otimes \sigma_{1,1}^\vee\oplus  V_{\langle 0 \rangle }\otimes\sigma_{1,1}^\vee .$$
The first two terms are exactly what one expects from Schur--Weyl duality. The third term arises because $\wedge^2 V$ is not irreducible: it contains the trvial representation spanned by the class of the symplectic form as a subrepresentation. 

The example $n=2$ generalizes to larger values of $n$ as follows. We define 
$$ V^{\langle n \rangle} \subset V^{\otimes n}$$
to be the subspace of \emph{traceless tensors}, i.e.\ the intersection of the kernels of all $\binom n 2$ maps 
$$ V^{\otimes n} \to V^{\otimes (n-2)}$$
given by contracting with the symplectic form. Alternatively, we can think of $V^{\langle n \rangle}$ as a quotient of $V^{\otimes n}$, where we divide by the images of all $\binom n 2$ maps $V^{\otimes (n-2)} \to V^{\otimes n}$ given by inserting the class of the symplectic form. The subspace $V^{\langle n \rangle}$ is clearly $\Sp(V)$-invariant, and Weyl proved that there is an isomorphism
$$ V^{\langle n \rangle} = \bigoplus_{\vert \lambda \vert = n}  V_\llambda \otimes\sigma_\lambda^\vee $$
for all $n$. Thus we know how to decompose the subspace $V^{\langle n \rangle}$ into irreducible representations of $\Sp(V)\times \sym_n $. Moreover, we may write $V^{\otimes n}$ as the direct sum of $V^{\langle n \rangle}$ and the image of all the maps $V^{\otimes (n-2)} \to V^{\otimes n}$; we may write $V^{\otimes (n-2)}$ as the direct sum of $V^{\langle n -2 \rangle}$ and the image of all maps $V^{\otimes (n-4)} \to V^{\otimes (n-2)}$, etc. This leads inductively to a decomposition of $V^{\otimes n}$ into irreducible $\Sp(V)\times \sym_n$-representations. All $V_\llambda$ with $\vert \lambda \vert \leq n$ and $\vert \lambda \vert \equiv n \pmod 2$ will occur in this decomposition.  We refer to this as \emph{Weyl's construction} of the irreducible representations of $\Sp(V)$.

\subsection{Brauer algebra}\label{nazarov1} We now wish to give a version of Schur--Weyl duality for the symplectic group in terms of mutual centralizer algebras for the action of $\Sp(V)$ on $V^{\otimes n}$. The centralizer of $\Sp(V)$ acting on $V^{\otimes n}$ is larger than $\Q[\sym_n]$. It can be described as the algebra of endomorphisms of $V^{\otimes n}$ generated by $\Q[\sym_n]$ and the maps given by compositions
$$ V^{\otimes n} \to V^{\otimes (n-2)} \to V^{\otimes n},$$
where the first map contracts two tensor factors using the symplectic form, and the second map inserts the form. Brauer \cite{brauer} introduced a diagrammatic calculus which is useful for describing endomorphisms in this centralizer algebra. 
We give here a category-theoretic treatment of the Brauer algebra. Somewhat similar presentations can be found in \cite{lehrerzhang,rubeywestbury}.

Let $n$ and $m$ be nonnegative integers. We define an \emph{$(n,m)$-Brauer diagram} to be a diagram of two rows containing $n$ and $m$ dots, respectively, and $(n+m)/2$ strands connecting these dots pairwise. The set of $(n,m)$-Brauer diagrams is empty unless $n \equiv m \pmod 2$. Here is a (4,6)-Brauer diagram:
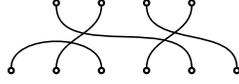
\[\begin{tikzpicture}[line width=0.6pt,scale=0.6]
\foreach \x in {1,...,6}{
	\node[ve] (b\x) at (\x,0) {};
}
\foreach \x in {2,...,5}{
	\node[ve] (u\x) at (\x,1.5) {};
}
\path (b6) edge[out=90,in=-90] (u4);
\path (b1) edge[out=90,in=90] (b3);
\path (u2) edge[out=-90,in=90] (b5);
\path (u3) edge[out=-90,in=90] (b2);
\path (u5) edge[out=-90,in=90] (b4);
\end{tikzpicture} \]
For any parameter $\delta \in \Q$, let $\B^{(\delta)}(n,m)$ be the $\Q$-vector space spanned by all $(n,m)$-Brauer diagrams. We define a composition map
$$ \B^{(\delta)}(n,m) \otimes \B^{(\delta)}(m,k) \to \B^{(\delta)}(n,k)$$
which is defined on basis elements as follows: to compose an $(n,m)$-Brauer diagram and an $(m,k)$-Brauer diagram, connect the strands on the bottom of the first diagram with those on the top of the second diagram, erase any loops that are formed in the process, and multiply the result by $\delta$ to the power of the number of erased loops. The following example illustrates a composition $\B^{(\delta)}(5,7) \otimes \B^{(\delta)}(7,3) \to \B^{(\delta)}(5,3)$: 
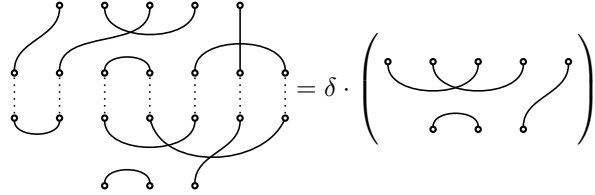
\[ \begin{tikzpicture}[baseline=(O.base),line width=0.6pt,scale=0.6]
\foreach \x in {2,...,6}{
	\node[ve] (u\x) at (\x,1.5) {};
}

\foreach \x in {1,...,7}{
	\node[ve] (b\x) at (\x,0) {};
	
}
\node (O) at (1,-.5) {}; 
\path (u2) edge[out=-90,in=90] (b1);
\path (u4) edge[out=-90,in=90] (b2);
\path (b3) edge[out=90,in=90] (b4);
\path (u5) edge[out=-90,in=-90] (u3);
\path (u6) edge[out=-90,in=90] (b6);
\path (b5) edge[out=90,in=90] (b7);
\foreach \x in {1,...,7}{
	\draw[dotted] (\x,-0.1)--(\x,-0.9);
}

\node at (5.2,0) {};
\pgftransformyshift{-2.5cm}
\foreach \x in {1,...,7}{
	\node[ve] (u\x) at (\x,1.5) {};
}
\foreach \x in {3,...,5}{
	\node[ve] (b\x) at (\x,0) {};
}

\path (b3) edge[out=90,in=90] (b4);
\path (u3) edge[out=-90,in=-90] (u5);
\path (u6) edge[out=-90,in=90] (b5);
\path (u1) edge[out=-90,in=-90] (u2);
\path (u4) edge[out=-75,in=-115] (u7);
\end{tikzpicture}= \delta \cdot \left(
\begin{tikzpicture}[baseline=(O.base),line width=0.6pt,scale=0.6]
\foreach \x in {1,...,5}{
	\node[ve] (u\x) at (\x,1.5) {};
}
\foreach \x in {2,...,4}{
	\node[ve] (b\x) at (\x,0) {};
}
\node at (1,-.2) {};
\node at (5,-.2) {};
\node (O) at (1,0.75) {};
\path (u5) edge[out=-90,in=90] (b4);
\path (u1) edge[out=-90,in=-90] (u3);
\path (u2) edge[out=-90,in=-90] (u4);
\path (b3) edge[out=90,in=90] (b2);
\end{tikzpicture} \right)\]
This composition defines in particular the structure of an associative algebra on $\B^{(\delta)}(n,n)$. This algebra is classically called the \emph{Brauer algebra}. 

\begin{defn}
	Let $V$ be an object of a symmetric monoidal category. A \emph{dual} of $V$ is an object $V^\vee$ equipped with unit and counit maps $\mathbf 1 \to V \otimes V^\vee$ and $V^\vee \otimes V \to \mathbf 1$ such that both the following compositions are identities:
	$$ V \cong \mathbf 1 \otimes V \to V \otimes V^\vee \otimes V \to V \otimes \mathbf 1 \cong V$$
	and
	$$ V^\vee \cong V^\vee \otimes \mathbf 1 \to V^\vee \otimes V \otimes V^\vee \to \mathbf 1 \otimes V^\vee \cong V^\vee.$$
	An object $V$ is \emph{self-dual} if it is equipped with a pair of maps   $\mathbf 1 \to V \otimes V$ and $V \otimes V \to \mathbf 1$ making $V$ into its own dual. It is \emph{symmetrically self-dual} if, in addition, the unit and counit are invariant under the flip map $V \otimes V \to V\otimes V$. If $V$ is dualizable, then we define the \emph{quantum dimension} of $V$ to be the element of $\End(\mathbf 1)$ given by the composition
	$$ \mathbf 1 \to V \otimes V^\vee \cong V^\vee \otimes V \to \mathbf 1.$$
\end{defn}

The following proposition can be seen as part of the diagrammatic calculus of ``string diagrams'', describing morphisms in tensor categories (see e.g.\ \cite[Chapter XIV]{kasselquantumgroups}). In this calculus the precise form of the diagrams depend on the properties of the tensor category. For example, in a symmetric monoidal category strings are allowed to cross each other freely, but in a braided monoidal category the strings must be considered as braids. If we did not insist that $V$ was \emph{symmetrically} self-dual in the following proposition, we would need to equip the strands in the Brauer algebra with orientations or framings. 

\begin{prop}\label{braueraction0}
	Let $V$ be a symmetrically self-dual object of quantum dimension $\delta$ in a $\Q$-linear symmetric monoidal category $\mathcal C$. There is a natural map $\B^{(\delta)}(n,m)\to \Hom_{\mathcal C}(V^{\otimes n},V^{\otimes m})$ which
	makes the following diagram commute:
	\[\begin{tikzcd}[cramped]
	\B^{(\delta)}(n,m) \otimes \B^{(\delta)}(m,k) \arrow[r] \arrow[d]& \B^{(\delta)}(n,k) \arrow[d]\\
	\Hom_{\mathcal C}(V^{\otimes n},V^{\otimes m}) \otimes \Hom_{\mathcal C}(V^{\otimes m},V^{\otimes k}) \arrow[r]& \Hom_{\mathcal C}(V^{\otimes n},V^{\otimes k}).
	\end{tikzcd}\]
\end{prop}

 The collection of maps $\B^{(\delta)}(n,m)\to \Hom_{\mathcal C}(V^{\otimes n},V^{\otimes m})$ are completely determined by the images of the three diagrams
 
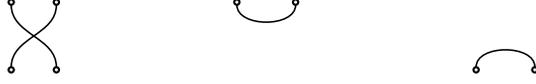
\[\begin{tikzpicture}[line width=0.6pt,scale=0.6]
 \node[ve] (b0) at (0,0) {};
 \node[ve] (u0) at (0,1.5) {};
 \node[ve] (b1) at (1,0) {};
 \node[ve] (u1) at (1,1.5) {};
 
 \node[ve] (u00) at (5,1.5) {};
 \node[ve] (u10) at (6.3,1.5) {};
 
 \node[ve] (b00) at (10.3,0) {};
 \node[ve] (b10) at (11.6,0) {};
 
 \path (b1) edge[out=90,in=-90] (u0);
 \path (b0) edge[out=90,in=-90] (u1);
 
 \path (b10) edge[out=90,in=90] (b00);
 \path (u00) edge[out=-90,in=-90] (u10);
 \end{tikzpicture} \]
 which generate the Brauer algebras in an appropriate sense; these diagrams are mapped to the flip map $V \otimes V \to V\otimes V$, the counit $V \otimes V \to \mathbf 1$ and the unit $\mathbf 1 \to V \otimes V$, respectively. 

\begin{rem}
	The proposition can be formulated using the language of PROPs: the collection $\{\B^{(\delta)}(n,m)\}_{n,m \geq 0}$ is a PROP, and $V$ is an algebra over this PROP in the category $\mathcal C$. 
\end{rem}

Let $sV$ denote the symplectic vector space $V$ considered as a $\Z/2$-graded vector space concentrated in \emph{odd} degree. Let $\mathbf 1$ be the monoidal unit in this category, i.e.\ the vector space $\Q$ placed in even degree. The symplectic form defines ``contraction'' and ``insertion'' maps
$$ sV \otimes sV \to \mathbf 1 \qquad \text{and} \qquad \mathbf 1 \to sV \otimes sV.$$ Taking into account the Koszul sign rule for $\Z/2$-graded vector spaces, both these maps are now $\sym_2$-invariant; that is, by shifting $V$ into odd degree, we have converted the symplectic form to a symmetric bilinear form. Equivalently, $sV$ is symmetrically self-dual in the category of $\Z/2$-graded vector spaces. The quantum dimension of $sV$ is $-2g$.


\begin{cor}\label{braueraction1}Let $V$ be a symplectic vector space of dimension $2g$. The map which sends an $(n,m)$-Brauer diagram to a morphism $(sV)^{\otimes n} \to (sV)^{\otimes m}$ makes the following diagram commute:
	\[\begin{tikzcd}[cramped]
	\B^{(-2g)}(n,m) \otimes \B^{(-2g)}(m,k) \arrow[r] \arrow[d]& \B^{(-2g)}(n,k) \arrow[d]\\
	\Hom_{\Sp(V)}((sV)^{\otimes n},(sV)^{\otimes m}) \otimes \Hom_{\Sp(V)}((sV)^{\otimes m},(sV)^{\otimes k}) \arrow[r]& \Hom_{\Sp(V)}((sV)^{\otimes n},(sV)^{\otimes k}),
	\end{tikzcd}\]
	and the vertical maps are surjective.
\end{cor}

\begin{proof}After the previous proposition, we only need to explain surjectivity. Surjectivity is equivalent to the statement that the space of symplectic invariant tensors inside $(sV)^{\otimes 2n}$ is spanned by the classes obtained by inserting the symplectic form $n$ times, which follows from Weyl's decomposition of $V^{\otimes 2n}$ into irreducible representations described in the previous section. Alternatively, surjectivity is part of the first fundamental theorem of invariant theory for the symplectic group.
	\end{proof}

In particular, the Brauer algebra $\B^{(-2g)}(n,n)$ surjects onto the centralizer $\End_{\Sp(V)}((sV)^{\otimes n})$. Moreover, there is an isomorphism $\End_{\Sp(V)}((sV)^{\otimes n}) \cong \End_{\Sp(V)}(V^{\otimes n})$ given by desuspending and carefully inserting signs. This isomorphism is described explicitly by Hanlon and Wales \cite[Theorem 2.10]{hanlonwales}. Let us explain why their result gives such an isomorphism. Hanlon and Wales define two versions of Brauer algebra, $\mathfrak A_f^{(x)}$ and $\mathfrak B_f^{(x)}$, for any natural number $f$ and any parameter $x$ in the ground field. These algebras act naturally on the $f$th tensor power of a vector space of dimension $x$ equipped with a symmetric or antisymmetric bilinear form, respectively. They show by a direct calculation that there is an isomorphism $\mathfrak A_f^{(x)} \cong \mathfrak B_f^{(-x)}$ for all $f$, $x$: in our terms, this isomorphism arises from the fact that the functor $V \mapsto sV$ maps a vector space of dimension $x$ to a space of dimension $-x$, and converts a symmetric bilinear form to an antisymmetric one and vice versa. Our algebra $\B^{(-2g)}(n,n)$ is identical with their $\mathfrak A_n^{(-2g)}$.

 The Brauer algebra contains the group algebra $\Q[\sym_n]$ --- as it should, since the centralizer of $\Sp(V)$ acting on $V^{\otimes n}$ should contain the centralizer of $\GL(V)$ --- as the subalgebra consisting of $(n,n)$-Brauer diagrams in which all strands are vertical, i.e. go from the top row to the bottom row. The inclusion $\Q[\sym_n] \to \B^{(-2g)}(n,n)$ has a left inverse, given by mapping any diagram containing a horizontal strand to zero; one checks that the subspace spanned by all diagrams containing a horizontal strand is an ideal.

\begin{thm}[Symplectic Schur--Weyl duality]\label{symplecticschurweyl}
	Let $V$ be a symplectic vector space of dimension $2g$.
	\begin{enumerate}
		\item The image of the Brauer algebra $\B^{(-2g)}(n,n)$  in $\End_\Q(V^{\otimes n})$ is the centralizer of $\Sp(V)$, and vice versa.
		\item There is an isomorphism
		$$ V^{\otimes n} \cong \bigoplus_{\substack{\vert \lambda \vert \leq n \\ \vert \lambda \vert \equiv 2 \pmod n}}  V_{\llambda}\otimes \beta_{\lambda,n}^\vee$$
		where $\beta_{\lambda,n}$ denotes the simple module over the Brauer algebra $\B^{(-2g)}(n,n)$ corresponding to $\lambda$. 
		\item For $\vert \lambda \vert = n$, the representation $\beta_{\lambda,n}$ coincides with the representation $\sigma_{\lambda^T}$ of $\sym_n$, considered as a module over the Brauer algebra via the map $\B^{(-2g)}(n,n) \to \Q[\sym_n]$ which sends any diagram containing a horizontal strand to zero. 
	\end{enumerate} 
\end{thm}

Part (2) follows from (1), given a description of how the split semisimple algebra $\End_{\Sp(V)}(V^{\otimes n})$ decomposes into simple algebras \cite[Corollary 3.5]{wenzl}.

\begin{rem}It may seem strange that the representation $\sigma_{\lambda^T}$, rather than $\sigma_\lambda$, appears in part (3) of Theorem \ref{symplecticschurweyl}. Indeed, we have seen from Weyl's construction that $V^{\otimes n}$, when decomposed into irreducible representations of $\Sp(V)\times \sym_n$, should contain the summands $V_\llambda \otimes \sigma_\lambda^\vee$ for $\vert \lambda \vert = n$. But Theorem \ref{symplecticschurweyl} says that $V^{\otimes n}$ contains the summands $V_\llambda \otimes \beta_{\lambda,n}^\vee$, and that $\beta_{\lambda,n} \cong \sigma_{\lambda^T}$ when $\vert \lambda \vert = n$. So why isn't this a contradiction?	The reason is that when $\sym_n$ acts on $V^{\otimes n}$ via the composition
	$$ \Q[\sym_n] \hookrightarrow \B^{(-2g)}(n,n) \to \End_\Q((sV)^{\otimes n}) \cong \End_\Q(V^{\otimes n}),$$
	then this action is \emph{not} equal to the standard action of $\sym_n$ on $V^{\otimes n}$ by permuting the factors; instead, one obtains the standard action twisted by the sign representation. So the isomorphism $\beta_{\lambda,n} \cong \sigma_{\lambda^T}$ does hold when $\beta_{\lambda,n}$ is considered as a $\Q[\sym_n]$-module by restriction of scalars, but this is not the same as the $\Q[\sym_n]$-module structure obtained by the natural action of $\sym_n$ on $V^{\otimes n}$.
	
	The conventions are more natural when $V$ is placed in odd degree: the composition $\Q[\sym_n] \to \B^{(-2g)}(n,n) \to \End_\Q((sV)^{\otimes n})$ does give the standard action of $\sym_n$ on $(sV)^{\otimes n}$, which now takes the Koszul sign rule into account. Thus $(sV)^{\otimes n}$ will contain $V_\llambda \otimes \sigma_{\lambda^T}^\vee$ as a summand, placed in odd/even degree according to whether $n$ is odd/even. 	We caution the reader that the calculations of this paper will require some care to be taken to tensor with the sign representation when appropriate, in particular when passing between Chow groups and cohomology groups. 
\end{rem}

\begin{rem}If we want to decompose $V^{\otimes n}$ into irreducible representations of $\Sp(V) \times \sym_n$, then we may start from the usual Schur--Weyl duality (which gives a decomposition into irreducible representations of $\GL(V) \times\sym_n$), and apply a branching formula for $\Sp(V) \subset \GL(V)$. Equivalently we could start with the symplectic 
	Schur--Weyl duality (which gives a decomposition into $\Sp(V) \times \B^{(-2g)}(n,n)$-representations) and try to determine how the modules $\beta_{\mu,n}$ over $\B^{(-2g)}(n,n)$ decompose into sums of Specht modules under restriction of scalars to $\Q[\sym_n]$: note that if 
	$$ \Res^{\GL(V)}_{\Sp(V)} V_\lambda \cong \bigoplus_\mu a_{\mu}^{ \lambda}\,\, V_{\langle \mu \rangle}$$
	for some integers $a_{\mu}^{ \lambda}$, then $\beta_{\mu,n} \cong \bigoplus_{\vert\lambda\vert=n} a_\mu^\lambda \, \sigma_{\lambda^T}$ as $\Q[\sym_n]$-modules. Then the decomposition of $V^{\otimes n}$ reads
	$$V^{\otimes n} \cong \bigoplus_{\vert \lambda \vert = n} \bigoplus_\mu a_{\mu}^{ \lambda}\,\,  V_{\langle \mu  \rangle} \otimes \sigma_\lambda^\vee. $$
	The discussion in the preceding paragraphs says that that $a_{\mu}^{ \lambda} \neq 0$ only for $\vert \mu \vert \equiv \vert \lambda \vert \pmod 2$ and $\vert \mu \vert < \vert \lambda \vert$, with the sole exception of $a_{\lambda}^{ \lambda} = 1$. 	The problem of calculating the coefficients $a_\mu^\lambda$ was first solved by Littlewood \cite{littlewood} and Newell \cite{newell}, and many subsequent authors have given  methods for computing them. 
	%
\end{rem}

\subsection{Projectors}
Given the above, it is natural to ask for an analogue of Young symmetrizers in the Brauer algebra. That is, one would like idempotents $\pi_{\lambda,n} \in \B^{(-2g)}(n,n)$ such that the image of $\pi_{\lambda,n}$ acting on $V^{\otimes n}$ is the irreducible summand $ V_\llambda \otimes \beta_{\lambda,n}^\vee$. The question of how to find such idempotents $\pi_\lambda$ was posed already by Weyl. Nevertheless, no explicit construction was known until Nazarov \cite{nazarovyangians} gave a simple formula describing $\pi_{\lambda,n}$ in the most interesting case $\vert\lambda\vert=n$. Although the statement of the result is elementary and involves only very classical representation theory, the proof proceeds through the theory of quantum groups.

The results of Section \ref{lowgenuscalculations}, and some of the examples in Section \ref{examplesection}, rely on computer calculations which require us to have explicit formulas for the idempotents $\pi_{\lambda,n}$ for $\vert \lambda \vert=n$. However, the reader does not need to know the precise expression for $\pi_{\lambda,n}$ to follow the arguments, only that such a formula exists. Nevertheless we state Nazarov's theorem here for completeness. For a partition $\lambda$, we define the \emph{row tableau} associated to $\lambda$ to be the Young tableau given by filling in the numbers $1,2,\ldots,n$ in the Ferrers diagram so that the first row gets the numbers $1,2,\ldots,\lambda_1$, the second row gets the numbers $\lambda_1+1,\lambda_1+2,\ldots,\lambda_1 +\lambda_2$, and so on. We define the \emph{content} of a box in the $i$th row and $j$th column of the Ferrers diagram to be $j-i$. For $k \in \{1,\ldots,n\}$, we define the number $c_k(\lambda)$ to be the content of the box labeled ``$k$'' in the row tableau corresponding to $\lambda$.  

For any $1 \leq i,j \leq n$, let $B_{ij}$ be the element of $\B^{(-2g)}(n,n)$ corresponding to contracting and inserting the $i$th and $j$th tensor factors with the symplectic form. That is, it has $n-2$ vertical strands, and two horizontal ones: one connecting the $i$th and $j$th ``inputs'', and one connecting the $i$th and $j$th ``outputs''. 

\begin{thm}[Nazarov] \label{nazarov}
	For any partition $\lambda$ of $n$, define
	$$ \pi_{\lambda,n} = \prod_{k,l} \left(1 + \frac{B_{kl}}{2g+1 + c_k(\lambda) + c_l(\lambda)}\right) \cdot c_{\lambda^T} \in \B^{(-2g)}(n,n).$$
	Here the product ranges over all pairs $1 \leq k < l \leq n$ such that the boxes labeled $k$ and $l$ are in distinct rows of the row tableau associated to $\lambda$. Since the operators $B_{kl}$ do not commute, this must be interpreted as an \emph{ordered} product: we order the terms in the product lexicographically by $(k,l)$. Finally, $c_{\lambda^T}$ is a Young symmetrizer. The image of $\pi_{\lambda,n}$ acting on $(sV)^{\otimes n}$ 
	is the summand $ V_{\llambda} \otimes \sigma_{\lambda^T}^\vee$, which is placed in odd degree if $n$ is odd and even degree if $n$ is even.
\end{thm}

\begin{rem}
	For our purposes it would be enough to have formulas for such idempotents in the smaller algebra $\End_{\Sp(V)}((sV)^{\otimes n})$, which a priori do not have to lift to an idempotent of $\B^{(-2g)}(n,n)$. This means that in principle we could have used results of Ram and Wenzl \cite{ramwenzl} instead of Nazarov's theorem.
\end{rem}


\section{A result of Ancona}

\newcommand{\CC}{\mathcal C}

\subsection{Schur functors}
Let $\CC$ be a $\Q$-linear symmetric monoidal category, and let $M \in \ob \CC$. Then $M^{\otimes n}$ has an action of the group algebra $\Q[\sym_n]$, in the sense that there is a homomorphism of $\Q$-algebras $\Q[\sym_n] \to \Hom_\CC(M^{\otimes n},M^{\otimes n})$. 

Let us suppose moreover that $\CC$ is pseudo-abelian, i.e.\ that every idempotent endomorphism in $\CC$ has an image. For $\vert\lambda\vert=n$, let $\pi_\lambda \in \Q[\sym_n]$ be a Young symmetrizer corresponding to $\lambda$, and define $\mathsf S^\lambda(M)$ to be the image of the idempotent $\pi_\lambda$ acting on $M^{\otimes n}$. We call $\mathsf S^\lambda$ the \emph{Schur functor} corresponding to $\lambda$. Then there is a decomposition \cite{delignecategoriestensorielles}
$$ M^{\otimes n} = \bigoplus_{\vert \lambda \vert = n}  \mathsf S^\lambda(M) \otimes\sigma_\lambda^\vee ,$$
where $\sigma_\lambda$ denotes the representation of the symmetric group corresponding to $\lambda$. When $\CC$ is the category of finite dimensional $\Q$-vector spaces, this is the decomposition of $M^{\otimes n}$ given by Schur--Weyl duality, described in Section \ref{schurweylsection}.

The key fact used in proving this result is that $\Q[\sym_n]$ is a semisimple algebra and in fact a product of matrix algebras over $\Q$: there is an isomorphism $\Q[\sym_n] \cong \prod_{\vert \lambda \vert = n} \End_\Q(\sigma_\lambda)$.

\subsection{Brauer algebra action}

We will need to generalize Proposition \ref{braueraction0} to a weaker notion of symmetrically self-dual object. 
An object $L$ of $\CC$ is called \emph{invertible} if the functor $- \otimes L$ is an equivalence of categories. If this is the case then $L$ is dualizable, the quasi-inverse is given by tensoring with $L^\vee$, and the maps $\mathbf 1 \to L \otimes L^\vee$ and $L \otimes L^\vee \to \mathbf 1$ are isomorphisms. We say that  $L$ is \emph{even} or \emph{odd} if $\sym_n$  acts on $\Hom_{\mathcal C}(L^{\otimes n},L^{\otimes n}) \cong \Hom_{\mathcal C}(\mathbf 1,\mathbf 1)$ by the trivial representation or the sign representation, respectively. 

We say that $M \in \ob \CC$ is \emph{weakly self-dual} if there is an even invertible object $L \in \ob \CC$, and unit and counit maps
$$ L \to M \otimes M \qquad M \otimes M \to L$$
such that the compositions
$$ M \otimes L \to M \otimes M \otimes M \to L \otimes M$$
and
$$ L \otimes M \to M \otimes M \otimes M \to M\otimes L$$
both equal the swap map. This implies that $M^\vee \cong M \otimes L^\vee$.  We call $M$ \emph{weakly symmetrically self-dual} if moreover the unit and counit maps are invariant under the swap map $M \otimes M \to M \otimes M$. 
%
We omit the proof of the following result, which generalizes Proposition \ref{braueraction0} and is an exercise in the diagrammatic calculus for rigid symmetric monoidal categories. 

\begin{prop}\label{braueraction}
	Let $M$ be a weakly symmetrically self-dual object of $\CC$ of quantum dimension $\delta$. There is an action of $\B^{(\delta)}(n,n)$ on $M^{\otimes n}$, under which the subalgebra $\Q[\sym_n]$ acts on $M^{\otimes n}$ in the usual way, and a Brauer diagram of the form
	\begin{tikzpicture}[baseline=(O.base),line width=0.6pt,scale=0.4]
	\foreach \x in {1,...,2}{
		\node[ve] (u\x) at (\x,1) {};
		\node[ve] (b\x) at (\x,0) {};
	}
	\node (O) at (1,.25) {}; 
	\path (u2) edge[out=-90,in=-90] (u1);
	\path (b1) edge[out=90,in=90] (b2);
	\end{tikzpicture} acts as the composition
	$$ M \otimes M \to L \to M \otimes M. $$
\end{prop} 
A more general statement is that any element of $\B^{(\delta )}(m,n)$ gives a well defined morphism in $\Hom_\CC(M^{\otimes m} \otimes L^{n-m}, M^{\otimes n}) = \Hom_\CC(M^{\otimes m},M^{\otimes n} \otimes L^{m-n})$, in a way compatible with composition. 

\begin{rem}\label{matrixalgebraremark}
	Proposition \ref{braueraction} does not in general lead to a decomposition of $M^{\otimes n}$ into summands indexed by irreducible representations of the Brauer algebra,  even over $\overline \Q$. The reason is that the algebra $\B^{(\delta)}(n,n)$ is not in general semisimple, which was the crucial property of $\Q[\sym_n]$ used for defining the decomposition of $M^{\otimes n}$ in terms of Schur functors. In the case of $\B^{(-2g)}(n,n)$, which acts naturally on $V^{\otimes n}$ for $V$ a symplectic vector space of dimension $2g$, what \emph{does} hold is that $\End_{\Sp(2g)}(V^{\otimes n})$, i.e.\ the image of $\B^{(-2g)}(n,n)$ in $\End_\Q(V^{\otimes n})$, is a product of matrix algebras over $\Q$.  \end{rem}

\subsection{Self-products of abelian schemes}Let $f \colon A \to S$ be an abelian scheme of relative dimension $g$, where we assume $S$ smooth and connected. Let $\V$ be the local system $R^1 f_\ast\Q$ on $S$ of rank $2g$. Then $\V$ is defined by a homomorphism $\pi_1(S,x_0) \to \Sp(2g,\Q)$. The Brauer algebra $\B^{(-2g)}(n,n)$ acts on the $n$-fold tensor power of the defining representation $V$ of $\Sp(2g)$, and hence also on $\V^{\otimes n}$. As explained in Theorem \ref{symplecticschurweyl} the Brauer algebra action gives rise to a decomposition
$$V^{\otimes n} \cong \bigoplus_{\substack{\vert \lambda \vert \leq n \\ \vert \lambda \vert \equiv n \pmod 2}}  V_{\langle \lambda \rangle} \otimes \beta_{\lambda,n}^\vee$$
which then gives us also a decomposition of $\V^{\otimes n}$, i.e.\ $\V^{\otimes n} \cong \bigoplus_{\substack{\vert \lambda \vert \leq n \\ \vert \lambda \vert \equiv n \pmod 2}}  \V_{\langle \lambda \rangle} \otimes \beta_{\lambda,n}^\vee.$

The next result is a special case of the main theorem of \cite{ancona}. See also \cite{moonenrefineddecomposition}.

\begin{thm}\label{ancona} The above decomposition of $\V^{\otimes n}$ lifts to the category of Chow motives over $S$: 
	$$ h^1(A/S)^{\otimes n} \cong \bigoplus_{\substack{\vert \lambda \vert \leq n \\ \vert \lambda \vert \equiv n \pmod 2}}  h^1(A/S)_{\langle \lambda \rangle} \otimes \L^{(n - \vert \lambda \vert)/2}\otimes \beta_{\lambda,n}^\vee.$$
\end{thm}


We note that the action of $\B^{(-2g)}(n,n)$ lifts to an action on $h^1(A/S)^{\otimes n}$. This follows from Proposition \ref{braueraction} given that the cup product $h^1(A/S) \otimes h^1(A/S)\to \L$ makes $h^1(A/S)$ weakly symmetrically self-dual, and that $\dim h^1(A/S)=-2g$. To see that $\dim h^1(A/S)=-2g$, observe that the quantum dimension of an object is preserved by any strict symmetric monoidal functor. We apply this to $\omega \colon \Mot_S \to \mathsf{grVect}_\Q$ given by 
$\omega(M) = \bigoplus_i \mathcal H^i(\mathsf{real}\, M)_{x_0}$, where $x_0 \in S$ is an arbitrary point. Since $\Mot_S(\mathbf 1,\mathbf 1) = \mathsf{grVect_\Q}(\mathbf 1,\mathbf 1) = \Q$ we have $\dim h^1(A/S) = \dim \omega(h^1(A/S))$. But $\omega (h^1(A/S))$ is the $2g$-dimensional vector space $H^1(A_{x_0},\Q)$ placed in degree $1$, so its dimension in the sense of graded vector spaces is $-2g$.

As explained in Remark \ref{matrixalgebraremark}, the obstruction to obtaining a result like Theorem \ref{ancona} in a general rigid symmetric monoidal category is that the algebra $\B^{(-2g)}(n,n)$ does not split as a product of matrix algebras; only its quotient $\End_{\Sp(2g)}(V^{\otimes n})$ does. The key point is then that the action of $\B^{(-2g)}(n,n)$ on $h^1(A/S)^{\otimes n}$ factors through $\End_{\Sp(2g)}(V^{\otimes n})$. In Ancona's paper \cite{ancona} this is proven using O'Sullivan's results on symmetrically distinguished cycles on abelian varieties \cite{osullivan}. An alternative proof (cf. \cite{anconathesis}, \cite[Theorem 4.8]{lehrerzhang}) proceeds by using invariant theory to show that the kernel of the homomorphism of PROPs of Proposition \ref{braueraction1} is the PROP-ideal generated by a single element of $\B^{(-2g)}(g+1,g+1)$ whose vanishing is equivalent to $\wedge^{2g+2}(V)=0$. Then the fact that this single relation holds also on the level of Chow groups is equivalent to the fact that $h^1(A/S)$ is a finite dimensional motive in the sense of Kimura. Either way it is clear that the result is at present quite special to abelian varieties.

\section{The K\"unneth decomposition of the tautological ring}\label{mainsection}

Let $p \colon C_g \to M_g$ be the universal genus $g$ curve, and $C_g^n$ the $n$-fold fibered power of $C_g$ over $M_g$.  There are $n$ natural line bundles $L_i$ on $C_g^n$; the fiber of $L_i$ over a moduli point is given by the cotangent space of the curve at the $i$th marking. We denote the first Chern class of $L_i$ by $\psi_i$. Thus $\psi_i$ is pulled back from $C_g$ along the map $C_g^n \to C_g$ that forgets all markings except the $i$th. 

We make the definition $\kappa_d = p_\ast \psi_1^{d+1} \in \CH^{d}(M_g)$. In particular, $\kappa_{-1}=0$ and $\kappa_0 = (2g-2)$. We denote by the same symbol $\kappa_d$ also the pullback of this class to $C_g^n$. 

For any distinct elements $i,j \in \{1,\ldots,n\}$ we denote by $\Delta_{ij} \in \CH^1(C_g^n)$ the class of the diagonal locus where the $i$th and $j$th marked point coincide with each other.

\begin{defn}The \emph{tautological ring} $R^\bullet(C_g^n)$ is the subring of $\CH^\bullet(C_g^n)$ generated by all $\psi$-classes, $\kappa$-classes and diagonal classes. The \emph{tautological cohomology ring} $\RH^\bullet(C_g^n)$ is the image of the tautological ring inside $H^\bullet(C_g^n,\Q)$ under the cycle class map. (The grading of $\RH^\bullet(C_g^n)$ is twice that of $R^\bullet(C_g^n)$, so that $\RH^k(C_g^n) \subset H^k(C_g^n,\Q)$.)\end{defn}

The generators for the tautological rings satisfy the following relations for all $i$, $j$ and $k$:
\begin{align}\label{basicrelations}\begin{split}
\Delta_{ij}\Delta_{ik} & = \Delta_{ij}\Delta_{jk}\\
\Delta_{ij}\psi_i &= \Delta_{ij}\psi_j \\
\Delta_{ij}^2 &= -\Delta_{ij}\psi_i\end{split}
\end{align}
The first two are geometrically obvious, and the third one is a consequence of the excess intersection formula. 

\begin{defn}\label{s-defn}
	Let $S_{n}^\bullet$ be the commutative graded $\Q$-algebra generated by classes $\Delta_{ij}$ and $\psi_i$ of degree $1$ (where $i$ and $j$ range from $1$ to $n$ and are distinct) and $\kappa_d$ of degree $d$ for all $d \geq 1$, modulo the above three relations.
\end{defn}

\begin{rem}
	By a ``commutative graded'' algebra (as opposed to a ``graded commutative'' algebra) we mean an algebra in which $x \cdot y = y \cdot x$ for all $x, y$, regardless of their degree; we do not impose the Koszul sign rule $x \cdot y = (-1)^{\vert x \vert \cdot \vert y \vert} y \cdot x$.
\end{rem}

For each $g \geq 2$, there is a natural surjection
$$ S_{n}^\bullet \to R^\bullet(C_g^n),$$
and describing the tautological ring $R^\bullet(C_g^n)$ is equivalent to describing the kernel of this surjection. The algebra $S_{n}^\bullet$ plays the same role for the study of $R^\bullet(C_g^n)$ as the \emph{strata algebra} does for the study of $R^\bullet(\overline M_{g,n})$, cf.\ e.g.\ \cite[0.3]{pandharipandepixtonzvonkine}. 

\begin{rem}
	As mentioned in the introduction, tautological classes are usually considered on the Deligne--Mumford spaces. In that case the tautological rings $R^\bullet(\overline M_{g,n})$ can be defined, following Faber and Pandharipande \cite{fpjems}, as the smallest collection of unital subrings of $\CH^\bullet(\overline M_{g,n})$ closed under pushforward along the gluing maps
	$$ \MM_{g,n+2} \to \MM_{g+1,n} \qquad \text{and}\qquad \MM_{g,n+1} \times \MM_{g',n'+1} \to \MM_{g+g',n+n'} $$
	and the forgetful maps
	$$ \MM_{g,n+1} \to \MM_{g,n}.$$
	Although it is not imposed in the definition, it turns out that the tautological rings are also closed under pullback along the same maps. A similar characterization can be given of the tautological rings $R^\bullet(C_g^n)$. For any function $\phi \colon \{1,\ldots,n\}\to \{1,\ldots,m\}$ there is a map $C_g^m \to C_g^n$, 
	$$ (C,x_1,x_2, \ldots,x_m) \mapsto (C,x_{\phi(1)},\ldots,x_{\phi(n)}).$$
	We call all maps of this form \emph{tautological}. Then it is not hard to see that the system of tautological rings $R^\bullet(C_g^n)$ can be defined to be the smallest collection of unital subrings closed under pushforward along all tautological maps, and it turns out a posteriori that the tautological rings are also closed under pullback along the same maps.  \end{rem}

\subsection{The K\"unneth decomposition of the universal curve}As explained in Section \ref{kunnethsection}, any choice of a cycle $\z \in \CH^1(C_g)$ of degree $1$ on each fiber of $p \colon C_g \to M_g$ gives rise to a decomposition of the relative Chow motive:
$$ h(C_g/M_g) = h^0(C_g/M_g) \oplus h^1(C_g/M_g) \oplus h^2(C_g/M_g).$$
Since $\CH^1(C_g) = \Q\{\kappa_1,\psi_1\}$, and $\kappa_1$ vanishes on the fibers of $p$, the only possibilities we have are $\z = \frac{1}{2g-2} \psi_1 + \mathrm{(const.)}\cdot \kappa_1$. Regardless of the constant we get $\z' = \z - \frac{1}{2}p^\ast p_\ast \z^2 = \frac{1}{2g-2} \psi_1 - \frac{1}{2(2g-2)^2}\kappa_1$. Hence without making any choices we get projectors $\pi_0$, $\pi_1$ and $\pi_2$ acting on $h(C_g/M_g)$, defined by
\begin{align*}
\pi_0 & = \frac{1}{2g-2}\psi_1 - \frac{1}{2(2g-2)^2} \kappa_1,\\
\pi_1 & = \Delta_{12} - \frac{1}{2g-2}(\psi_1+\psi_2) + \frac{1}{(2g-2)^2}\kappa_1, \\
\pi_2 & = \frac{1}{2g-2} \psi_2 - \frac{1}{2(2g-2)^2} \kappa_1.
\end{align*}
We have isomorphisms $\mathbf 1 \cong h^0(C_g/M_g)$ and $\mathbb L \cong h^2(C_g/M_g)$, where $\mathbf 1$ and $\mathbb L$ denote the unit object and the Lefschetz motive in the category of Chow motives over $M_g$. Thus the interesting motive is $h^1(C_g/M_g)$. 

We may form the Chow groups of these relative motives: there is an isomorphism
$$ \CH^k(C_g) = \CH^k(M_g,h(C_g/M_g))= \bigoplus_{i=0}^2 \CH^k(M_g,h^i(C_g/M_g)),$$
where 
\begin{align*}
\CH^k(M_g,h^0(C_g/M_g)) & = \mathrm{Im}(\pi_0 \colon \CH^k(C_g)\to \CH^k(C_g)) \cong \CH^k(M_g) \\
\CH^k(M_g,h^1(C_g/M_g)) & = \mathrm{Im}(\pi_1 \colon \CH^k(C_g)\to \CH^k(C_g)) \\
\CH^k(M_g,h^2(C_g/M_g)) & = \mathrm{Im}(\pi_2 \colon \CH^k(C_g)\to \CH^k(C_g)) \cong \CH^{k-1}(M_g).
\end{align*}
The isomorphism $\CH^k(M_g,h^0(C_g/M_g)) \cong \CH^k(M_g)$ is induced by the pullback $p^\ast$, and the isomorphism $\CH^k(M_g,h^2(C_g/M_g)) \cong \CH^{k-1}(M_g)$ by the proper pushforward $p_\ast$. Informally, the Chow groups of $h^1(C_g/M_g)$ capture the parts of the Chow groups of $C_g$ that do not come from the base $M_g$.

Now let us consider the $n$-fold fibered power $C^n_g \to M_g$. Then $h(C_g^n/M_g) = h(C_g/M_g)^{\otimes n}$, so our decomposition yields an equally canonical isomorphism
$$ h(C_g^n/M_g) = \bigoplus_{i_1,\ldots,i_n \in \{0,1,2\}} \bigotimes_{j=1}^n h^{i_j} (C_g/M_g).$$
We call this the \emph{relative K\"unneth decomposition} of $h(C_g^n/M_g)$. By extension, we will also refer to $ \CH^k(C_g^n/M_g) = \bigoplus_{i_1,\ldots,i_n \in \{0,1,2\}} \CH^k(M_g,\bigotimes_{j=1}^n h^{i_j} (C_g/M_g))$ as the relative K\"unneth decomposition of the Chow groups of $C_g^n$.

For any $i_1,\ldots,i_n \in \{0,1,2\}$ we get a projector $\pi_{i_1} \times \ldots \times \pi_{i_n}$ acting on $h(C_g^n/M_g)$ with image $\bigotimes_{j=1}^n h^{i_j} (C_g/M_g)$. In particular this projector acts by correspondences on $\CH^k(C_g^n)$ with image $\CH^k(M_g,\bigotimes_{j=1}^n h^{i_j} (C_g/M_g))$. We write $\pi_1^{\times n}$ for the projector $\pi_1 \times \pi_1 \times \ldots \times \pi_1$.

\subsection{Tautological maps} Let us consider how the decomposition just defined behaves under the tautological maps between the moduli spaces $C_g^n$. 

\subsubsection{Cross product}
\label{crossproduct}
The isomorphism $h(C_g^n/M_g) \otimes h(C_g^m/M_g) \cong h(C_g^{n+m}/M_g)$ yields \emph{cross product} maps
$$ \CH^k(C_g^n) \otimes \CH^l(C_g^m) \to \CH^{k+l}(C_g^{n+m});$$
explicitly, $\alpha \times \beta = \mathrm{pr}_1^\ast(\alpha) \cdot \mathrm{pr}_2^\ast(\beta)$, where $\mathrm{pr}_1$ and $\mathrm{pr}_2$ denote projections onto the first $n$ and last $m$ factors, respectively. 

Since the relative K\"unneth decomposition of $C_g^{n+m}$ is the tensor product of the relative K\"unneth decompositions of $C_g^n$ and $C_g^m$, it follows that the cross product maps are compatible with the projectors $\pi_i$ in a strong sense: for $i_1,\ldots,i_n \in \{0,1,2\}$ and $j_1,\ldots,j_m \in \{0,1,2\}$ we have
$$ (\pi_{i_1} \times \ldots \times \pi_{i_n}) \circ \alpha \times (\pi_{j_1} \times \ldots \times \pi_{j_m}) \circ \beta = (\pi_{i_1} \times \ldots \pi_{i_n} \times \pi_{j_1} \times \ldots \times \pi_{j_m}) \circ (\alpha \times \beta).$$

\subsubsection{Forgetful maps} Let $p \colon C_g^{n+m} \to C_g^n$ be the map that forgets the last $m$ markings. Considering $p$ as a correspondence gives maps of Chow motives
$$ h(C_g^n/M_g) \to h(C_g^{n+m}/M_g) \qquad \text{and} \qquad h(C_g^{n+m}/M_g) \to h(C_g^n/M_g) \otimes \L^{m},$$
which upon taking Chow groups gives the maps
$$ p^\ast \colon \CH^k(C_g^n) \to \CH^k(C_g^{n+m}) \qquad \text{and}\qquad  p_\ast \colon \CH^k(C_g^{n+m}) \to \CH^{k-m}(C_g^n).$$ 
Now the map $ h(C_g^n/M_g) \to h(C_g^{n+m}/M_g)$ coincides with the composition $h(C_g^n/M_g) \cong h(C_g^n/M_g) \otimes h^0(C_g/M_g)^{\otimes m} \subset h(C_g^{n+m}/M_g)$, and $h(C_g^{n+m}/M_g) \to h(C_g^n/M_g) \otimes \L^{m}$ coincides with the composition $h(C_g^{n+m}/M_g) \twoheadrightarrow h(C_g^{n}/M_g) \otimes h^2(C_g/M_g)^{\otimes m} \cong h(C_g^n/M_g) \otimes \L^m$. It follows that the maps $p^\ast$ and $p_\ast$ are also compatible with the relative K\"unneth decomposition of Chow groups: 
\begin{itemize}
	\item The map $p^\ast$ is given by mapping each summand $(\pi_{i_1} \times \ldots \times \pi_{i_n}) \circ \CH^k(C_g^n)$ isomorphically onto the summand $(\pi_{i_1} \times \ldots \times \pi_{i_n} \times \pi_0 \times \ldots \times \pi_0) \circ \CH^k(C_g^{n+m})$. This can also be seen from the fact that $p^\ast$ is given by cross product with the class $1 \in \CH^0(C_g^m)$.
	\item The map $p_\ast$ maps each summand  $(\pi_{i_1} \times \ldots \times \pi_{i_n} \times \pi_2 \times \ldots \times \pi_2) \circ \CH^k(C_g^{n+m})$ isomorphically onto the summand $(\pi_{i_1} \times \ldots \times \pi_{i_n}) \circ \CH^{k-m}(C_g^n)$, and $p_\ast$ vanishes on all summands not of this form.	
\end{itemize}

\subsubsection{Diagonals}\label{cupproduct} The diagonal $C_g \to C_g^2$, considered as a correspondence, defines a map of Chow motives $h(C_g/M_g) \otimes h(C_g/M_g) \to h(C_g/M_g)$. This is the cup product on the level of Chow motives. Forming the relative K\"unneth decomposition on both sides, we see that the cup product is the sum of maps $h^i(C_g/M_g) \otimes h^j(C_g/M_g) \to h^k(C_g/M_g)$. 

We caution the reader that this is \emph{not} in general a multiplicative decomposition, in the sense of \cite{voisindecompositionbook}: that is, the maps $h^i(C_g/M_g) \otimes h^j(C_g/M_g) \to h^k(C_g/M_g)$ are \emph{not} only nonzero for $i+j=k$. To see this, note that if $\delta \subset C_g^3$ denotes the small diagonal, considered as a correspondence $C_g^2 \vdash C_g$, then the decomposition is multiplicative if and only if 
$$ \pi_k \circ \delta \circ (\pi_i \times \pi_j) = 0$$
for $i+j\neq k$. Now we have
\begin{align*}
\pi_k \circ \delta \circ (\pi_i \times \pi_j) & = (p_{126})_\ast(p_{13}^\ast (\pi_i) \cdot p_{24}^\ast(\pi_j) \cdot \Delta_{345} \cdot p_{56}^\ast(\pi_k)) \\ & = (p_{456})_\ast(\Delta_{123} \cdot p_{14}^\ast (\pi_i^t) \cdot p_{25}^\ast(\pi_j^t) \cdot p_{46}^\ast (\pi_k)) \\
&= (\pi_{2-i} \times \pi_{2-j} \times \pi_k) \circ \Delta_{123}.
\end{align*}
Thus we get a more symmetric condition for the decomposition to be multiplicative: we must have $(\pi_a \times \pi_b \times \pi_c) \circ \Delta_{123}=0$ for $a+b+c \neq 4$. In particular, e.g. the nonvanishing of the Gross--Schoen cycle (Example \ref{grossschoen}) implies that the decomposition is not multiplicative.

However, let us also remark that when $g=2$, we do have that $(\pi_a \times \pi_b \times \pi_c) \circ \Delta_{123}=0$ for $a+b+c \neq 4$, and the decomposition is multiplicative. More generally, it follows from the results of \cite{tavakolhyperelliptic} that the decomposition is multiplicative over the moduli space of \emph{hyperelliptic} curves of arbitrary genus. 

In any case, failure of decomposition to be multiplicative implies that the cup product in the algebra $\CH^\bullet(C_g^n)$ will look somewhat strange with respect to the relative K\"unneth decomposition of the Chow groups of $\CH^\bullet(C_g^n)$. The situation is analogous to the what happens in topology, when one has a multiplicative spectral sequence $E_r^{pq} \implies H^\bullet$, and the cup product on the $E_\infty$ page of the spectral sequence is different from the cup product in the algebra $H^\bullet$. On the level of Betti realizations, this is more than an analogy. The cohomology groups $H^\bullet(C_g^n,\Q)$ carry a \emph{Leray filtration}, and the associated graded $\gr_L H^\bullet(C_g^n,\Q)$ is isomorphic to the $E_\infty$ page of the Leray spectral sequence for $C_g^n \to M_g$. Our canonical decomposition of the Chow motive $h(C_g^n/M_g)$ gives, on the level of cohomology, an isomorphism of $\Q$-vector spaces $H^\bullet(C_g^n,\Q) \cong \gr_L H^\bullet(C_g^n,\Q)$. But this is not an isomorphism of algebras: the multiplication in $\gr_L H^\bullet(C_g^n,\Q)$ is defined by using only the maps $h^i(C_g/M_g) \otimes h^j(C_g/M_g) \to h^k(C_g/M_g)$ for $i+j=k$, and discarding all other parts of the cup product $h(C_g/M_g)^{\otimes 2} \to h(C_g/M_g)$.

\subsection{Decomposition into representations of the symplectic group} \label{explicitbraueraction}Let $J_g \to M_g$ be the universal jacobian. By Proposition \ref{curvejacobianisomorphism} there is an isomorphism of Chow motives over $M_g$:
$$ h^1(C_g/M_g) \cong h^1(J_g/M_g).$$
It follows that Theorem \ref{ancona} gives us a decomposition
$$ h^1(C_g/M_g)^{\otimes n} \cong \bigoplus_{\substack{\vert \lambda \vert \leq n \\ \vert \lambda \vert \equiv n \pmod 2}} h^1(C_g/M_g)_{\llambda} \otimes \L^{n-\vert \lambda \vert} \otimes \beta_{\lambda,n}^\ast.$$
We denote the motive $h^1(C_g/M_g)_{\llambda}$ by $\VV_\llambda$.  We often write $\VV$ for the motive $\VV_{\langle 1 \rangle}$. 

We also denote by $\VV^{\langle n \rangle}$ the summand of $\VV^{\otimes n}$ given by $\bigoplus_{\vert \lambda \vert=n} \VV_\llambda \otimes \sigma_{\lambda^T}^\vee$, and we refer to this as the \emph{primitive part} of $\VV^{\otimes n}$. 



One can make the action of the Brauer algebra $\B^{(-2g)}(n,n)$ on $\VV^{\otimes n}$ more explicit. 
Let $B$ be an $(n,n)$-Brauer diagram. Label the nodes in the Brauer diagram along the top row as $1,\ldots,n$ and along the bottom row as $n+1,\ldots,2n$. Write $(ij) \in B$ to denote that the $i$th and $j$th row are connected by a strand. Then 
$$ \prod_{(ij) \in B} p_{ij}^\ast(\pi_{1}) \in \CH^n(C_g^{2n})$$
is a well defined correspondence $C_g^n \vdash C_g^n$, where we consider $\pi_1$ as a cycle in $\CH^1(C_g^2)$ and $p_{ij}$ denotes the projection onto the $i$th and $j$th factor. This correspondence gives a map $h(C_g^n/M_g) \to h(C_g^n/M_g)$ that preserves the summand $h^1(C_g/M_g)^{\otimes n}$, and we obtain a well defined action of $\B^{(-2g)}(n,n)$ on $h^1(C_g/M_g)^{\otimes n}$. This action agrees with the one defined in Proposition \ref{braueraction}. 

On the level of Chow groups, we can also describe the action of a Brauer diagram as follows. An $(n,n-2)$-Brauer diagram in which the $i$th and $j$ dot on the top row are connected by a strand, and all others are vertical, gives rise to the following morphism of Chow groups:
$$ \CH^\bullet(M_g,\VV^{\otimes n}) \hookrightarrow \CH^\bullet(C_g^n) \stackrel{\Delta_{ij}^\ast}\longrightarrow \CH^\bullet(C_g^{n-1}) \stackrel {p_\ast}\longrightarrow \CH^{\bullet-1}(C_g^{n-2}), $$
where $p$ denotes the projection that forgets the marked point corresponding to the diagonal $\Delta_{ij}$. The image of the composition of these morphisms actually lands inside the summand $\CH^{\bullet-1}(M_g,\VV^{\otimes n-2}) \subset \CH^{\bullet-1}(C_g^{n-2})$. Similarly, an $(n-2,n)$-Brauer diagram in which the $i$th and $j$th dot on the bottom row are connected by a strand gives rise to the following morphism of Chow groups:
$$ \CH^\bullet(M_g,\VV^{\otimes (n-2)}) \hookrightarrow \CH^\bullet(C_g^{n-2}) \stackrel{p^\ast}{\longrightarrow} \CH^\bullet(C_g^{n-1}) \stackrel{(\Delta_{ij})_\ast}{\longrightarrow} \CH^{\bullet+1}(C_g^n) \stackrel{\pi_1^{\times n}}{\longrightarrow} \CH^{\bullet+1}(M_g,\VV^{\otimes n}).$$

\subsection{Decomposing the tautological ring} We have explained that $h(C_g^n/M_g)$ is a direct sum of terms of the form $h^0(C_g/M_g)^{\otimes n_0} \otimes h^1(C_g/M_g)^{\otimes n_1} \otimes h^2(C_g/M_g)^{\otimes n_2}$, with $n_0+n_1+n_2=n$. Since we also have $h^0(C_g/M_g) \cong \mathbf 1$ and $h^2(C_g/M_g) \cong \L$, and $h^1(C_g/M_g)^{\otimes n_1}$ is a direct sum of terms of the form $\VV_\llambda \otimes \L^{n_1 - \vert \lambda \vert}$,
we conclude that $h(C_g^n/M_g)$ is a direct sum of motives of the form $\VV_\llambda$ and their Tate twists.

\begin{thm} \label{decompositionwelldefined} Let $n$ be arbitrary, and consider $C_g^n \to M_g$. Choose any decomposition 
	$$ h(C_g^n/M_g) \cong \bigoplus_i \VV_{\langle \lambda_i \rangle} \otimes \L^{m_i}$$
	in $\Mot_{M_g}$. Then under the equality $\CH^k(C_g^n) \cong \CH^k(M_g,h(C_g^n/M_g))$ we have the following compatibility:
	\[\begin{tikzcd}[sep=small]
	\CH^k(C_g^n) \arrow[Isom]{r}& \bigoplus_i \CH^{k-m_i}(M_g,\VV_{\langle \lambda_i \rangle}) \\
	R^k(C_g^n)\arrow[Subseteq]{u} \arrow[Isom]{r} & \bigoplus_i R^{k-m_i}(M_g,\VV_{\langle \lambda_i \rangle}), \arrow[Subseteq]{u}
	\end{tikzcd}
	\]
	where both horizontal arrows are induced by our choice of decomposition of $ h(C_g^n/M_g) $.
\end{thm}

\begin{proof}
	Consider a summand $\VV_{\llambda} \otimes \L^{m_1}$ of $h(C_g^{n_1}/M_g)$ and a summand  $\VV_{\llambda} \otimes \L^{m_2}$ of $h(C_g^{n_2}/M_g)$. Then there exists a correspondence $\pi \in \CH(C_g^{n_1+n_2})$ --- not a correspondence of degree $0$, in general --- which maps the first summand isomorphically onto the second, considered as a correspondence $C_g^{n_1} \vdash C_g^{n_2}$. Moreover, $\pi$ can be built out of the projectors onto the K\"unneth components of $h(C_g^n/M_g)$ and the correspondences given by Brauer diagrams. As such, $\pi$ is actually a tautological class. 
	
	Now $\CH^\bullet(M_g,\VV_{\llambda} \otimes \L^{m_1})$ is a summand of $\CH^\bullet(C_g^{n_1})$ and as such there is a well defined subspace of tautological classes inside it, which we denote $R^\bullet(M_g,\VV_{\llambda} \otimes \L^{m_1})$. Similarly for the other summand. Now the correspondence $\pi$ which gives the isomorphism between the two summands is a tautological class, and in particular it maps tautological classes to tautological classes and gives an isomorphism (not preserving the grading unless $m_1=m_2$), $R^\bullet(M_g,\VV_{\llambda} \otimes \L^{m_1}) \cong R^\bullet(M_g,\VV_{\llambda} \otimes \L^{m_2})$.
	
	Finally since the projectors onto each summand $\VV_{\langle \lambda_i\rangle} \otimes \L^{m_i}$ of $h(C_g^n/M_g)$ are all given by tautological classes, the tautological ring $R^\bullet(C_g^n)$ is the direct sum of its projections onto each of the summands in the decomposition of $\CH^\bullet(C_g^n)$. 
\end{proof}

\subsection{Curves with rational tails}\label{rationaltails} Let $X$ be a smooth projective variety, and $X[n]$ the Fulton--MacPherson compactification of the configuration space of $n$ distinct ordered points on $X$ \cite{fmcompactification}. A result of Li \cite{limotives} expresses the Chow motive $h(X[n])$ as a direct sum of Chow motives of the form $h(X)^{\otimes i} \otimes \L^j$ for $0 \leq i \leq n$; this can be seen by an inductive argument from the blow-up formula and the construction of $X[n]$ as an iterated blow-up of the cartesian product $X^n$.

The analogous statement remains true (with the same proof) for a family $X \to S$ of smooth projective varieties, and the relative Chow motive of the relative Fulton--MacPherson compactification. In particular we may consider the universal family $C_g \to M_g$ over the moduli space of curves. In this case the relative configuration space of $n$ distinct ordered points is the space $M_{g,n}$, and the relative Fulton--MacPherson compactification of $M_{g,n}$ is the moduli space $M_{g,n}^{\rt}$ of curves with \emph{rational tails}, which is an iterated blow-up of $C_g^n$. By ``relative compactification'' we mean that the map $M_{g,n}^\rt \to M_g$ is proper. 

It follows from the above considerations that the results of this section remain valid when $C_g^n$ is replaced with $M_{g,n}^\rt$. There will in particular exist a decomposition of Chow motives
$$ h(M_{g,n}^\rt/M_g) \cong \bigoplus_i \VV_{\langle \lambda_i \rangle} \otimes \L^{m_i},$$
and moreover, there is a {canonical} choice of such decomposition in which each correspondence projecting onto a summand is a tautological class (cf.\ \cite[Theorem 3.2]{limotives}). It follows in particular that 
$$ R^\bullet(M_{g,n}^\rt) \cong \bigoplus_i R^{\bullet-m_i}(M_g,\VV_{\langle\lambda_i\rangle}).$$
The results of \cite{limotives} give explicit formulas expressing the motive $h(M_{g,n}^{\rt}/M_g)$ in terms of motives $\VV^{\otimes i} \otimes \L^j$ and hence in terms of Tate twists of motives $\VV_\llambda$. For an $\sym_n$-equivariant version (formulated in that paper only in terms of cohomology) see \cite{mixedhodge}.

\subsection{Tautological cohomology groups}
For a partition $\lambda$ we let $\V_\llambda$ be the $\Q$-local system on $M_g$ defined by the representation of $\Sp(2g)$ of highest weight $\lambda$. Then the Chow motive $\VV_\llambda$ has as its Betti realization $\V_\llambda[-\vert\lambda\vert]$, i.e.\ the local system $\V_\lambda$, considered as a complex concentrated in cohomological degree $\vert \lambda \vert$, and there is a cycle class map
$$ \CH^k(M_g,\VV_\llambda) \to H^{2k-\vert\lambda\vert}(M_g,\V_\llambda).$$
We denote by $\RH^\bullet(M_g,\V_\llambda)$ the image of $R^\bullet(M_g,\VV_\llambda)$ under the cycle class map. For $\lambda = 0$ we get the usual tautological cohomology groups of $M_g$. A folklore conjecture says that any homological equivalence between tautological classes is a rational equivalence, which would imply that
$$ R^k(M_g,\VV_\llambda) \to \RH^{2k-\vert\lambda\vert}(M_g,\V_\llambda)$$
is always an isomorphism.

We caution the reader that since the Betti realization of $\VV$ is not the local system $\V$, but $\V[-1]$, there are many opportunities to get confused about the Koszul sign rule when comparing results in cohomology and in Chow. 

All the results of this section are valid mutatis mutandis also on the level of cohomology. One can either see this formally by applying the Betti realization functor or by repeating the proofs in the cohomological setting. For example, let us state the cohomological version of Theorem \ref{decompositionwelldefined}: 

\begin{thm} Let $n$ be arbitrary, and consider $f \colon C_g^n \to M_g$. Choose a decomposition 
	$$Rf_\ast\Q \cong \bigoplus_i \V_{\langle \lambda_i \rangle} [-m_i]$$
	in $D^b_c(M_g)$. Under the equality $H^k(C_g^n,\Q) \cong \mathbb H^k(M_g,Rf_\ast\Q)$ we have the following compatibility:
	\[\begin{tikzcd}[sep=small]
	H^k(C_g^n,\Q) \arrow[Isom]{r}& \bigoplus_i H^{k-m_i}(M_g,\V_{\langle \lambda_i \rangle}) \\
	\RH^k(C_g^n)\arrow[Subseteq]{u} \arrow[Isom]{r} & \bigoplus_i \RH^{k-m_i}(M_g,\V_{\langle \lambda_i \rangle}), \arrow[Subseteq]{u}
	\end{tikzcd}
	\]
	where both horizontal arrows are induced by the choice of decomposition of $ Rf_\ast\Q $.
\end{thm}

\section{Examples}\label{examplesection}

\subsection{Example: $\psi_1^n$.}\label{example1}Let us consider the class $\psi_1^n \in \CH^{n}(C_g^1)$. Its image under the projectors $\pi_0$, $\pi_1$ and $\pi_2$ is given by
\begin{align*}
\pi_0 \circ \psi_1^n &= (p_2)_\ast(\frac 1{2g-2}\psi_1^{n+1} - \frac{1}{2(2g-2)^2} \kappa_1 \psi_1^n ) = \frac{1}{2g-2} \kappa_{n} - \frac{1}{2(2g-2)^2} \kappa_1 \kappa_{n-1}, \\
\pi_1 \circ \psi_1^n &= (p_2)_\ast(\Delta_{12}\psi_1^n -\frac{1}{2g-2} \psi_1^n\psi_2 -\frac{1}{2g-2} \psi_1^{n+1} + \frac{1}{(2g-2)^2} \kappa_1 \psi_1^n )\\
&= \psi_1^n - \frac{1}{2g-2} \kappa_{n-1}\psi_1 - \frac{1}{2g-2} \kappa_n + \frac{1}{(2g-2)^2} \kappa_1\kappa_{n-1},\\
\pi_2 \circ \psi_1^n & = (p_2)_\ast(\frac 1{2g-2}\psi_1^n\psi_2 - \frac{1}{2(2g-2)^2} \kappa_1 \psi_1^n ) = \frac{1}{2g-2} \kappa_{n-1}\psi_1 - \frac{1}{2(2g-2)^2} \kappa_1 \kappa_{n-1}.
\end{align*} 
Here $p_2 \colon C_g^2 \to C_g^1$ forgets the \emph{first} marked point. These are thus the projections of the class $\psi_1^n \in \CH^n(C_g)$ into the three summands $\CH^{n}(M_g,h^0(C_g/M_g))$, $\CH^{n}(M_g,h^1(C_g/M_g))$ and $\CH^{n}(M_g,h^2(C_g/M_g))$, respectively. We can make some simple observations/sanity checks:
 \begin{itemize}
 	\item The three classes sum to $\psi_1^n$.
 	\item The class in $\CH^{n}(M_g,h^0(C_g/M_g))$ is pulled back from $\CH^{n}(M_g)$.
 	\item When $n=1$, the class in $\CH^1(M_g,h^2(C_g/M_g))$ restricts to a cycle of degree $1$ in $\CH^1$ of a fiber of $\lambda \colon C_g \to M_g$.
 	\item The classes $\pi_0 \circ \psi_1^n$ and $\pi_1 \circ \psi_1^n$ push forward to zero under $\lambda \colon C_g \to M_g$. The class $\pi_2 \circ \psi_1^n$ has the same pushforward as $\psi_1^n$. 	
 \end{itemize}
We note that $\pi_1 \circ \psi_1^n$ vanishes for $n= 0$ and $n= 1$, using that $\kappa_0 = 2g-2$. This could also have been seen from the fact that
\begin{align*}
 \CH^k(C_g) &\cong \CH^k(M_g,\mathbf 1) \oplus \CH^k(M_g,\VV_{\langle 1\rangle}) \oplus \CH^k(M_g,\L) \\
 &= \CH^k(M_g) \oplus \CH^k(M_g,\VV) \oplus \CH^{k-1}(M_g)
\end{align*}
which (using Harer's calculation of the Picard groups of $M_g$ and $C_g$ \cite{harersecondhomology}) implies that $\CH^0(M_g,\VV) = \CH^1(M_g,\VV) = 0$. The first of these classes which can be nontrivial is thus $\pi_1 \circ \psi_1^2 \in \CH^2(M_g,\VV)$. This class vanishes for $g \leq 4$, but is nonzero if $g \geq 5$. 

To prove that the class is nonzero for $g \geq 5$ it is easiest to work in cohomology. For large $g$, nontriviality is a consequence of Harer stability and the Mumford conjecture \cite{madsenweiss}: as $g \to \infty$, $H^\bullet(C_g^1)$ stabilizes to a polynomial ring in the $\kappa$-classes and $\psi_1$ \cite[Proposition 2.1]{looijengastable}. More precisely, $H^4(C_g^1)$ is in the stable range when $g \geq 7$, so there are no relations between the generators in this degree and $\pi_1\circ \psi_1^2$ must be nonzero. Nontriviality for $g=5,6$ can be checked e.g.\ by multiplying the class with $\psi_1^2$ and pushing it down to $M_g$. One computes that
$$ \lambda_\ast (\psi_1^2  \cdot (\pi_1 \circ \psi_1^2)) = \kappa_3 - \frac{2}{2g-2} \kappa_2\kappa_1 + \frac 1 {(2g-2)^2} \kappa_1^3.$$
This class can then be multiplied with $\kappa_1^{g-5}$ to get a class in the socle of the tautological ring, which can be verified to be nonzero by integrating it against $\lambda_g\lambda_{g-1}$. (We discuss the $\lambda_g\lambda_{g-1}$-pairing more in Section \ref{Gorsection}.)

The vanishing for $g \leq 4$ can be proven by standard methods for computing tautological rings and a dimension count. Let us consider the case $g=4$: in this case one only needs to know that $R^1(M_4) \cong R^2(M_4) \cong \Q$ and that $R^2(C_4^1) \cong \Q^2$. Now we have the relative K\"unneth decomposition
$$ R^2(C_4) = R^2(M_4) \oplus R^2(M_4,\VV) \oplus R^1(M_4),$$
where the first and last terms are one-dimensional since $R^1(M_4) \cong R^2(M_4) \cong \Q$; consequently, the middle term has to vanish since $R^2(C_4)$ is two-dimensional. But $\pi_1 \circ \psi_1^2$ is an element of $R^2(M_4,\VV)$, so it must vanish.  


\subsection{Example: the diagonal.}Let us decompose the class $\Delta_{12} \in \CH^1(C_g^2)$ into summands. One finds exactly three nonzero terms:
\begin{align*}
(\pi_2 \times \pi_0) \circ \Delta_{12} &= \frac{1}{2g-2}\psi_1 - \frac{1}{2(2g-2)^2}\kappa_1 &&\in \CH^1(M_g,\L \otimes \mathbf 1), \\
(\pi_1 \times \pi_1) \circ \Delta_{12} &= \Delta_{12} - \frac{1}{2g-2}(\psi_1+\psi_2) + \frac{1}{(2g-2)^2}\kappa_1 &&\in \CH^1(M_g,\VV \otimes \VV),\\
(\pi_0 \times \pi_2) \circ \Delta_{12} &= \frac{1}{2g-2}\psi_2 - \frac{1}{2(2g-2)^2}\kappa_1 &&\in \CH^1(M_g,\mathbf 1 \otimes \L),
\end{align*}
where $h^0(C_g/M_g) = \mathbf 1$, $h^1(C_g/M_g) = \VV$, $h^2(C_g/M_g) = \L$. Thus the terms in the decomposition of $\Delta_{12}$ are given exactly by the projectors $\pi_i$ themselves, considered as classes on $C_g^2$. 

Note that $\VV \otimes \VV \cong \VV_{\langle 2\rangle} \oplus \VV_{\langle 1,1\rangle} \oplus \L$. In fact we have $\CH^1(M_g,\VV_{\langle 2\rangle}) = \CH^1(M_g,\VV_{\langle 1,1\rangle}) = 0$, and the class $\pi_1$ is a generator for the summand $\CH^1(M_g,\L) \cong \CH^0(M_g) \cong \Q$. 

The fact that $\CH^1(M_g,\VV_{\langle 2\rangle}) = \CH^1(M_g,\VV_{\langle 1,1\rangle}) = 0$ can be proven by a dimension count argument much like the one from the previous example, using $\CH^1(C_g^2) = \Q\{\kappa_1,\Delta_{12},\psi_1,\psi_2\}$, which follows from Harer's calculation of the Picard group of $M_{g,n}$ \cite{harersecondhomology}. Now in the relative K\"unneth decomposition of $ \CH^1(C_g^2)$ we find the summands $\CH^1(M_g,\mathbf 1 \otimes \mathbf 1) \cong \Q\{\kappa_1\}$, $\CH^1(M_g,\mathbf 1 \otimes \L)$, $\CH^1(M_g,\L \otimes \mathbf 1)$ and $\CH^1(M_g,\L) \subset \CH^1(M_g,\VV \otimes \VV)$. All these last three terms are nonzero since $\CH^1(M_g,\L) \cong \CH^0(M_g) \cong \Q$. 

\subsection{Example: the Faber--Pandharipande cycle.} We consider the class $\Delta_{12}\psi_1 \in \CH^2(C_g^2)$. Applying the operator $\pi_1^{\times 2}$ gives the class\begin{align}\begin{split}\label{unrefinedFP}
\pi_1^{\times 2} \circ \Delta_{12}\psi_1 & =  \Delta_{12}\psi_1 - \frac{1}{2g-2}\psi_1\psi_2  - \frac 1 {2g-2}(\psi_1^2 +\psi_2^2)  + \frac 1{(2g-2)^2}\kappa_1(\psi_1+\psi_2) \\ & + \frac 1{(2g-2)^2}\kappa_2 -\frac 1{(2g-2)^3}\kappa_1^2 
\end{split}
\end{align} 
which now defines an element of $\CH^2(M_g,\VV^{\otimes 2})$. Since the class \eqref{unrefinedFP} is $\sym_2$-invariant, it is in fact a class in $\CH^2(M_g,\Sym^2\VV)$. We call this class the \emph{unrefined Faber--Pandharipande cycle}.

Now we have a decomposition $\Sym^2\VV = \VV_{\langle 1,1 \rangle} \oplus \L$. So \eqref{unrefinedFP} can be written as the sum of a class in $\CH^2(M_g,\VV_{\langle 1,1\rangle})$ and a class in $\CH^2(M_g,\L) \cong \CH^1(M_g)$. As explained in Subsection \ref{nazarov1}, Nazarov's theorem gives a general method to write down a projector acting on $\CH^\bullet(M_g,\VV^{\otimes n})$ whose image is a particular summand $\CH^\bullet(M_g,\VV_\llambda) \otimes \sigma_{\lambda^T}^\vee$, where $\lambda$ is a partition of $n$. Although Nazarov's theorem is overkill in this case, where one could quite easily figure out the right projector by hand, the result is that the correspondence
$$ \pi = \frac 1 2 (b_{13}b_{24} + b_{14}b_{23}) + \frac 1 {2g} b_{12}b_{34} \in \CH^2(C_g^4)$$
acts on $\CH^\bullet(M_g,\VV^{\otimes 2})$ with image $\CH^\bullet(M_g,\VV_{\langle 1,1\rangle})\otimes \sigma_2^\vee \cong \CH^\bullet(M_g,\VV_{\langle 1,1\rangle})$. Here $b_{ij}$ denotes the class $\Delta_{ij} - \frac{1}{2g-2}(\psi_i+\psi_j) + \frac 1 {(2g-2)^2}\kappa_1$, i.e.\ the pullback of the correspondence defining the projector $\pi_1$. 

Applying $\pi$ gives the class 
\begin{align}\begin{split}\label{refinedFP}
\Delta_{12}\psi_1 - \frac{1}{2g-2}\psi_1\psi_2  - \frac 1 {2g-2}(\psi_1^2 +\psi_2^2)  + \frac 1{(2g-2)^2}\kappa_1(\psi_1+\psi_2) + \frac 1{(2g-2)^2}\kappa_2 \\ -\frac 1{(2g-2)^3}\kappa_1^2 -  \frac{2g-1}{2g(2g-2)}\kappa_1(\Delta_{12}- \frac{1}{2g-2}(\psi_1+\psi_2) + \frac{1}{(2g-2)^2}\kappa_1)
\end{split}
\end{align} 
which now defines an element of $\CH^2(M_g,\VV_{\langle 1,1\rangle})$. We call \eqref{refinedFP} the \emph{refined Faber--Pandharipande cycle}. This is the projection of $\Delta_{12}\psi_1$ onto the summand $\CH^2(M_g,\VV_{\langle 1,1\rangle})$. Comparing the expressions for the refined and unrefined Faber--Pandharipande cycles, we see that we have subtracted the term
$$ \frac{2g-1}{2g(2g-2)}\kappa_1b_{12}.$$
This term is thus the projection of $\Delta_{12}\psi_1$ onto the summand $\CH^2(M_g,\L) \subset \CH^2(M_g,\Sym^2 \VV)$. We remark that since we know from the previous example that $b_{12}$ lies in the summand $\CH^1(M_g,\L) \subset \CH^1(M_g,\VV^{\otimes 2})$, it follows that $\kappa_1b_{12}$ indeed gives a class in $\CH^2(M_g,\L) \subset\CH^2(M_g,\VV^{\otimes 2})$.

Let $X$ be a curve of genus $g$. It is interesting to consider the image of the Faber--Pandharipande cycle in $\CH^2(X^2)$. In this case, removing terms which obviously vanish leaves
$$ \Delta_{12}\psi_1 - \frac{1}{2g-2}\psi_1\psi_2 \in \CH^2(X^2).$$
This simplified expression is what is more commonly referred to as the Faber--Pandharipande cycle. Green and Griffiths \cite{greengriffiths} proved that the Faber--Pandharipande cycle is nonzero in $\CH^2(X^2)$ for a very general curve $X$ of genus $\geq 4$ over the complex numbers. The third author gave a different proof of this result \cite{yinFP} valid also in positive characteristic. Such a result is rather subtle since the Faber--Pandharipande cycle is not only homologically trivial but Abel--Jacobi trivial.

Reasoning as in Example \ref{example1}, one can show that the refined Faber--Pandharipande cycle is zero for $g=2,3$ but that it is nonzero for all $g \geq 4$. Thus the refined Faber--Pandharipande cycle is nonzero precisely in those genera where it is nonzero in $\CH^2(X^2)$ for a generic curve $X$. By contrast the unrefined Faber--Pandharipande cycle is nonzero also for $g=3$, even though $R^2(M_3,\VV_{\langle 1,1\rangle})=0$. This illustrates the utility of Nazarov's refined projectors when trying to determine precisely which of these local systems have nonzero tautological groups.

\subsection{Example: the Gross--Schoen cycle.}\label{grossschoen} Let us consider the class $\Delta_{123} \in \CH^2(C_g^3)$. We can apply $\pi_1^{\times 3}$ to this class to get 
\begin{align*}
\pi_1^{\times 3} \circ \Delta_{123}  &= \Delta_{123} - \frac{1}{2g-2}((\Delta_{12}+\Delta_{23})\psi_1 + (\Delta_{13}+\Delta_{23})\psi_2 + (\Delta_{12}+\Delta_{13})\psi_3)  \\
& + \frac 1 {(2g-2)^2}(\psi_1^2 +\psi_2^2 + \psi_3^2)  + \frac 2 {(2g-2)^2}
(\psi_1\psi_2   + 
\psi_1\psi_3 + 
\psi_2\psi_3)\\
& + \frac 1 {(2g-2)^2}\kappa_1(\Delta_{12}+\Delta_{13}+\Delta_{23})- \frac 3 {(2g-2)^3}
\kappa_1(\psi_1+\psi_2+\psi_3)  \\&  - \frac 1 {(2g-2)^3}\kappa_2 + \frac 3{(2g-2)^4}\kappa_1^2.
\end{align*} 
We call this the \emph{Gross--Schoen cycle}. Being $\sym_3$-invariant, it defines a class in the summand $\CH^2(M_g,\Sym^3 \VV) \subset \CH^2(M_g,\VV^{\otimes 3})$. There is now a decomposition
$$ \Sym^3 \VV \cong \VV_{\langle 1,1,1\rangle} \oplus \VV_{\langle 1 \rangle} \otimes \L,$$ 
and one could try to define a ``refined'' Gross--Schoen cycle by projecting onto the summand $\CH^2(M_g,\VV_{\langle 1,1,1\rangle})$, just as we did for the Faber--Pandharipande cycle in the previous example. However, the difference between the refined and unrefined Gross--Schoen cycles would then be an element of $\CH^2(M_g,\VV_{\langle 1 \rangle} \otimes \L) = \CH^1(M_g,\VV)$, which always vanishes. By using Nazarov's theorem one can construct a ``refined projector'' onto $\CH^\bullet(M_g,\VV_{\langle 1,1,1\rangle})$; the image of $\Delta_{123}$ under this refined projector agrees with the image under the naive projector $\pi_1^{\times 3}$, which is a nontrivial consistency check.

The cycle originally considered by Gross and Schoen in \cite{grossschoen} is related to ours as follows. Let $X$ be a smooth curve and $\z$ a degree $1$ zero-cycle on $X$. Then they studied the cycle 
$$ \Delta_{123} - \z_1\Delta_{23} - \z_2\Delta_{13} - \z_3\Delta_{12} + \z_2\z_3 + \z_1\z_3 + \z_1\z_2 \in \CH^2(X^3),$$
where $\z_i$ denotes the pullback of $\z$ from the projection map onto the $i$th factor. Considering our cycle in $\CH^2(X^3)$ and removing terms which obviously vanish gives
\begin{gather*}
\Delta_{123} - \frac{1}{2g-2}((\Delta_{12}+\Delta_{23})\psi_1 + (\Delta_{13}+\Delta_{23})\psi_2 + (\Delta_{12}+\Delta_{13})\psi_3) \\
+ \frac 2 {(2g-2)^2}
(\psi_1\psi_2   + 
\psi_1\psi_3 + 
\psi_2\psi_3)
\end{gather*}
which does \emph{not} coincide with the usual Gross--Schoen cycle for $\z = \frac{1}{2g-2}\psi$. However, the difference between the two cycles is given by a sum of Faber--Pandharipande cycles. In particular, our Gross--Schoen cycle and the usual one will have the same image under the Abel--Jacobi map, since the Faber--Pandharipande cycle is Abel--Jacobi trivial.

The Gross--Schoen cycle defines a nontrivial class in $\CH^2(M_g,\VV_{\langle 1,1,1\rangle})$ for all $g \geq 3$. Nonvanishing of the Gross--Schoen cycle is equivalent to nonvanishing of the \emph{Ceresa cycle} \cite{ceresa}, which is known to be nonzero in $\CH^2(X^3)$ if $X$ is a very general curve of genus $g \geq 3$. See \cite{fakhruddin} for this result in positive characteristic.

\section{Consequences for Gorenstein conjectures}\label{Gorsection}

By a theorem of Looijenga \cite{looijengatautological}, it is known that 
$$ R^{g-2+n}(C_g^n) \cong \Q$$
and that the tautological ring vanishes above this degree. More precisely, Looijenga proved the vanishing and that this group is at most $1$-dimensional, and Faber \cite{fabernonvanishing} found an example of a nonzero tautological class in this degree. The top nonzero degree of the tautological ring is called the \emph{socle}. 

This isomorphism can be made explicit in the following way. We define a map $$ R^{g-2}(M_g) \to \Q$$
by 
$$ \alpha \mapsto \int_{\overline M_g} \overline \alpha \cdot \lambda_g \lambda_{g-1 },$$
where $\overline \alpha$ denotes the closure of an algebraic cycle representing the class $\alpha$, and $\lambda_i$ denotes the $i$th Chern class of the Hodge bundle. We recall that the Hodge bundle is the locally free sheaf of rank $g$ whose fiber at a moduli point $[C]$ is the space of holomorphic differentials on $C$. A priori the integral would seem to not be well defined, since it depends on the choice of an algebraic cycle representing $\alpha$, but the integral is in fact well defined since multiplication by $\lambda_g \lambda_{g-1 }$ kills everything supported on the Deligne--Mumford boundary. For $n >0$, one has an isomorphism
$$ R^{g-2+n}(C_g^n) \to R^{g-2}(M_g)$$
given by pushforward (with inverse given by pullback and multiplication by $\frac{1}{(2g-2)^n}\psi_1\psi_2\cdots\psi_n$). 

All in all, this means that we have a pairing, which we denote by brackets:
\[
\begin{tikzcd}[row sep = 0pt]
	R^k(C_g^n) \otimes R^{g-2+n-k}(C_g^n)  \arrow[r]& \Q \\
	\alpha \otimes \beta\hspace{1.1cm} \arrow[r, mapsto] & \langle \alpha, \beta \rangle
	\end{tikzcd}\]
given by cup product, pushing down to $M_g$, and integrating against $\lambda_g \lambda_{g-1}$. Arbitrary integrals of tautological classes on $\MM_{g,n}$ can be calculated algorithmically and efficiently \cite{faberalgorithms,johnsontautologicalSAGE}, in particular integrals over $\MM_g$ of top degree classes in $R^{g-2}(M_g)$ paired with $\lambda_g\lambda_{g-1}$. 

\begin{lem}Let $\alpha \in R^k(C_g^n)$ and $\beta \in R^{g-2+n-k}(C_g^n)$, and let $i_1,\ldots, i_n \in \{0,1,2\}$. There is an equality
	\[ \big\langle (\pi_{i_1} \times \ldots\times \pi_{i_n}) \circ \alpha , \beta\, \big\rangle =  \big\langle \alpha , (\pi_{2-i_1} \times \ldots\times \pi_{2-i_n}) \circ \beta \,\big\rangle.\]
\end{lem}
\begin{proof}
	Consider $\alpha$ as a morphism $\mathbf 1 \to h(C_g^n/M_g) \otimes \L^{-k}$, and $\beta$ as a morphism $h(C_g^n/M_g)) \otimes \L^{-k} \to \L^{2-g}$. The composition $\beta \circ \alpha \in \Mot_{M_g}(\mathbf 1,\L^{2-g}) = \CH^{g-2}(M_g)$ is the product $\alpha \cdot \beta$ pushed forward to $M_g$. Now if $\pi$ is any correspondence $C_g^n \vdash C_g^n$ then 
	$$ \beta \circ (\pi \circ \alpha) = (\beta \circ \pi) \circ \alpha.$$
	But $\beta \circ \pi = \pi^t \circ \beta$ by Remark \ref{transposeremark}. Since $(\pi_{i_1} \times \ldots\times \pi_{i_n})^t = (\pi_{2-i_1} \times \ldots\times \pi_{2-i_n})$ we are done.
	%
\end{proof}

As explained in Subsection \ref{cupproduct}, the cup product in $\CH^\bullet(C_g^n)$ (and then also in $R^\bullet(C_g^n)$) is in general not so easily described in terms of the relative K\"unneth decomposition of these algebras. The next proposition shows, however, that the \emph{socle pairing} on $R^\bullet(C_g^n)$ takes a very simple form with respect to the K\"unneth decomposition.

\begin{prop}\label{samepairing}
	The Gram matrix describing the socle pairing in the algebra $R^\bullet(C_g^n)$ is block-diagonal with respect to the relative K\"unneth decomposition of $R^\bullet(C_g^n)$. More precisely, the summand
	$$ R^k(M_g,h^{i_1}(C_g/M_g) \otimes \ldots \otimes h^{i_n}(C_g/M_g)) \subset R^k(C_g^n)$$
	pairs to zero with all summands in complementary degree except for 
	$$ R^{g-2+n-k}(M_g,h^{2-i_1}(C_g/M_g)\otimes \ldots \otimes h^{2-i_n}(C_g/M_g)) \subset R^{g-2+n-k}(C_g^n).$$
\end{prop}

\begin{proof}
	Take $\alpha \in R^k(C_g^n)$. Suppose it lies in the summand $$R^{k}(M_g,h^{i_1}(C_g/M_g) \otimes \ldots \otimes h^{i_n}(C_g/M_g));$$
	equivalently, $(\pi_{i_1} \times \ldots\times \pi_{i_n}) \circ \alpha = \alpha$. For $\beta$ in complementary degree we have
	$$ \langle \alpha, \beta\rangle = \big\langle (\pi_{i_1} \times \ldots\times \pi_{i_n}) \circ \alpha, \beta\,\big \rangle = \big\langle \alpha , (\pi_{2-i_1} \times \ldots\times \pi_{2-i_n}) \circ \beta \,\big\rangle$$
	by the previous lemma. 	But if $\beta$ lies in any summand except for
	$$ R^{g-2+n-k}(M_g,h^{2-i_1}(C_g/M_g)\otimes \ldots \otimes h^{2-i_n}(C_g/M_g)),$$
	then $(\pi_{2-i_1} \times \ldots\times \pi_{2-i_n}) \circ \beta = 0$, hence $\langle \alpha,\beta \rangle = 0$, as claimed.
\end{proof}

\begin{rem}
	The previous proposition shows that the socle pairing in $R^\bullet(C_g^n)$ depends only on the cup product maps  $h^{i}(C_g/M_g) \otimes h^{2-i}(C_g/M_g) \to h^2(C_g/M_g) \cong \L$. Since for $i=0$ or $i=2$ these maps are given by the canonical isomorphisms $\mathbf 1 \otimes \L \to \L$ and $\L \otimes \mathbf 1 \to \L$, the socle pairing in fact only depends nontrivially on the maps $$R^k(M_g,h^1(C_g/M_g)^{\otimes m}) \otimes R^{g-2+m-k}(M_g,h^1(C_g/M_g)^{\otimes m}) \to R^{g-2+m}(M_g,\L^{\otimes m})\cong R^{g-2}(M_g).$$ 
\end{rem}

\begin{rem}
	On the level of Betti realizations, Proposition \ref{samepairing} says that even though the algebras $H^\bullet(C_g^n,\Q)$ and $\gr_L H^\bullet(C_g^n,\Q)$ (and then also $\RH^\bullet(C_g^n,\Q)$ and $\gr_L \RH^\bullet(C_g^n,\Q)$) have very different cup product, both algebras $\RH^\bullet(C_g^n,\Q)$ and $\gr_L \RH^\bullet(C_g^n,\Q)$ will have identical socle pairings. 
	
	In particular, the algebra $\RH^\bullet(C_g^n)$ is Gorenstein (i.e.\ satisfies Poincar\'e duality) if and only if the same holds for the algebra $\gr_L \RH^\bullet(C_g^n)$. That said, Proposition \ref{samepairing} is not actually needed to prove this last fact. Since the Leray filtration is compatible with cup product, we know a priori that the Gram matrix describing the socle pairing for the algebra $\RH^\bullet(C_g^n)$ is block-triangular, and that the diagonal blocks coincide with the Gram matrix for the socle pairing in the algebra $\gr_L \RH^\bullet(C_g^n)$. In particular both matrices have the same rank. In these terms, Proposition \ref{samepairing} says that even though the intersection pairing for $\RH^\bullet(C_g^n)$ is a priori only block-triangular, it turns out to actually be block-diagonal. 
\end{rem}

\begin{thm}\label{faberconj}Fix a genus $g$.	The following are equivalent:
	\begin{enumerate}
		\item All algebras $R^\bullet(C_g^n)$ for $n \geq 0$ are Gorenstein.
		\item For each partition $\lambda$, the pairing
		$$ R^k(M_g,\VV_\llambda) \otimes R^{g-2+\vert \lambda \vert-k}(M_g,\VV_\llambda) \to R^{g-2+ \vert \lambda \vert}(M_g,\L^{\vert \lambda \vert}) \cong R^{g-2}(M_g) \cong \Q$$
		is perfect.
	\end{enumerate}
	The pairing in (2) comes from the map of motives $\VV_\llambda \otimes \VV_\llambda \to \L^{\vert \lambda \vert}$ given by the fact that $\VV_\llambda$ is self-dual.
	
\end{thm}

\begin{proof}If we decompose $h(C_g^n/M_g)$ as a direct sum of motives $\VV_\llambda \otimes \L^m$, then the socle pairing in $R^\bullet(C_g^n)$ is the direct sum of the various pairings 
	$$ R^k(M_g,\VV_\llambda) \otimes R^{g-2+\vert\lambda\vert-k}(M_g,\VV_\llambda) \to R^{g-2}(M_g) \cong \Q.$$
	Thus the socle pairing in the  tautological ring of $C_g^n$ is perfect if and only if the same holds for the pairing for each of the motives $\VV_\llambda$. 
\end{proof}

A variant of the preceding theorem, with the same proof, is as follows:

\begin{thm}Fix a genus $g$.	The following are equivalent:
	\begin{enumerate}
		\item All algebras $R^\bullet(C_g^n)$ for $0 \leq n \leq N$ are Gorenstein.
		\item For each partition $\lambda$ with $\vert \lambda \vert \leq N$, the pairing
		$$ R^k(M_g,\VV_\llambda) \otimes R^{g-2+\vert\lambda\vert-k}(M_g,\VV_\llambda) \to R^{g-2}(M_g) \cong \Q$$
		is perfect.
	\end{enumerate}
\end{thm}

We also wish to mention the following result, which was proven by somewhat different arguments in \cite{universalcurverationaltails,tavakolconjecturalconnection}.

\begin{thm}
	The following statements are equivalent:\begin{enumerate}
		\item All algebras $R^\bullet(C_g^n)$ for $0 \leq n \leq N$ are Gorenstein.
		\item The algebra $R^\bullet(M_{g,N}^\rt)$ is Gorenstein.
	\end{enumerate}
\end{thm}

This shows that the Faber conjecture for the spaces $M_{g,n}^\rt$ can also be equivalently reformulated in terms of the motives $\VV_\llambda$. We saw in Section \ref{rationaltails} that the tautological groups of $M_{g,n}^\rt$ can be expressed in terms of the tautological groups of the motives $\VV_\llambda$; however, there does not appear to be an analogue of Theorem \ref{samepairing} for the spaces $M_{g,n}^\rt$: the socle pairing for $M_{g,n}^\rt$ is block upper triangular with respect to the natural decomposition, but not in general block diagonal.

\subsection{Symmetric powers}

We may also consider the \emph{symmetric powers} of the universal curve. 

\begin{defn}
	We define $C_g^{(n)}$ to be the $n$th symmetric power of the universal curve over $M_g$; that is, $C_g^{(n)} = C_g^n/\sym_n$. We define its tautological ring by $R^\bullet(C_g^{(n)}) = R^\bullet(C_g^n)^{\sym_n}$.
\end{defn}

\begin{lem}
	Suppose that the Chow groups $\CH^\bullet(C_g^n)$ are decomposed as a direct sum of summands $\CH^\bullet(M_g,\VV_\llambda)$. The only local systems occuring in the subspace $\CH^\bullet(C_g^{(n)}) \subseteq \CH^\bullet(C_g^n)$ are those of the form $\lambda = (1,1,1,\ldots)$, i.e. those that occur as summands of the symmetric powers of $\VV$. 
\end{lem}

\begin{proof}
	Consider first the summand $\CH^k(M_g,\VV^{\otimes n})$. The $\sym_n$-invariants in this subspace are $\CH^k(M_g,\Sym^n \VV)$, which proves the lemma in this case. In general for $n=n_0+n_1+n_2$, the summand
	$$ \CH^k(M_g,h^0(C_g/M_g)^{\otimes n_0} \otimes h^1(C_g/M_g)^{\otimes n_1 } \otimes h^2(C_g/M_g)^{\otimes n_2}),$$
	together with its conjugates under the action of $\sym_n$, can be written as the induced representation
	$$\mathrm{Ind}_{\sym_{n_0} \times \sym_{n_1} \times \sym_{n_2}}^{\sym_n} \,  \CH^k(M_g,\VV^{\otimes n_1} \otimes \L^{n_2}).$$
	In particular, the $\sym_n$-invariants in this induced representation are isomorphic to \[\CH^{k-n_2}(M_g,\VV^{\otimes n_1})^{\sym_{n_1}} = \CH^{k-n_2}(M_g,\Sym^{ n_1}\VV)\]
	by Frobenius reciprocity. 
\end{proof}

\begin{thm}
	Fix a genus $g$. The following are equivalent:
	\begin{enumerate}
		\item For all $n \geq 0$, $R^\bullet(C_g^{(n)})$ is a Gorenstein ring.
		\item For some $n \geq g$, $R^\bullet(C_g^{(n)})$ is a Gorenstein ring.
	\end{enumerate}
\end{thm}

\begin{proof}The ring
	$R^\bullet(C_g^{(n)})$ is Gorenstein if and only if all motives $\VV_\llambda$ occuring in the decomposition of $h(C^{(n)}_g/M_g)$ have the property that the pairing
	$$ R^k(M_g,\VV_\llambda) \otimes R^{g-2+\vert\lambda\vert-k}(M_g,\VV_\llambda) \to \Q$$
	is perfect. But the motives $\VV_\llambda$ occuring in the  decomposition of the $n$th symmetric power are exactly those with 	$\lambda = (1,1,\ldots)$, where $\vert \lambda \vert \leq n$, by the previous lemma. The result follows from the fact that the motive $\VV_{\langle 1,1,1,\ldots \rangle}$ is zero if the number of $1$'s is greater than $g$. 
\end{proof}

More generally, we can consider the ``partial symmetric powers'', i.e.\ the tautological rings of $C_g^{n+k}/\sym_n$, the $n$-fold symmetric power of the universal curve over $C_g^k$. For $k=1$ these rings were considered in \cite{yin}, where they were proven to be intimately related to the tautological ring of the universal jacobian variety over $C_g = M_{g,1}$. The previous theorem admits a variant for the partial symmetric powers as well.

\begin{thm}
	Fix a genus $g$. The following are equivalent:
	\begin{enumerate}
		\item For all $n \geq 0$, $R^\bullet(C_g^{n+k}/\sym_n)$ is a Gorenstein ring.
		\item For some $n \geq g+k$, $R^\bullet(C_g^{n+k}/\sym_n)$ is a Gorenstein ring.
	\end{enumerate}
	
\end{thm}
\begin{proof}The relative Chow motive of $C_g^{n+k}/\sym_n$ over $M_g$ is the tensor product $h(C_g^{(n)}/M_g) \otimes h(C_g^k/M_g)$. Since $h(C_g^{k}/M_g)$ only contains motives $\VV_\lambda$ with $\vert \lambda \vert \leq k$, and $h(C_g^{(n)}/M_g)$ only contains motives $\VV_\lambda$ with $\vert \lambda \vert \leq g$ (by the argument of the preceding proof), the result follows.
\end{proof}

\begin{rem}
	In \cite[Theorem 7.15]{yin} it is proven that if $R^\bullet(C_g^{n+1}/\sym_n)$ is Gorenstein for some $n \geq 2g-1$ then $R^\bullet(C_g^{n+1}/\sym_n)$ is Gorenstein for all $n \geq 0$. The proof uses the relationship with the tautological ring of the universal jacobian $J_g$ over $C_g$, and that $C_g^{n+1}/\sym_n$ is a projective bundle over $J_g$ for $n \geq 2g-1$. Thus the arguments here re-prove this result with a slightly better lower bound.
\end{rem}

\section{Twisted commutative algebras and tautological rings}
\label{tca-section}
In the next sections we will analyze the structure of the collection of tautological rings $R^\bullet(C_g^n)$ for fixed $g$ but varying $n$. When we consider the direct sum $\bigoplus_n R^\bullet(C_g^n)$ we obtain the structure of a \emph{twisted commutative algebra}. 

\begin{defn}
	A \emph{twisted associative algebra} is an $\N$-graded unital associative algebra (say over $\Q$)
	$$ A = \bigoplus_{n \geq 0} A(n)$$
	together with an action of the symmetric group $\sym_n$ on the summand $A(n)$, such that the multiplication
	$$ A(n) \otimes A(m) \to A(n+m)$$
	is equivariant for the action of $\sym_n \times \sym_m$ on both sides. We say that $A(n)$ is the \emph{arity $n$ component} of $A$. 
\end{defn}
\begin{defn}
	Let $A = \bigoplus_{n \geq 0} A(n)$ be a twisted associative algebra. We say that $A$ is a \emph{twisted commutative algebra} if the diagram
	\[\begin{tikzcd}
	A(n) \otimes A(m) \arrow[r] \arrow[d]&  A(n+m)\arrow[d]\\
	A(m) \otimes A(n) \arrow[r] & A(m+n)
	\end{tikzcd}\]
	commutes for all $n,m \geq 0$, where the horizontal maps are given by multiplication, the left vertical map swaps the two factors, and the right map is given by acting via the ``box permutation'' swapping the first $n$ and the last $m$ elements. 
\end{defn}

\begin{rem}
	Let us make three remarks concerning the definition.
	\begin{enumerate}
		\item In all our examples, we will have what should more properly be called a ``{twisted commutative graded algebra}'' --- each summand $A(n)$ is itself $\Z$-graded, and the multiplication respects the $\Z$-grading.
		\item The notion of a twisted commutative \emph{graded} algebra is potentially ambiguous: in the commutativity condition, should the Koszul sign rule be applied to the map $A(n) \otimes A(m) \to A(m) \otimes A(n)$ that swaps the two factors? In fact we will require both possible conventions in this paper: when working with Chow groups we do not impose the Koszul sign rule, but when working with cohomology groups we do impose it. This is because the Chow ring $\CH^\bullet(X)$ of an algebraic variety $X$ is commutative in the strict sense, whereas the cohomology ring $H^\bullet(X)$ is commutative in the graded sense. We will pass over this ambiguity in silence for the rest of the paper; this should not cause any confusion.
		\item There are many equivalent ways to axiomatize the notion of a twisted commutative algebra. Here is an alternative one: let $B$ be the symmetric monoidal category of finite sets and bijections, with monoidal structure given by disjoint union. A twisted commutative algebra is a lax symmetric monoidal functor $B \to \mathsf{Vect}_\Q$ (or to the category of graded $\Q$-vector spaces, with or without the Koszul sign rule). 
	\end{enumerate}
	For more on twisted commutative algebras see e.g.\  \cite{ginzburgschedler} or \cite[Chapitre 4]{joyalanalyticfunctors}.
\end{rem}

Our main example will be the following.	Fix a genus $g \geq 2$. The direct sum $\bigoplus_{n \geq 0}\CH^\bullet(C_g^n)$ is an example of  a twisted commutative algebra. The multiplication 
	$$ \CH^k(C_g^n) \otimes \CH^l(C_g^m) \to \CH^{k+l}(C_g^{n+m})$$
	is given by the cross product, as defined in \ref{crossproduct}. More generally, for any partition $\{1,\ldots,n\} = T \sqcup T'$ we have maps $\CH^k(C_g^T) \otimes \CH^l(C_g^{T'}) \to \CH^{k+l}(C_g^n)$. We will refer to maps of this form, too, as cross product maps; this should not cause any confusion. 

We now have the following proposition, which in a sense explains why we will find the notion of a twisted commutative algebra useful. We have defined maps $R^\bullet(C_g^n) \to R^{\bullet}(M_g,\VV^{\otimes n}) \to R^{\bullet}(M_g,\VV^{\langle n\rangle})$; recall that $\VV^{\langle n \rangle}$ denotes the ``primitive part'' of $\VV^{\otimes n}$ and was defined in Subsection \ref{explicitbraueraction}. These maps are not in any sense ring homomorphisms (in fact there is no ring structure on the latter two spaces).  Nevertheless these will define homomorphisms of \emph{twisted} commutative algebras, when we consider all $n$ simultaneously:

\begin{prop}\label{tca}
	Fix $g \geq 2$, and consider the following commutative diagram:
	\[\begin{tikzcd}[sep=small]
	& \bigoplus_{n \geq 0}\CH^\bullet(C_g^n) \arrow[r, two heads] & \bigoplus_{n \geq 0}\CH^{\bullet}(M_g,\VV^{\otimes n}) \arrow[r, two heads]  &  \bigoplus_{n \geq 0}\CH^{\bullet}(M_g,\VV^{\langle n\rangle}) \\
	\bigoplus_{n \geq 0}S_{n}^\bullet \arrow[dr, two heads] \arrow[ur]& & & \\
	&
	 \bigoplus_{n \geq 0}R^\bullet(C_g^n)\arrow[r, two heads]  \arrow[uu, hook]&  \bigoplus_{n \geq 0}R^{\bullet}(M_g,\VV^{\otimes n}) \arrow[r, two heads] \arrow[uu, hook]&  \bigoplus_{n \geq 0}R^{\bullet}(M_g,\VV^{\langle n\rangle})\arrow[uu, hook]
	\end{tikzcd}\]	Each entry in this diagram is a twisted commutative algebra, and the arrows in this diagram are morphisms of twisted commutative algebras.
\end{prop}

Let us remind the reader of the definitions needed to make sense of Proposition \ref{tca}. $S_{n}^\bullet$ was defined in Definition \ref{s-defn}; it is the graded polynomial algebra on classes $\psi_i$, $\Delta_{ij}$ and $\kappa_i$, modulo the geometrically obvious relations of Eq.\ \eqref{basicrelations}.

The maps $\CH^\bullet(C_g^n) \to \CH^{\bullet}(M_g,\VV^{\otimes n})$ are given by the projectors $\pi_1^{\times n}$. The same is true for the maps $R^\bullet(C_g^n) \to R^{\bullet}(M_g,\VV^{\otimes n})$.

The maps $\CH^{\bullet}(M_g,\VV^{\otimes n}) \to  \CH^{\bullet}(M_g,\VV^{\langle n\rangle})$ and the twisted commutative algebra structure on $\bigoplus_{n \geq 0 }\CH^{\bullet}(M_g,\VV^{\langle n\rangle})$ are both defined by the fact that $\VV^{\langle n\rangle}$ is in a canonical way a direct summand of $\VV^{\otimes n}$. As such, the natural projection $\VV^{\otimes n} \to \VV^{\langle n \rangle}$ defines the map $\CH^{\bullet}(M_g,\VV^{\otimes n}) \to  \CH^{\bullet}(M_g,\VV^{\langle n\rangle})$. The multiplication in the twisted commutative algebra $\bigoplus_{n \geq 0}\CH^{\bullet}(M_g,\VV^{\langle n\rangle})$ is defined by using the composition $$
\begin{tikzcd} \VV^{\langle n \rangle} \otimes \VV^{\langle m \rangle} \arrow[r, hook]& \VV^{\otimes n} \otimes \VV^{\otimes m} = \VV^{\otimes (n+m)} \arrow[r, two heads] & \VV^{\langle n+m \rangle}\end{tikzcd} $$
to define a product $\CH^{\bullet}(M_g,\VV^{\langle n \rangle}) \otimes \CH^{\bullet}(M_g,\VV^{\langle m\rangle}) \to \CH^{\bullet}(M_g,\VV^{\langle n+m\rangle})$. This is associative: the diagram
$$\begin{tikzcd}
\VV^{\langle n \rangle} \otimes \VV^{\langle m \rangle} \otimes \VV^{\langle k\rangle} \arrow[d]\arrow[r]& \VV^{\langle n+m \rangle}  \otimes \VV^{\langle k\rangle} \arrow[d]  \\
\VV^{\langle n \rangle} \otimes \VV^{\langle m +k \rangle} \arrow[r]& \VV^{\langle n+m+k \rangle}  
\end{tikzcd}$$commutes, since both compositions coincide with the map given by 
$$\begin{tikzcd} \VV^{\langle n \rangle} \otimes \VV^{\langle m \rangle} \otimes \VV^{\langle k \rangle} \arrow[r, hook]& \VV^{\otimes (n+m+k)} \arrow[r, two heads] & \VV^{\langle n+m+k \rangle}.\end{tikzcd}$$


\begin{proof}(of Proposition \ref{tca}.) That the map $\bigoplus_{n \geq 0 }S_{n}^\bullet \to \bigoplus_{n \geq 0 }\CH^\bullet(C_g^n)$ is a homomorphism of twisted commutative algebras is obvious. That the maps $\pi_1^{\times n} \colon \CH^\bullet(C_g^n) \to \CH^\bullet(M_g,\VV^{\otimes n})$ are homomorphisms with respect to the cross product is explained in \ref{crossproduct}. That the maps $\CH^\bullet(M_g,\VV^{\otimes n}) \to \CH^\bullet(M_g,\VV^{\langle n \rangle})$ are homomorphisms with respect to the cross product is also clear, since the multiplication in the twisted commutative algebra $\CH^\bullet(M_g,\VV^{\langle n \rangle})$ was defined by lifting elements to $\CH^\bullet(M_g,\VV^{\otimes n})$, and using the multiplication in the twisted commutative algebra ``upstairs'' to multiply. 
\end{proof}

\begin{defn}
	Let $S \to R_g \to R_g' \to R_g''$ be the four twisted commutative algebras linked by the chain of surjections
	$$ \bigoplus_{n \geq 0}S_{n}^\bullet \to \bigoplus_{n \geq 0}R^\bullet(C_g^n) \to \bigoplus_{n \geq 0}R^{\bullet}(M_g,\VV^{\otimes n}) \to \bigoplus_{n \geq 0}R^{\bullet}(M_g,\VV^{\langle n\rangle}).$$ \end{defn}

\begin{prop}\label{generatorsprop}
	The twisted commutative algebra $S$ is the free twisted commutative algebra generated by the elements $\kappa_d \in S^d_0$ for $d \geq 1$, $\psi_1^m \in S^m_1$ for $m \geq 0$, and $\Delta_{12\ldots n}\psi_1^m \in S_n^{n-1+m}$ for $m \geq 0$. 
\end{prop}

%

\begin{proof}It is straightforward that every monomial in $S_{n}^\bullet$ can be uniquely reduced modulo the relations of Eq. \eqref{basicrelations} to a product
	$$ \prod_{i=1}^m \kappa_{d_i} \cdot \prod_{j=1}^k \Delta_{P_j}\psi_{P_j}^{e_j} $$
	where $(d_1,\ldots,d_m) \in \Z_{> 0}^m$, $(e_1,\ldots,e_k) \in \Z_{\geq 0}^k$, $P_1,\ldots,P_k$ is some partition of the set $\{1,\ldots,n\}$ into nonempty blocks, and $\psi_{P_j}$ denotes $\psi_a$ for any $a \in P_j$ (this is also observed in \cite[Lemma 5]{jandachow}). For example, the monomial $\kappa_1^2\Delta_{13}\Delta_{14} \psi_3^2 \in S_{4}^\bullet$ would correspond to $(d_1,d_2) = (1,1)$, $P_1 =\{2\}$, $e_1=0$, $P_2 = \{1,3,4\}$, $e_2 = 2$. But such a product is exactly the same as a cross product of the generators for the twisted commutative algebra $S$.
%
\end{proof}

\begin{defn}
	For $n \geq 0$ and $r \geq n-1$ we put
	$$ D_{n,r} = \begin{cases}
	\kappa_{r} & n=0 \\
	\Delta_{12\ldots n} \psi_1^{1-n+r} & n \geq 1.
	\end{cases}$$
	Note that $D_{1,r} = \psi_1^r$, $D_{0,-1}=\kappa_{-1}=0$, and $D_{0,0} = \kappa_0=2g-2$. The previous proposition can be stated in a more compact form in terms of this notation: specifically, that the twisted commutative algebra $S$ is freely generated by $\sym_n$-invariant classes $D_{n,r}$ placed in arity $n$ and degree $r$, where $n=0$ and $r \geq 1$ or $n \geq 1$ and $r \geq n-1$. \end{defn}

The fact that the classes $D_{n,r}$ generate $S$ implies that their images generate the twisted commutative algebras $R_g$, $R_g'$ and $R_g''$, since the map from $S$ to these algebras is surjective.

\begin{prop}\label{kernel}Fix $g \geq 2$ and consider the surjections $R_g \to R_g' \to R_g''$. The kernels of these maps are ideals in the respective twisted commutative algebras. 
	\begin{enumerate}
		\item The kernel of $R_g \to R_g'$ is the ideal generated by $1 \in \CH^0(C_g^1)$ and $\psi_1 \in \CH^1(C_g^1)$.

		\item The kernel of $R_g \to R_g''$ is the ideal generated by $1 \in \CH^0(C_g^1)$, $\psi_1 \in \CH^1(C_g^1)$ and $\Delta_{12} \in \CH^1 (C_g^2)$. Equivalently, the kernel of $R_g' \to R_g''$ is the ideal generated by $\pi_1^{\times 2} \Delta_{12} = \Delta_{12} - \frac{1}{2g-2}(\psi_1+\psi_2) + \frac{1}{(2g-2)^2}\kappa_1$.
			\end{enumerate}
	
\end{prop}
%

\begin{proof}
	(1) The kernel of $\CH^k(C_g^n) \to \CH^{k}(M_g,\VV^{\otimes n})$ equals the image of the projectors $\pi_{i_1} \times \ldots \times \pi_{i_n}$ where $(i_1,i_2,\ldots,i_n) \neq (1,1,\ldots,1)$; equivalently, the image of all projectors $$\id \times \id \times \ldots \times \pi_i \times \ldots \times \id $$
	(i.e.\ all factors except one are given by the identity correspondence, the diagonal), where $i = 0, 2$. By $\sym_n$-symmetry, let's assume that all factors except the first are given by the identity. Let $\alpha \in \CH^\bullet(C_g^n)$. One checks that
	$$ (\pi_0 \times \id \times \ldots \times \id) \circ \alpha = 1 \times \alpha' - \frac{1}{2(2g-2)^2}\kappa_1 \times 1 \times \alpha''$$
	and
	$$ (\pi_2 \times \id \times \ldots \times \id) \circ \alpha = \psi_1 \times \alpha'' - \frac{1}{2(2g-2)^2}\kappa_1 \times 1 \times \alpha'',$$
	where $\alpha' = ({p}_{23\ldots n})_\ast(\psi_1 \cdot \alpha)$ and $\alpha'' = ({p}_{23\ldots n})_\ast( \alpha)$. Hence both projectors map all cycles $\alpha$ into the ideal generated by $\psi_1 \in \CH^1(C_g^1)$ and $1 \in \CH^1(C_g^1)$. Conversely, one checks that $\pi_1$ annihilates both $1$ and $\psi_1$. 
	
	(2) The map $\CH^{k}(M_g,\VV^{\otimes n}) \to \CH^{k}(M_g,\VV^{\langle n \rangle})$ is defined by the projection $\VV^{\otimes n} \to \VV^{\langle n \rangle}$, and the kernel of $\VV^{\otimes n} \to \VV^{\langle n \rangle}$ is spanned by the image of all $\binom n 2$ maps $\VV^{\otimes (n-2)} \otimes \L \to \VV^{\otimes n}$ given by $(n-2,n)$-Brauer diagrams of the form considered in the second half of Section \ref{explicitbraueraction}. But it is clear from the description in Section \ref{explicitbraueraction} that an element of $\CH^{k}(M_g,\VV^{\otimes n})$ is in the image of one of the maps $\CH^{k-1}(M_g,\VV^{\otimes (n-2)}) \to \CH^{k}(M_g,\VV^{\otimes n})$ precisely if it can be written in a nontrivial way as a cross product with $\pi_1^{\times 2}\Delta_{12}$.
%
%
%
\end{proof}

\begin{rem}
	Let us emphasize that the word ``ideal'' in the preceding proposition should be understood in the sense of twisted commutative algebras; that is, the smallest twisted commutative submodule containing the given elements. In particular, the ring structures of (say) the individual tautological rings $R^\bullet(C_g^n)$ are not what is important.
\end{rem}

\begin{cor}\label{irrelevantgenerators} The twisted commutative algebra $R_g''$ is generated by the images of the elements $D_{n,r}$ such that $n= 0$ and $r \geq 1$, or $n \geq 1$ and $r \geq \max(n-1,2)$.  
\end{cor}

\begin{proof}
	We have seen that $S$ is generated by the classes $D_{n,r}$ for $n=0$ and $r \geq 1$ or $n \geq 1$ and $r \geq n-1$. Since the generators $D_{1,0}$, $D_{1,1}$ and $D_{2,1}$ go to zero under $S \to R_g''$ by Proposition \ref{kernel}, we deduce that $R_g''$ is generated by the images of the remaining generators. 
\end{proof}

\begin{cor}
	The arity $n$ component $R_g''(n)$ vanishes in degrees below $\frac{2n} 3$.
\end{cor}

\begin{proof}Every generator $D_{n,r}$ fulfills this bound, since $\max(n-1,2) \geq \frac{2n}{3}$ for all natural numbers $n$ (with equality only for $n=3$). Since the bound is linear, and degrees and arities are both additive under cross product, the result follows.
\end{proof}

\begin{rem}\label{cohomologytca}The cohomology groups of the spaces $C_g^n$ also form twisted commutative algebras, and so do the cohomology groups of the local systems $\V^{\langle n \rangle}$ on $M_g$. In particular we have a chain of surjections of twisted commutative algebras in graded vector spaces:
$$\bigoplus_{n \geq 0}H^{\bullet}(C_g^n,\Q) \to \bigoplus_{n \geq 0}H^{\bullet-n}(M_g,\V^{\otimes n}) \to \bigoplus_{\lambda}H^{\bullet-n}(M_g,\V_\llambda)\otimes \sigma_{\lambda^T}^\vee.$$
If we consider $\bigoplus_{n \geq 0}\CH^\bullet(C_g^n)$, $ \bigoplus_{n \geq 0}\CH^\bullet(M_g,\VV^{\otimes n})$ and $ \bigoplus_{\lambda}\CH^\bullet(M_g,\VV_\llambda)\otimes \sigma_{\lambda^T}^\vee$ also as twisted commutative algebras in graded vector spaces, but with doubled degrees, then they map compatibly to the cohomological versions of these twisted commutative algebras under the cycle class map. We also get twisted commutative algebras of tautological classes
$$\bigoplus_{n \geq 0}\RH^{\bullet}(C_g^n,\Q) \to \bigoplus_{n \geq 0}\RH^{\bullet-n}(M_g,\V^{\otimes n}) \to \bigoplus_{\lambda}\RH^{\bullet-n}(M_g,\V_\llambda)\otimes \sigma_{\lambda^T}^\vee.$$
There is a natural ``suspension'' operation on twisted graded commutative algebras \cite[4.1]{spectralsequencestratification} which has the effect of shifting the grading on the arity $n$ component by $n$ and tensoring with the sign representation of $\sym_n$.  In this way one can get rid of both the annoying degree shift which appears in cohomology and the conjugate of the partition $\lambda$: one finds that there is a natural structure of twisted commutative algebra on
$$ \bigoplus_{\lambda} H^\bullet(M_g,\V_\lambda) \otimes \sigma_\lambda^\vee$$
with a subalgebra $ \bigoplus_{\lambda} \RH^\bullet(M_g,\V_\lambda) \otimes \sigma_\lambda^\vee$ of tautological classes.
\end{rem}

\section{A consequence of the FZ relations}\label{FZsection}

In this section we will recall the \emph{FZ relations} between tautological classes in $R^\bullet(C_g^n)$, and draw some simple consequences from them. In particular, we will prove an analogue of the following theorem, which was conjectured by Faber \cite{faberconjectures} and proved independently by Ionel \cite{ionel} and Morita \cite{moritagenerators}. (Morita's proof was only valid in cohomology, but Ionel's proof worked in Chow, too.)

\begin{thm}[Ionel, Morita] The tautological ring $R^\bullet(M_g)$ is generated by the classes $\kappa_r$ for which $3r< g+1$. 
\end{thm}

This theorem is a direct consequence of the FZ relations. We will see that the FZ relations can be used to prove the following stronger result:

\begin{thm}\label{fzconstraint}
	Fix a genus $g \geq 2$. The twisted commutative algebra $R_g = \bigoplus_{n \geq 0}R^\bullet(C_g^n)$ is generated by the classes $D_{n,r}$ for which 	$3r-n < g+1$.
\end{thm}

This implies in particular the result of Ionel--Morita, since the arity $0$ component of $R_g$ is the tautological ring $R^\bullet(M_g)$, and $D_{0,r}$ is the kappa class $\kappa_r$.

\subsection{The FZ relations} In the early 2000's, Faber and Zagier (in unpublished work) formulated a conjectural infinite family of relations in the tautological ring $R^\bullet(M_g)$. These relations were proven using the geometry of stable quotients by Pandharipande and Pixton \cite{pandharipandepixton}. Around the same time, Pixton found a generalization of this conjecture to incorporate also marked points and an extension of these relations to the Deligne--Mumford boundary. These extended FZ relations on $\MM_{g,n}$ were subsequently proven in cohomology by Pandharipande--Pixton--Zvonkine \cite{pandharipandepixtonzvonkine} and on the level of Chow rings by Janda \cite{jandachow,jandachow2}. 

The FZ relations on $C_g^n$ take a simpler form than on $\MM_{g,n}$. Let us state the result in this case, following \cite[Section 4]{jandachow}.

	Let 
	$$ A(z) = \sum_{i\geq 0} \frac{(6i)!}{(2i)!(3i)!}z^i \qquad B(z) = \sum_{i \geq 0} \frac{(6i)!}{(2i)!(3i)!} \frac{(6i+1)}{(6i-1)} z^i.$$
	We introduce a sequence of further power series $C_n$ by 
	\begin{equation}\label{diffeq1}
	C_1 = \frac{B}{A}; \qquad C_{n+1} = (12z^2 \frac{d}{dz} - 4nz) C_n.
	\end{equation}  
	We note that $C_n$ is a multiple of $z^{n-1}$. We will also define 
	$$ C_0 = \log(A),$$
	which is a multiple of $z^1$. Then we have
	\begin{equation}\label{diffeq2}
	C_{1} = -1 + 144z + 2^53^3z^2 \frac{d}{dz} C_0, 
	\end{equation}
	so that the coefficients $C_1$ (and hence also the higher $C_n$) are in fact recursively expressed in terms of those of $C_0$, except for low order terms.
	
	For any power series $F(z) = \sum_{i\geq 0} a_i z^i$ in $\Q[\![z]\!]$, we define bracket operators
	$$ \{F\}_\kappa = \sum_{i \geq 0} \kappa_i a_i z^i$$
	and 
	$$ \{F\}_{\Delta_S} = \sum_{i \geq 0} (-1)^{\vert S \vert - 1} \Delta_S \psi_S^{i - \vert S \vert + 1}a_i z^i$$
	for any $S \subseteq \{1,\ldots,n\}$; here $\psi_S$ denotes $\psi_j$ for any $j \in S$. 
	
	We use $[F]_{z^r}$ to denote the coefficient of $z^r$ in a power series.
	
	\begin{thm}[Janda, Pixton--Pandharipande, Pixton--Pandharipande--Zvonkine]
		For any $r$ such that $3r -g-1-n$ is a nonnegative even integer, the expression
		\[ \big[ \exp(-\{\log(A)\}_\kappa) \sum_{P \text{ \emph{partition of} } n} \prod_{S \in P} \{C_{\vert S\vert} \}_{\Delta_S}\big]_{z^r} \]
		vanishes in $\CH^r(C_g^n)$. 
	\end{thm}
	
	\begin{defn}
		We denote  the above expression $\big[ \exp(-\{\log(A)\}_\kappa) \sum_{P} \prod_{S \in P} \{C_{\vert S\vert} \}_{\Delta_S}\big]_{z^r}$ by $\FZ_{g,n,r}$.
	\end{defn}

	\begin{lem}[Ionel]\label{ionellemma}All coefficients $[C_n]_{z^r}$, for $n=0$ and $r \geq 1$ or $n \geq 1$ and $r \geq n-1$, are strictly positive rational numbers, except the constant term of $C_1$ which is negative.
	\end{lem}
	
	\begin{proof}The case $n=0$ is \cite[Lemma 3.6]{ionel}, since the coefficients of $C_0$ are (up to rescaling) the numbers she denotes $c_{k,k}$.  The case $n=1$ follows from this by the differential equation \eqref{diffeq2}; in fact, it is even stated in Ionel's lemma, since the coefficients of $C_1$ are (up to rescaling) the numbers she calls $c_{k,k-1}$. The differential equation \eqref{diffeq1} says that
		$$ [C_{n+1}]_{z^{r+1}} = (12r - 4n) [C_n]_{z^{r}}$$
		for $n \geq 1$, and one checks that $12r - 4n$ is strictly positive in all cases of interest except $n=1$, $r=0$, where it is negative: consequently, all coefficients of $C_n$ for $n \geq 2$, $r \geq n-1$ are strictly positive, too.
	\end{proof}

	We may now prove Theorem \ref{fzconstraint}.
	
	\begin{proof}(of Theorem \ref{fzconstraint}) We know that the twisted commutative algebra $R_g$ is generated by the images of the classes $D_{n,r} \in S_g$. Consider some generator $D_{n,r}$ for which $3r-g-1 -n \geq 0$. If $3r-g-1 -n$ is even then one of the terms in the relation $\FZ_{g,n,r}$ equals
		$$ (-1)^{n-1}[C_n]_{z^r} \cdot D_{n,r},$$ and all other terms are products of generators with smaller $r$. By Lemma \ref{ionellemma}, $[C_n]_{z^r}$ is nonzero, and this relation can be used to express the class $D_{n,r}$ in terms of ``simpler'' generators.

	If $3r-g-1-n$ is odd, we may instead consider the FZ relation $\psi_{n+1}\cdot \FZ_{g,n+1,r}$, and push forward along the map forgetting the last marked point to get a codimension $r$ relation on $C_g^n$. When we push down a monomial in the kappa, diagonal and psi-classes from $C_g^{n+1}$ to $C_g^n$, we get a multiple of $D_{n,r}$ exactly when we push down $D_{n+1,r+1} = \Delta_{12\ldots n+1}\psi_1^{r-n+1}$ (which pushes forward to $D_{n,r}$) and when we push down $D_{n,r} \times D_{1,1} = \Delta_{12\ldots n}\psi_1^{r-n+1}\psi_{n+1}$, which pushes forward to $(2g-2)D_{n,r}$. Thus the resulting relation on $C_g^n$ will have as one of its terms
	$$ (-1)^n\big([C_{n+1}]_{z^{r}} - (2g-2) [C_n]_{z^r}[C_1]_{z^0}\big) \cdot D_{n,r},$$
	and all other terms are products of generators with smaller values of $r$. By Lemma \ref{ionellemma} the coefficients $[C_{n+1}]_{z^{r}}$ and $[C_n]_{z^r}$ are positive and the coefficient $[C_1]_{z^0}$ is negative, so the coefficient behind $D_{n,r}$ is nonzero and we may use this relation to eliminate the generator $D_{n,r}$.
	\end{proof}
%
%
%
%
%
%

\section{Low genus calculations}\label{lowgenuscalculations}

In this section of the paper we will completely calculate the groups $R^k(M_g,\VV_\llambda)$ for all $k$ and $\lambda$ when $g \leq 4$.

\begin{thm}\label{lowgenustheorem}Recall the twisted commutative algebra $R_g'' = \bigoplus_{\lambda} R^\bullet(M_g,\VV_\llambda) \otimes \sigma_{\lambda^T}^\ast$, defined for any $g \geq 2$. 
	\begin{enumerate}
		\item The twisted commutative algebra $R_2''$ is trivial. Equivalently, $ R^k(M_2,\VV_\llambda) = 0$ unless $k=0$, $\lambda =0$, for which $R^0(M_2,\VV_{\langle0\rangle}) = R^0(M_2) \cong \Q$.
		\item The twisted commutative algebra $R_3''$ is generated by $\kappa_1$ and the Gross--Schoen cycle. We have
		$$R^0(M_3,\VV_{\langle0\rangle}) \cong R^1(M_3,\VV_{\langle0\rangle}) \cong R^2(M_3,\VV_{\langle111\rangle}) \cong \Q,$$ and all other tautological groups of all other motives $\VV_\llambda$ on $M_3$ vanish. The group $R^1(M_3,\VV_{\langle0\rangle})$ is spanned by $\kappa_1$ and the group $R^2(M_3,\VV_{\langle111\rangle}) $ is spanned by the Gross--Schoen cycle.
		\item  The twisted commutative algebra $R_4''$ is generated by $\kappa_1$, the Gross--Schoen cycle, and the Faber--Pandharipande cycle.
		The complete list of motives $\VV_\llambda$ on $M_4$ with nontrivial tautological groups are
		\begin{gather*}
		R^0(M_4,\VV_{\langle0\rangle}) \cong R^1(M_4,\VV_{\langle0\rangle}) \cong R^2(M_4,\VV_{\langle0\rangle}) \cong \Q, \\
		R^2(M_4,\VV_{\langle 111\rangle}) \cong R^3(M_4,\VV_{\langle 111\rangle}) \cong \Q, \\
		R^2(M_4,\VV_{\langle 11\rangle}) \cong \Q, \\
		R^4(M_4,\VV_{\langle 2211\rangle}) \cong \Q.
		\end{gather*} 
		The group $R^k(M_4,\VV_{\langle 0 \rangle})$ is spanned by $\kappa_1^k$. The group $R^2(M_4,\VV_{\langle 111\rangle})$ is spanned by the Gross--Schoen cycle, and $R^3(M_4,\VV_{\langle 111\rangle})$ by the product of $\kappa_1$ and the Gross--Schoen cycle. The group $R^2(M_4,\VV_{\langle 11\rangle})$ is spanned by the Faber--Pandharipande cycle. Finally, $R^4(M_4,\VV_{\langle 2211\rangle})$ is spanned by the cross product of two Gross--Schoen cycles; that is, the projection of $\Delta_{123}\Delta_{456}$ into the summand $\CH^4(M_4,\VV_{\langle 2,2,1,1\rangle}) \otimes \sigma_{4,2}$ gives a generator.
	\end{enumerate}
	In all cases, Poincar\'e duality holds, in the sense that
	$$ R^k(M_g,\VV_\llambda) \otimes R^{g-2+\vert \lambda \vert -k}(M_g,\VV_\llambda) \to R^{g-2}(M_g,\VV_{\langle0\rangle}) \cong \Q$$
	is a perfect pairing.
	 
\end{thm}

The proof of this theorem occupies the rest of this section. In all genera, the strategy of the proof will be the same:
\begin{itemize}
\item Using Corollary \ref{irrelevantgenerators} and Theorem \ref{fzconstraint} we get a finite list of generators for the twisted commutative algebra $R_g''$. Thus we have reduced the problem to finding the complete set of relations between these generators.
\item We use the FZ relations to obtain relations between the generators. Since we are working in the twisted commutative algebra $R_g''$, it is enough to consider the FZ relations modulo the equivalence relation $\equiv$, which often dramatically simplifies the relations. In this way we find that all but a finite list of twisted tautological classes are zero, and we are done if we can prove nonvanishing of each of these. 
\item Using Nazarov's Theorem \ref{nazarov} and our Theorem \ref{decompositionwelldefined}, we can represent each of the remaining potentially nonzero twisted tautological classes by an {explicit} class in $R^\bullet(C_g^n)$, so that we reduce the problem to proving that a finite number of tautological classes (without twisted coefficients) in $C_g^n$ are nonzero. This is now done by a standard argument: we multiply with some other class in complementary degree to land in the top degree part of the tautological ring, and then push down to get an element in the top degree of the tautological ring of $M_g$, whose structure we understand completely. 
\end{itemize}

To formulate the calculations, it will be convenient to introduce the following notation: for $x, y \in \CH^k(C_g^n)$, we write $x \equiv y$ to denote that $x$ and $y$ have the same image in $\CH^k(M_g,\VV^{\langle n \rangle})$. Equivalently by Theorem \ref{kernel}, $x \equiv y$ if $x$ and $y$ are equivalent modulo the twisted commutative algebra-ideal generated by $1 \in \CH^0(C_g^1)$, $\psi_1 \in \CH^1(C_g^1)$ and $\Delta_{12} \in \CH^1(C_g^2)$.

\begin{rem}We caution the reader that the relation $\equiv$ does \emph{not} respect the multiplication in the rings $R^\bullet(C_g^n)$, and is \emph{not} preserved when pushing forward a relation along a diagonal inclusion. That is, if we are given an element $\mathsf R \in S_n^k$ such that $\mathsf R \equiv 0$ in $R^k(C_g^n)$, then it does not follow e.g. that $\psi_1 \cdot \mathsf R \equiv 0$ in $R^{k+1}(C_g^{n})$, nor that the pushforward of $\mathsf R$ to $R^{k+1}(C_g^{n+1})$ along a diagonal inclusion vanishes modulo $\equiv$. One must therefore be careful to first multiply or push forward and only afterward reduce modulo $\equiv$. \end{rem}

\subsection{Genus two}

By Corollary \ref{irrelevantgenerators} and Proposition \ref{fzconstraint}, \emph{all} of the generators $D_{n,r}$ go to zero in $R_2''$. So $R_2''$ is the free twisted commutative algebra on \emph{no} generators, i.e. it contains only the unit element in arity $0$. This proves the genus $2$ case of Theorem \ref{lowgenustheorem}. 

This result was previously obtained (in a different form) in \cite{tavakol2}. 

The analogous statement is also true in genus one: $R^k(M_{1,1},\VV_{\langle a \rangle}) \cong \Q$ for $k=a=0$, and vanishes otherwise; this reformulates a result from \cite{tavakol1}. This statement is not hard to prove in our framework, but we have chosen to simplify the exposition by only talking about tautological groups $R^\bullet(M_g,\VV_\llambda)$ on moduli spaces of unpointed curves. 

\subsection{Genus three} By Corollary \ref{irrelevantgenerators} and Proposition \ref{fzconstraint}, the twisted commutative algebra $R_3''$ is generated by the images of $D_{0,1}$ and $D_{3,0}$, i.e. $\kappa_1$ and the Gross-Schoen cycle. Thus $R_3''$ is completely determined if we can find the complete set of relations between these generators. We claim that the product of any two generators vanishes. We have three products we need to check are zero:

\begin{enumerate}
	\item The relation $\kappa_1^2 = 0$ in $R^2(M_3,\VV_{\langle 0 \rangle})$ is well known and is a very special case of Looijenga's theorem (see Section \ref{Gorsection}). 
	\item Modulo the equivalence relation $\equiv$, the relation $\FZ_{3,3,3}$ simplifies to
	$$ 18432 \Delta_{123}\psi_1  - 960 \Delta_{123}\kappa_1 \equiv 0.$$
	That is, this is the expression obtained from $\FZ_{3,3,3}$ by removing all terms which are divisible (in the twisted commutative algebra $R_3$) by $1 \in R^0(C_g^1)$, $\psi_1 \in R^1(C_g^1)$ and $\Delta_{12} \in R^1(C_g^2)$. Note that this is the relation we used to show that the generator $\Delta_{123}\psi_1$ can be expressed in terms of simpler generators in the proof of Theorem \ref{fzconstraint}.
	
	Now consider instead the pushforward of $\FZ_{3,2,2}$ along a diagonal inclusion $C_g^2 \hookrightarrow C_g^3$. Modulo $\equiv$, that relation simplifies to
	$$ -1152 \Delta_{123}\psi_1 + 240 \Delta_{123}\kappa_1 \equiv 0.$$
	It is now clear that we obtain $\Delta_{123}\kappa_1 \equiv 0$, so that the product of the Gross--Schoen cycle and $\kappa_1$ vanishes in $R_g''$.

	\item Observe first of all that modulo the relation $\equiv$, the only nonzero monomials in $S_6^4$ (see Definition \ref{s-defn}) are $\Delta_{123}\Delta_{456}$ and its $\sym_6$-conjugates.
	
	For distinct elements $i,j \in \{1,...,6\}$, consider the pushforward of the relation $\FZ_{3,5,3}$ along the corresponding diagonal inclusion. The observation just made, and the fact that $\FZ_{3,5,3}$ is $\sym_5$-invariant, implies that the resulting relation  takes the form
	$$ \sum_{\substack{S \sqcup T = \{1,\ldots,6\} \\ \vert S \vert = \vert T \vert = 3 \\ i,j \in S}} \Delta_S \Delta_T \equiv 0$$
	(up to a nonzero constant), as all other terms in the pushforward of $\FZ_{3,5,3}$ vanish modulo $\equiv$. We think of these relations as $\binom 6 2 = 15$ equations in $\frac 1 2 \binom 6 3 = 10$ unknowns $\Delta_S \Delta_T$. It is a simple matter of linear algebra to check that the matrix of equations has full rank, so that  $\Delta_S \Delta_T \equiv 0$ for all $S$, $T$.
\end{enumerate}

We should also verify that $\kappa_1$ and the Gross--Schoen cycle are both nonzero in genus three, and that Poincar\'e duality holds. Nonvanishing of $\kappa_1$ is well known. The square of the Gross--Schoen cycle, pushed down to $M_3$, equals $\frac 7 4 \kappa_1$ (as one can verify on a computer). This proves both nonvanishing of the Gross--Schoen cycle and Poincar\'e duality, since the pairing $R^2(M_3,\VV_{\langle 1,1,1\rangle}) \otimes R^2(M_3,\VV_{\langle 1,1,1\rangle}) \to R^1(M_3,\VV_{\langle 0 \rangle}) \cong \Q$ is exactly given by multiplying the two cycles and pushing the result down to $M_3$. This proves the genus $3$ case of Theorem \ref{lowgenustheorem}.

\subsection{Genus four} By Corollary \ref{irrelevantgenerators} and Proposition \ref{fzconstraint} we now have three generators for the twisted commutative algebra $R_4''$: $\kappa_1$, the Gross--Schoen cycle, and the Faber--Pandharipande cycle. 

We claim that $\kappa_1^2$ and the product of $\kappa_1$ and the Gross--Schoen cycle are both nonzero in $R_4''$. Moreover, let us consider the product of two Gross--Schoen cycles. If we consider the subspace of the twisted commutative algebra $S$ spanned by all possible products of two classes $D_{3,0}$, then this decomposes as a representation 
of $\sym_6$ as $\sigma_{4,2} \oplus \sigma_6 = \Ind_{(\sym_3)^2 \rtimes \sym_2}^{\sym_6} \mathbf 1$. We claim that the representation $\sigma_6$ goes to zero in $R_4''$, but that the representation $\sigma_{4,2}$ survives to $R_4''(6)$.

Finally, we also claim that all other products of generators for the twisted commutative algebra $R_4''$ vanish. More precisely, we need to check the following relations:
\begin{enumerate}
	\item $\kappa_1^3 \equiv 0$
	\item $\kappa_1^2 \Delta_{123} \equiv 0$
	\item $\kappa_1 \Delta_{12} \psi_1 \equiv 0$
	\item $\sum_{\substack{S \sqcup T = \{1,\ldots,6\}\\ \vert S \vert = \vert T \vert = 3 } }\Delta_{S}\Delta_{T}\equiv 0$
	\item $\Delta_{123}\Delta_{456}\Delta_{789}\equiv 0$
	\item $\Delta_{123}\Delta_{45}\psi_4\equiv 0$
	\item $\Delta_{123}\Delta_{456}\kappa_1\equiv 0$
	\item $\Delta_{12}\psi_1\Delta_{34}\psi_3\equiv 0$
\end{enumerate}

(1) It's well known that $\kappa_1^3=0$.

(2) Modulo $\equiv$ the only nonzero monomials in $S_3^4$ are $\Delta_{123}\psi_1^2$, $\Delta_{123}\psi_1\kappa_1$, $\Delta_{123}\kappa_1^2$, and the $\sym_3$-conjugates of $\psi_1^2 \Delta_{23}\psi_2$. The relation $\FZ_{4,1,2}$ reduces to $\psi_1^2 \equiv 0$, which also implies $\psi_1^2 \Delta_{23}\psi_2 \equiv 0$. This leaves us with three potentially nonzero monomials. But $\FZ_{4,3,4}$, the pushforward along a diagonal of $\FZ_{4,2,3}$, and the pushforward of $\FZ_{4,1,2}$, gives us three linearly independent linear relations between these monomials modulo $\equiv$.


(3) The only nonzero monomials in $S_2^3$ modulo $\equiv$ are $\Delta_{12}\psi_1^2$ and $\Delta_{12}\kappa_1\psi_1$. The relations $\FZ_{4,2,3}$ and the pushforward of the relation $\FZ_{4,1,2}$ along the diagonal give two distinct linear relations between these monomials modulo $\equiv$, so both are zero modulo $\equiv$. 

(4) The relation $\psi_7 \cdot \FZ_{4,7,4}$, pushed down along the forgetful map that forgets the $7$th marked point, reduces to this expression modulo $\equiv$. Alternatively, since the expression is $\sym_6$-invariant, its image in $R_4''$ lands in the summand $R^\bullet(M_4,\VV_{\langle 1,1,1,1,1,1\rangle}) \otimes \sigma_6^\vee$. But the motive $\VV_{\langle 1,1,1,1,1,1\rangle}$ is zero. 

(5) Modulo $\equiv$, the only nonzero monomials in $S_9^6$ are the $\sym_9$-conjugates of $\Delta_{123}\Delta_{456}\Delta_{789}$. For any four indices $i,j,k,l$ we may consider the pushforward of $\FZ_{4,7,4}$ along the $i,j$-th and $k,l$-th diagonal, to get a relation which up to a scalar must equal
$$ \sum_{\substack{S \sqcup T \sqcup U = \{1,\ldots,9\}\\ \vert S \vert = \vert T \vert = \vert U \vert = 3 \\ i,j \in S \,\, k,l \in T}}  \Delta_S \Delta_T \Delta_U \equiv 0.$$
This gives $\frac 1 2 \binom 9 {2,2,5} = 378$ relations between $\frac 1 {3!} \binom 9 {3,3,3} = 280$ unknowns, and one can check using a computer that the resulting matrix has full rank; in particular, $\Delta_{123}\Delta_{456}\Delta_{789}\equiv 0$.

(6) Modulo $\equiv$, there are exactly $11$ nonzero monomials in $S_5^4$: $\Delta_{12345}$, and the $\sym_5$-conjugates of $\Delta_{123}\Delta_{45}\psi_4$. The relation $\FZ_{4,5,4}$, and the $10$ different pushforwards of the relation $\FZ_{4,4,3}$ along a diagonal inclusion, give $11$ linear relations between these monomials modulo $\equiv$. The resulting $11 \times 11$ matrix is invertible and we conclude that $\Delta_{123}\Delta_{45}\psi_4 \equiv 0$. 

(7) The nonzero monomials in $S_6^5$ are $\Delta_{123456}$, $\Delta_{123}\Delta_{456}\psi_1$, $\Delta_{123}\Delta_{456}\kappa_1$, $\Delta_{12}\Delta_{3456}\psi_1$, and their $\sym_6$-conjugates. This implies that the relation $\FZ_{4,5,4}$, pushed forward along the $i,j$-th diagonal and multiplied with $\kappa_1$, must be equal to
$$ \sum_{\substack{S \sqcup T = \{1,\ldots,6\} \\ \vert S \vert = \vert T \vert = 3 \\ i,j \in S}} \Delta_S \Delta_T \kappa_1\equiv 0$$
up to a nonzero scalar, since all the monomials involved in this relation must involve $\kappa_1$ and have $i,j$ in the same diagonal block. But this implies $\Delta_{123}\Delta_{456}\kappa_1\equiv 0$ by the same calculation as for relation (3) in genus $3$. 

(8) The nonzero monomials in $S_4^4$ are $\Delta_{12}\psi_1 \Delta_{34}\psi_3$, $\psi_1^2 \Delta_{234}$, $\Delta_{1234}\psi_1$, $\Delta_{1234}\kappa_1$ and their $\sym_4$-conjugates. Since $\psi_1^2 \equiv 0$, as observed in (2) above, we will also have $\psi_1^2\Delta_{234}\equiv 0$. Moreover, the relation $\FZ_{4,4,3}$ simplifies to $\Delta_{1234}\equiv 0$, so that also $\Delta_{1234}\kappa_1 \equiv 0$. This leaves only four potentially nonzero monomials. The relations given by $\psi_1 \cdot \FZ_{4,4,3}$, $\Delta_{12}\cdot \FZ_{4,4,3}$, $\Delta_{13}\cdot \FZ_{4,4,3}$ and $\Delta_{1,4}\cdot \FZ_{4,4,3}$ give four linearly independent relations between these monomials, and we conclude that they all vanish modulo $\equiv$.

We should also prove that all these cycles are nonzero and that Poincar\'e duality holds. We use that the relation $3\kappa_1^2  = -32\kappa_2$ holds in $R^2(M_4) = \Q\{\kappa_1^2\}$. The Gross--Schoen cycle squared, pushed down to $M_4$, equals $\frac 3 2 \kappa_1$. This shows both that the Gross--Schoen cycle is nonzero and that its product with $\kappa_1$ is nonzero. The Faber--Pandharipande cycle squared, pushed down to $M_4$, equals $\frac {19}{96}\kappa_1^2 -\frac 43 \kappa_2$, which is then also nonzero. The projection of $\Delta_{123}\Delta_{456}$ onto $R^4(M_4,\VV_{\langle 2,2,1,1\rangle}) \otimes \sigma_{4,2}$, squared and pushed down to $M_4$, equals $\frac{19877}{29160}\kappa_1^2  -\frac{25}{729}\kappa_2$, which is then also nonzero.  

This settles the genus $4$ case, and hence concludes the proof of Theorem \ref{lowgenustheorem}.

%
%
\subsection{Genus five}\label{genus5}
Let us briefly comment on the situation in genus $5$. The generators for the twisted commutative algebra $R_5''$ are $\kappa_1$, $\psi_1^2$, $\Delta_{12}\psi_1$, $\Delta_{123}$, and $\Delta_{1234}$. For $n \leq 7$ we find that the algebra $R^\bullet(C_5^n)$ is Gorenstein, and we can compute the groups $R^\bullet(M_5,\VV_\llambda)$ for $\vert \lambda \vert \leq 7$ by methods like those used in lower genera. One finds the following table of results, in which the classes in the right hand column project onto generators for the tautological groups listed in the left hand column. 
\begin{align*}
&R^0(M_5) \cong R^1(M_5) \cong R^2(M_5) \cong R^3(M_5) \cong \Q & \\
& R^2(M_5,\VV_{\langle 1\rangle}) \cong \Q  & \psi_1^2\\
&R^2(M_5,\VV_{\langle 1,1\rangle}) \cong R^3(M_5,\VV_{\langle 1,1\rangle}) \cong \Q & \Delta_{12}\psi_1, \Delta_{12}\psi_1 \kappa_1\\
&R^2(M_5,\VV_{\langle 1,1,1\rangle}) \cong R^3(M_5,\VV_{\langle 1,1,1\rangle}) \cong R^4(M_5,\VV_{\langle 1,1,1\rangle})\cong\Q& \Delta_{123}, \Delta_{123}\kappa_1, \Delta_{123}\kappa_1^2\\
&R^3(M_5,\VV_{\langle 1,1,1,1\rangle}) \cong R^4(M_5,\VV_{\langle 1,1,1,1\rangle}) \cong\Q & \Delta_{1234}, \Delta_{1234}\kappa_1\\
&R^4(M_5,\VV_{\langle 1,1,1,1,1\rangle}) \cong \Q & \Delta_{123}\Delta_{45}\psi_4 \\
&R^4(M_5,\VV_{\langle 2,1,1,1\rangle}) \cong \Q & \Delta_{123}\Delta_{45}\psi_4  \\
&R^4(M_5,\VV_{\langle 2,2,1\rangle}) \cong \Q  & \Delta_{123}\Delta_{45}\psi_4  \\
&R^4(M_5,\VV_{\langle 2,2,1,1\rangle}) \cong \Q & \Delta_{123}\Delta_{456}   \\
&R^5(M_5,\VV_{\langle 2,2,1,1\rangle}) \cong \Q & \Delta_{123}\Delta_{456}\kappa_1   \\
&R^5(M_5,\VV_{\langle 2,2,1,1,1\rangle}) \cong \Q & \Delta_{123}\Delta_{4567} \\
&R^5(M_5,\VV_{\langle 2,2,2,1\rangle}) \cong \Q & \Delta_{123}\Delta_{4567} \\
\end{align*}
For $n=8$ the Faber conjecture predicts the vanishing results $\Delta_{1234}\Delta_{5678} \equiv \Delta_{123}\Delta_{456}\Delta_{78}\psi_7 \equiv 0$. Assuming that the FZ relations are all relations between tautological classes, one finds that $\Delta_{1234}\Delta_{5678}$ and $ \Delta_{123}\Delta_{456}\Delta_{78}\psi_7$ should both have nonzero image in $R_5''(8)$, and we expect that
$$ R^6(M_5,\VV_{\langle 2,2,2,2 \rangle }) \cong R^6(M_5,\VV_{\langle 3,2,2,1\rangle}) \cong \Q.$$
Either of these nonvanishings would imply that $R^\bullet(C_5^8)$ is not Gorenstein. Proving them seems like a hard problem; nevertheless, we consider this to be progress in trying to find a counterexample to the Faber conjecture. Trying to prove that a specific cohomology group (or cohomology class) does not vanish is a far more appealing problem than, say, trying to prove that the rank of $R^6(C_5^8)$ is greater than $35166$. Moreover, our approach relates the Faber conjecture to actively studied questions about \emph{modified diagonals}, see e.g.\ \cite{ogradydiagonals,voisinmodifieddiagonals,moonenyin}.

%
%

\section{Relation to work of Looijenga}\label{looijenga-section}

\subsection{The theorems of Harer and Madsen--Weiss}
Let $M_{g,\vec 1}$ denote the moduli space parametrizing smooth genus $g$ curves equipped with a marked point and a nonzero tangent vector at the marking. The (analytifications of the) spaces $M_g$ and $M_{g,\vec 1}$ are both $K(\pi,1)$ spaces in the orbifold sense, meaning in particular that their cohomology is given by the group cohomology of their fundamental groups. Whereas the (orbifold) fundamental group of $M_g$ is the mapping class group of a closed genus $g$ surface, the fundamental group of $M_{g,\vec 1}$ is the mapping class group of a genus $g$ surface with a parametrized boundary component. As such, there is a ``stabilization'' map
$$ H_k(M_{g,\vec 1},\Z) \to H_{k}(M_{g+1,\vec 1},\Z)$$
which on the level of fundamental groups is given by gluing a torus with two boundary components onto the boundary of the genus $g$ surface. See also \cite[Section 4]{hainlooijenga} for how to define these stabilization maps algebro-geometrically.

The celebrated stability theorem of Harer \cite{harerstability} asserts that the stabilization map is an isomorphism for $g \gg k$.  

\begin{thm}[Harer] $H_k(M_{g,\vec 1},\Z) \to H_{k}(M_{g+1,\vec 1},\Z)$ is an isomorphism for $k \leq \frac 2 3 (g - 1)$. \end{thm}

If we are interested primarily in closed surfaces, we also have a stabilization result for the forgetful map $M_{g,\vec 1} \to M_g$ that forgets the marking and the tangent vector:

\begin{thm}[Harer] \label{harer2} The map $H_k(M_{g,\vec 1},\Z) \to H_{k}(M_g,\Z)$ is an isomorphism for $k \leq \frac 2 3 g$. \end{thm}

The original bounds for the stable range of Harer have successively been improved by multiple people to obtain these results, see \cite{boldsen,wahlstability}. We note that to obtain stability with integer coefficients in Theorem \ref{harer2} it is crucial that $M_g$ is considered as a stack --- if we work with its coarse moduli space, the result is only valid with $\Q$-coefficients. 

It is a formidable problem to actually compute the stable homology of $M_g$. With $\Q$-coefficients, an answer was conjectured by Mumford and proven by Madsen--Weiss \cite{madsenweiss}:

\begin{thm}[Madsen--Weiss]
	The map $\Q[\kappa_1,\kappa_2,\kappa_3,\ldots] \to H^\bullet(M_g,\Q)$ is an isomorphism in the stable range, i.e. in degrees up to $\frac 2 3 (g-1)$. 
\end{thm}

\begin{rem}If we formally denote the value of the stable cohomology by $H^\bullet(M_\infty,\Q)$, then the statement is that $H^\bullet(M_\infty,\Q)$ is a polynomial algebra in the $\kappa$ classes. Since the tautological ring of $M_g$ is defined as the algebra generated by the $\kappa$ classes, it therefore makes sense say that \emph{the tautological cohomology of $M_g$ is the image of the stable cohomology in the unstable cohomology.}
\end{rem}

\subsection{Twisted coefficients}
One can also ask whether homological stability holds with coefficients in a local system $\V_\llambda$. In this case, stabilization should be interpreted as appending an integer $\lambda_{g+1} = 0$ to the weight vector $\lambda_1 \geq \ldots \geq \lambda_g \geq 0$. The analogue of Harer stability holds in this case, too, by a theorem of Ivanov \cite{ivanovtwisted}:

\begin{thm}[Ivanov]
	The map $H^k(M_g,\V_\llambda) \to H^k(M_{g+1},\V_\llambda)$ is an isomorphism for $g \gg k, \vert \lambda\vert$. 
\end{thm}

We should remark that Ivanov's statement was not specifically about the local systems $\V_\llambda$; his theorem is valid for a more general notion of coefficient system of finite degree which makes sense over an arbitrary base ring, and the local systems $\V_\llambda$ are an example of such.

 Ivanov's theorem did not actually calculate the stable cohomology with twisted coefficients. Rationally, the stable cohomology with coefficients in $\V_\llambda$ was calculated by Looijenga \cite{looijengastable}, in a paper that strongly influenced our way of thinking on these subjects. 

The first step in Looijenga's calculation of $H^\bullet(M_\infty,\V_\llambda)$ is to compute the stable cohomology of the spaces $C_g^n$. His result can be reformulated in a rather appealing way in terms of twisted commutative algebras:

\begin{thm}[Looijenga] \label{looijenga1} The twisted commutative algebra $\bigoplus_{n \geq 0}H^\bullet(C_\infty^n,\Q)$ is free on $\sym_n$-invariant generators $D_{n,r}$ in arity $n$ and cohomological degree $2r$ for $n =0$, $r \geq 1$ and $n\geq 1,r \geq n-1$. In other words, the map $S_{n}^\bullet \to H^{2\bullet}(C_g^n)$ is an isomorphism in the stable range. 
\end{thm}

We note that this theorem contains in particular the Madsen--Weiss theorem, by restricting to the case $n=0$ (in which case the generators $D_{0,r}$ are kappa classes), even though Looijenga's paper predates the Madsen--Weiss theorem. Thus Looijenga's theorem was rather that the stable cohomology of $M_g$ with twisted coefficients is a free module over the stable cohomology with constant coefficients with explicitly given generators; plugging in the Madsen--Weiss theorem gives the above result.

To compute $H^\bullet(M_\infty,\V_\llambda)$ from Theorem \ref{looijenga1}, one notes that there is a surjection of twisted commutative algebras $\bigoplus_{n \geq 0} H^\bullet(C_\infty^n,\Q) \to \bigoplus_{n \geq 0} \bigoplus_{\vert \lambda \vert = n} H^{\bullet-n}(M_g,\V_\llambda) \otimes \sigma_{\lambda^T}$, whose kernel is  the ideal generated by the classes $D_{1,0}$, $D_{1,1}$ and $D_{2,1}$. Thus one finds:

\begin{thm}[Looijenga] The twisted commutative algebra $\bigoplus_{\lambda}H^{\bullet-\vert\lambda\vert}(M_\infty,\V_\llambda)\otimes \sigma_{\lambda^T}$ is the free twisted commutative algebra on $\sym_n$-invariant generators $D_{n,r}$ in arity $n$ and cohomological degree $2r$ for $n =0$, $r \geq 1$ and $n\geq 1,r \geq \max(2,n-1)$.
\end{thm}

By decomposing this free twisted commutative algebra into irreducible representations of $\sym_n$, one finds a calculation of the stable cohomology $H^\bullet(M_\infty,\V_\llambda)$ for any $\lambda$. Looijenga does not state his result in these terms: he defines a certain algebra $B_n^\bullet$ which he decomposes into irreducible representations of $\sym_n$, and this algebra (tensored with the polynomial ring in the kappa classes) is the arity $n$ component of the free twisted commutative algebra in the previous result.

The conclusion is in any case the following. For constant coefficents, the stable cohomology of $M_g$ is a \emph{free polynomial algebra} on the $\kappa$-classes. The image of the stable cohomology inside the unstable cohomology can be defined to be the \emph{tautological cohomology of $M_g$}. If we consider instead the stable cohomology with all possible twisted coefficients, i.e.\ the direct sum $\bigoplus_\lambda H^{\bullet-\vert \lambda \vert}(M_g,\V_\llambda) \otimes \sigma_{\lambda^T}$, then this is a \emph{free twisted commutative algebra}, and the image of the stable cohomology inside the unstable cohomology is now exactly what we defined to be the \emph{tautological cohomology of $M_g$ with twisted coefficients}.

\section{The ``primary approximation'' to the cohomology of the moduli space}\label{hain-section}

Prior to this paper, Hain \cite{hainnormalfunctions} proposed a definition of tautological cohomology groups $\RH^\bullet(M_g,\V_\llambda)$ of $M_g$ with coefficients in a symplectic representation, which is a priori different from ours. In this section we will show that the two definitions coincide. In case $\lambda = 0$, this gives a new proof of a theorem of Kawazumi and Morita \cite{kawazumimorita}. We note that Hain asked in loc.\ cit.\ whether his construction could be lifted to the level of Chow groups; our constructions provide such a lifting. 

Let $\mathcal O(\Sp({2g}))$ be the algebraic coordinate ring of the symplectic group over $\Q$. By the Peter--Weyl theorem, there is an isomorphism of $\Sp(2g) \times \Sp(2g)$-bimodules
$$ \mathcal O(\Sp(2g)) \cong \bigoplus_\lambda V_\llambda \otimes V_\llambda^\vee, $$
where the sum runs over all irreducible representations of the symplectic group. We consider $V_\llambda$ to have a left action and $V_\llambda^\vee$ with a right action. Using the left action of $\Sp(2g)$ on $\mathcal O(\Sp(2g))$, we may consider it as defining a local system of algebras  $\O(\Sp(2g))$ on $M_g$. Taking its cohomology, we get that
$$ \mathsf T_g \stackrel{\mathrm{def}}= H^\bullet(M_g,\O(\Sp(2g))) = \bigoplus_{\lambda} H^\bullet(M_g,\V_\llambda) \otimes V_\llambda^\vee$$
is in a natural way a commutative $\Q$-algebra. See \cite[Section 9.5]{hainlooijenga} or \cite{hainnormalfunctions}.

\begin{rem}A perhaps more down to earth way to understand this multiplication is as follows. Suppose that we have $V_\llambda \otimes V_{\langle \mu \rangle} \supset V_{\langle \nu \rangle}$. Then we get a multiplication map 
	$$ H^\bullet(M_g,\V_\llambda) \otimes H^\bullet(M_g,\V_{\langle \mu \rangle}) \to H^\bullet(M_g,\V_{\langle \nu \rangle})$$
	which however depends nontrivially on the choice of an intertwiner $V_\llambda \otimes V_{\langle \mu \rangle} \to V_{\langle \nu \rangle}$. What is instead completely well defined is the map
	$$  \Hom_{\Sp(2g)}(V_\llambda \otimes V_{\langle\mu\rangle},V_{\langle\nu\rangle}) \otimes H^\bullet(M_g,\V_\llambda) \otimes H^\bullet(M_g,\V_{\langle \mu \rangle}) \to H^\bullet(M_g,\V_{\langle \nu \rangle}),$$
	which (using the canonical identification $\Hom(M,N) = N\otimes M^\vee$) can be thought of equivalently as an $\Sp(2g)$-equivariant map
	\[ H^\bullet(M_g,\V_\llambda) \otimes V_\llambda^\vee \otimes H^\bullet(M_g,\V_{\langle \mu \rangle}) \otimes V_{\langle \mu \rangle}^\vee \to H^\bullet(M_g,\V_{\langle \nu \rangle}) \otimes V_{\langle \nu \rangle}^\vee.\qedhere\]\end{rem}



Let us now consider the Gross--Schoen cycle as a class $ \alpha \in H^1(M_g,\V_{\langle 1,1,1\rangle})$. We have a vector subspace $\alpha \otimes V_{\langle 1,1,1\rangle}^\vee \subset \mathsf T_g$, and therefore by the universal property of a polynomial algebra a morphism of graded commutative rings
$$ \wedge^\bullet V_{\langle 1,1,1\rangle}^\vee \to \mathsf T_g.$$
There is an inclusion $V_{\langle 2,2\rangle}^\vee \subset \wedge^2 V_{\langle 1,1,1\rangle}^\vee$.  
Since $\alpha$ is tautological, every class in the image of this homomorphism is tautological; it follows from this that the summand $V_{\langle 2,2\rangle}^\vee \subset \wedge^2 V_{\langle 1,1,1\rangle}^\vee$ lies in the kernel, since one can compute from our results (or rather the work of Looijenga) that $\RH^2(M_g,\V_{\langle 2,2\rangle})=0$. We denote the algebra $\wedge^\bullet V_{\langle 1,1,1\rangle}^\ast/(V_{\langle 2,2\rangle}^\vee) $ by $\mathsf A_g$. It follows that there exist an $\Sp(2g)$-equivariant ring homomorphism
$$ \varphi \colon \mathsf A_g \to \mathsf T_g.$$

\begin{defn}
	The \emph{H-tautological ring} is the subring $\mathsf R_g \subset \mathsf T_g$ given as the image of $\varphi$. By decomposing the $H$-tautological ring into irreducible summands for its natural action of $\Sp(2g)$ we get a subspace of $H$-tautological classes inside $H^\bullet(M_g,\V_\llambda)$ for any partition $\lambda$.
\end{defn}

\begin{rem}When $g=2$ the local system $\V_{\langle 1,1,1\rangle}$ vanishes (and so does the Gross--Schoen cycle), and the $H$-tautological ring consists only of the unit element in $H^0(M_2,\V_{\langle 0 \rangle})$.
\end{rem}

\begin{rem}The definition above may seem very ad hoc --- why should the Gross--Schoen cycle play a more distinguished role than any other tautological class? A more ``invariant'' definition is that the $H$-tautological ring is the subring of $H^\bullet(M_g,\O(\Sp(2g)))$ generated by all normal functions over $M_g$ \cite{haintorelligroupsandgeometry}. 
\end{rem}

\begin{rem}It is a striking fact that unlike the usual tautological ring of $M_g$ or $C_g^n$, the $H$-tautological ring is generated by a single algebraic cycle class. 
\end{rem}

Restricting to symplectic invariants, we get a map
$$ \varphi^{\Sp(2g)} \colon \mathsf A_g^{\Sp(2g)} \to \mathsf T_g^{\Sp(2g)} = H^\bullet(M_g,\Q).$$
This morphism is exactly what Morita calls the \emph{primary approximation} to the cohomology ring of $M_g$. Morita originally described it in rather different terms \cite{moritaprimaryapproximation}; this re-interpretation is due to Hain. A theorem of Kawazumi and Morita \cite{kawazumimorita} asserts that the image of $\varphi^{\Sp(2g)}$ is the tautological cohomology ring of $M_g$. We will prove a more general result below.
 
\begin{rem} The above map can also be understood in terms of relative Malcev completion \cite{haininfinitesimal}. Hain constructs a Lie algebra $\u_g$ of mixed Hodge structures with $\Sp(2g)$-action (the Lie algebra of the pro-unipotent radical of the relative completion of the mapping class group) and an $\Sp(2g)$-equivariant map $H^\bullet(\u_g) \to H^\bullet(M_g,\O(\Sp(2g)))$. The results of \cite{haininfinitesimal} (and subsequent improvements) show that the weights of $\u_g$ are negative and that $\mathrm{Gr}^W_{-1}\u_g \cong V_{\langle 1,1,1\rangle}^\vee$, $\mathrm{Gr}^W_{-2}\u_g \cong V_{\langle 2,2\rangle}^\vee$. It follows that the algebra $\mathsf A_g$ is the pure part  $\bigoplus_k W_kH^k(\u_g)$ of the Chevalley--Eilenberg cohomology of $\u_g$, so the $H$-tautological ring can also be defined as the image of the lowest weight part of the cohomology of $\u_g$.
\end{rem}

\begin{lem}\label{h-tautological-lemma}Let $n,m$ be integers with $n \geq 0$, $m \geq 0$ and $n+2m-2 > 0$. Construct a $(3\cdot (n+2m-2),n)$-Brauer diagram as follows: for $i=1,\ldots,n+2m-3$, draw a horizontal strand connecting the $(3i)$th node on the top row to the $(3i+1)$st. Of the remaining $n+2m$ nodes on the top row, pick $n$ of them arbitrarily and connect them to the nodes along the bottom row, and connect the remaining $2m$ nodes arbitrarily to each other by $m$ horizontal strands. Consider the resulting map
	$$ H^{n+2m-2}(M_g,\V^{\otimes 3(n+2m-2) }) \to H^{n+2m-2}(M_g,\V^{\otimes n}).$$
	The image of $\pi_1^{\times 3}(\Delta_{123})^{\times (n+2m-2)}$ under this map is the class $\pi_1^{\times n}(\Delta_{12\ldots n} \psi_1^m)$ if $n > 0$, and $\kappa_{m-1}$ if $n = 0$.
\end{lem}

	
\begin{proof}
	This is an easy consequence of the discussion in Section \ref{explicitbraueraction}. Namely, to compute the image of $\pi_1^{\times 3}(\Delta_{123})^{\times (n-2+2m)}$ we start with the cycle $\Delta_{123}\Delta_{456}\Delta_{789}\ldots$, restrict to a suitable diagonal locus --- specifically, the diagonals are specified by the horizontal strands in the Brauer diagram --- and then project away from the markings corresponding to these diagonals. This gives $\Delta_{12\ldots n} \psi_1^m$ if $n > 0$, and $\kappa_{m-1}$ if $n = 0$,  after which we should apply $\pi_1^{\times n}$, which gives the result.\end{proof}

\begin{thm}\label{h-tautological}
	The space of $H$-tautological classes inside $H^\bullet(M_g,\O(\Sp(2g)))$ coincides with the space  $\RH^\bullet(M_g,\O(\Sp(2g))) = \bigoplus_\lambda \RH^\bullet(M_g,\V_\llambda) \otimes V_\lambda^\vee$  of tautological classes in our sense. 
\end{thm}

\begin{proof}We note first that $\RH^\bullet(M_g,\O(\Sp(2g)))$ is a subalgebra of $H^\bullet(M_g,\O(\Sp(2g))$. Indeed, consider the multiplication map
	$$  \Hom_{\Sp(2g)}(V_\llambda \otimes V_{\langle\mu\rangle},V_{\langle\nu\rangle}) \otimes H^\bullet(M_g,\V_\llambda) \otimes H^\bullet(M_g,\V_{\langle \mu \rangle}) \to H^\bullet(M_g,\V_{\langle \nu \rangle}).$$
	Every element of $\Hom_{\Sp(2g)}(V_\llambda \otimes V_{\langle\mu\rangle},V_{\langle\nu\rangle})$ is given by Brauer diagrams. It follows that if we realize the cohomologies of the different local systems as summands of the cohomologies of fibered powers $C_g^n$, then the multiplication $H^\bullet(M_g,\V_\llambda) \otimes H^\bullet(M_g,\V_{\langle \mu \rangle}) \to H^\bullet(M_g,\V_{\langle \nu \rangle})$ is induced by an algebraic correspondence given by tautological cycles,, for any choice of intertwiner $V_\llambda \otimes V_{\langle\mu\rangle} \to V_{\langle\nu\rangle}$. This means in particular that the cup product maps tautological classes to tautological classes.
	
	In particular, this means that every $H$-tautological class is a tautological class in our sense: since the $H$-tautological ring is generated by the Gross--Schoen cycle, it must be contained in the subalgebra of all tautological classes. 
	
	Conversely we need to prove that every tautological class in ours sense can be written as a product of Gross--Schoen cycles. It is enough to prove this for the generators of the twisted commutative algebra $R_g''$, i.e. the images of the classes $\Delta_{12\ldots n}\psi_1^m$ and the $\kappa$-classes. If we consider the Brauer diagram  of Lemma \ref{h-tautological-lemma} as an element of $\Hom_{\Sp(2g)}((\V_{\langle 1,1,1\rangle})^{\otimes (n-2+2m)},\V^{\langle n\rangle}),$ then the image of this Brauer diagram and $n-2+2m$ copies of the Gross--Schoen cycle under the cup product map
	$$ \Hom_{\Sp(2g)}((\V_{\langle 1,1,1\rangle})^{\otimes (n-2+2m)},\V^{\langle n\rangle}) \otimes H^1(M_g,(\V_{\langle 1,1,1\rangle})^{\otimes (n-2+2m)} \to H^{n-2+2m}(M_g,\V^{\langle n\rangle})$$
	equals the image of $\Delta_{12\ldots n}\psi_1^m$ under the projection $H^{2(n-1+m)}(C_g^n,\Q) \to H^{n-2+2m}(M_g,\V^{\langle n\rangle})$, as Lemma \ref{h-tautological-lemma} shows. The result follows.
\end{proof}

\begin{cor}[Kawazumi--Morita]
	The image of $(\wedge^\bullet V_{\langle 1,1,1\rangle}/V_{\langle 2,2\rangle})^{\Sp(2g)} \to H^\bullet(M_g,\Q)$ is the tautological cohomology ring of $M_g$.
\end{cor}

\begin{proof}This is the case $\lambda = 0$ of the preceding theorem.
\end{proof}

Morita \cite{moritacohomologicalstructure} has conjectured that the map $\varphi^{\Sp(2g)}$ is injective: that is, it defines an isomorphism between $\mathsf A_g^{\Sp(2g)}$ and the tautological ring $R^\bullet(M_g)$. Compared to other conjectural descriptions of the tautological ring, e.g.\ the conjecture that all relations are consequences of the FZ relations, this gives a remarkably quick and direct description of the tautological ring (even though it is not so easy to determine the structure of the algebra $\mathsf A_g^{\Sp(2g)}$). A natural generalization of Morita's conjecture is to ask whether $\varphi \colon \mathsf A_g \to \mathsf R_g$ is an isomorphism. A consequence of our results is that this is not the case, however:

\begin{prop}The map $\varphi$ is not injective when $g=4$. 
\end{prop}

\begin{proof}Using a computer algebra system, one can verify that the third exterior power $\wedge^3 V_{\langle 1,1,1\rangle}^\vee$ contains the irreducible representation $V_{\langle 3,2,2,2\rangle}^\vee$ as a summand. On the other hand the degree $3$ part of the ideal $(V_{\langle 2,2\rangle}^\vee)$ is a quotient of $V_{\langle 2,2\rangle}^\vee \otimes V_{\langle 1,1,1\rangle}^\vee$, which contains only representations of weight at most $7$. It follows that $\mathsf A_4$ has a nontrivial summand $V_{\langle 3,2,2,2\rangle}^\vee$ in degree $3$. But our calculations of the tautological groups with twisted coefficients in genus four show that $\RH^3(M_4,\V_{\langle 3,2,2,2\rangle})=0$, so this summand must be in the kernel of $\varphi$.
\end{proof}

As pointed out in the introduction of this section, Hain asked whether the construction of a tautological algebra inside $H^\bullet(M_g,\O(\Sp(2g)))$ could be lifted to the level of Chow groups, and our construction in this paper gives an affirmative answer to this question. However, there does not seem to be any sensible way to get the grading on the level of Chow groups, at least not without introducing fractional Tate twists. The source of the problem is that an intertwiner in $\Hom_{\Sp(2g)}(V_\llambda \otimes V_{\langle \mu \rangle},V_{\langle \nu \rangle})$ does not give rise to a morphism of Chow motives $\VV_\llambda \otimes \VV_{\langle \mu \rangle} \to \VV_{\langle \nu \rangle}$ unless $\vert \lambda \vert + \vert \mu \vert = \vert \nu \vert$; in general one only gets a morphism to a Tate twist of $\VV_{\langle \nu \rangle}$. One option is to work instead with Chow motives with respect to ungraded correspondences --- one can make sense of $\O(\Sp(2g))$ as a relative Chow motive over $M_g$ with respect to ungraded correspondences --- but the Chow groups of a motive with respect to ungraded correspondences only form a vector space, not a graded vector space, and so the grading needs to be put in ``by hand''. Alternatively, if we allow half-integer Tate twists, then we can replace $\VV_\llambda$ with $\VV_\llambda \otimes \L^{-\vert \lambda\vert/2}$ throughout, which will allow us to recover the cohomological grading (halved).

\subsection{The twisted commutative algebra and the Peter--Weyl theorem}We have now seen two a priori completely different ways to build an algebra by considering all local systems $\V_\llambda$ simultaneously: the commutative ring $\mathsf T_g = \bigoplus_\lambda H^\bullet(M_g,\V_\llambda) \otimes V_\llambda^\vee$ and the twisted commutative algebra $\bigoplus_\lambda H^\bullet(M_g,\V_\llambda) \otimes \sigma_\lambda^\vee$ (see Remark \ref{cohomologytca}). We now want to explain a connection 
between the two constructions.

Suppose that $A$ is a ring in the category of graded $\Sp(2g)$-representations. We associate to $A$ two twisted commutative algebras $L_A$ and $L_A^\circ$ given by 
$$ n \mapsto L_A(n) = (A \otimes V^{\otimes n})^{\Sp(2g)}$$
and
$$ n \mapsto L_A^\circ(n) = (A \otimes V^{\langle n \rangle})^{\Sp(2g)}.$$
(Recall that $V^{\langle n \rangle} = \bigoplus_{\lambda \vdash n} V_\llambda \otimes \sigma_\lambda^\vee$.) The multiplication in $L_A$ is given by
\begin{align*} L_A(n) \otimes L_A(m) & = (A \otimes V^{\otimes n})^{\Sp(2g)} \otimes (A \otimes V^{\otimes m})^{\Sp(2g)} \subseteq (A \otimes V^{\otimes n} \otimes A \otimes V^{\otimes m})^{\Sp(2g)} \\
& \stackrel{\mathrm{mult.}}\longrightarrow (A \otimes V^{\otimes n} \otimes V^{\otimes m})^{\Sp(2g)} = L_A(n+m),
\end{align*}
and similarly for $L_A^\circ$. 
We obtain two functors from the category of rings with action of $\Sp(2g)$ to the category of twisted commutative algebras. 

We will apply our functors to the rings $\mathsf A_g$, $\mathsf R_g$ and $\mathsf T_g$ defined above.\footnote{Strictly speaking these are rings in the category of representations of $\Sp(2g)^{\mathrm{op}}$, rather than $\Sp(2g)$: if we want to work with usual representations we should have $\wedge^\bullet V_{\langle 1,1,1\rangle}/(V_{\langle 2,2\rangle})$ rather than $\wedge^\bullet V_{\langle 1,1,1\rangle}^\vee/(V_{\langle 2,2\rangle}^\vee)$. We will ignore this detail.} We find for example that
\begin{align*} 
L_{\mathsf T_g}^\circ(n) & = \mathrm{Hom}_{\Sp(2g)}(H^\bullet(M_g,\O(\Sp(2g)),V^{\langle n\rangle}) \\
& \cong \bigoplus_\lambda \bigoplus_{\mu \vdash n} H^\bullet(M_g,\V_\lambda) \otimes \mathrm{Hom}_{\Sp(2g)}(V_\lambda,V_\mu) \otimes \sigma_\mu^\vee \\
& \cong\bigoplus_{\lambda \vdash n}  H^\bullet(M_g,\V_\llambda) \otimes \sigma_\lambda^\vee = H^\bullet(M_g,\V^{\langle n \rangle}).  
\end{align*}
In particular, $L_{\mathsf R_g}$ and $L_{\mathsf R_g}^\circ$ are the cohomological analogues of the twisted commutative algebras denoted $R_g'$ and $R_g''$ in Section \ref{tca-section}. So our twisted commutative algebras can be recovered functorially from Hain's tautological ring $\mathsf R_g$, which explains how the two constructions are related. 

A  slightly more refined version of the above construction is possible. Recall that a twisted commutative algebra can be defined as a lax symmetric monoidal functor from $\coprod_{n \geq 0} \sym_n$ to graded vector spaces. We define instead a \emph{twisted Brauer algebra} as a lax symmetric monoidal functor from $\mathfrak{Br}^{(-2g)}$ to graded vector spaces. Recall that the category $\mathfrak{Br}^{(-2g)}$ was defined in Section \ref{nazarov1}; it is the category whose objects are the natural numbers and morphisms $n \to m$ are $(n,m)$-Brauer diagrams, with symmetric monoidal structure given on objects by addition and on morphisms by disjoint union. There is an evident inclusion $\coprod_{n \geq 0} \sym_n \to \mathfrak{Br}^{(-2g)}$, by interpreting a bijection $[n] \to [n]$ as an $(n,n)$-Brauer diagram with only vertical strands, which defines a forgetful functor from twisted Brauer algebras to twisted commutative algebras. 

It is not hard to see that if $A$ is a ring with $\Sp(2g)$-action, then $L_A$ may in fact be considered as a twisted Brauer algebra.\footnote{One may also consider $L_A^\circ$ as a twisted Brauer algebra, but one does not gain anything in doing so: all maps $L_A^\circ(n) \to L_A^\circ(m)$ for $n \neq m$ given by $(n,m)$-Brauer diagrams are zero.} This has some advantages. For example, we noted that the ring $\mathsf R_g$ is generated by a single algebraic cycle class (more precisely, by a single copy of the representation $V_{\langle 1,1,1\rangle}^\vee$), whereas the twisted commutative tautological algebras $R_g$, $R_g'$ and $R_g''$ had large numbers of generators. If we consider $L_{\mathsf R_g}$ as a twisted \emph{Brauer} algebra rather than a twisted commutative algebra, it is in fact generated by a single element in arity $3$ corresponding to the Gross--Schoen cycle. This shows that by considering twisted Brauer algebras one retains slightly more information. 
%
%
%
%
%


\bibliographystyle{alpha}
\bibliography{../database}

\end{document}